\documentclass[11pt]{article}
\usepackage[T1]{fontenc}
\usepackage[utf8]{inputenc}
\usepackage{lmodern} 
\usepackage{amsmath,amsfonts,amssymb,amsthm}
\usepackage{graphicx}
\usepackage{float}
\usepackage{amsrefs}

\usepackage[a4paper,top=2.5cm,bottom=2.5cm,left=2.5cm,right=2.5cm]{geometry}
\usepackage[colorlinks=true, allcolors=blue]{hyperref}
\usepackage{caption}
\usepackage[justification=centering]{subcaption}

\usepackage{booktabs}
\usepackage[table]{xcolor}

\usepackage{xspace}

\usepackage{tikz}

\usepackage{listings}
\usepackage{threeparttable}
\usepackage{multirow}
\usepackage{siunitx}

\usepackage{authblk}

\DeclareMathOperator{\thepace}{pace}

\newcommand{\LBpace}[1][]{\thepace^{LB}_{#1}}
\DeclareMathOperator{\npace}{npace}

\newcommand{\NLBpace}[1][]{\npace^{LB}_{#1}}
\definecolor{darkviolet}{rgb}{0.58, 0.0, 0.83}

\newcommand{\ndelta}[2]{{#1}^{#2}}
\newcommand{\x}{\boldsymbol{x}}

\newcommand{\RLTlooseB}{\texttt{RLT-looseB}\xspace}
\newcommand{\RLTtightB}{\texttt{RLT-tightB}\xspace}
\providecommand{\card}[1]{\lvert#1\rvert}
\providecommand{\co}[1]{\bar{#1}}
\newcommand{\bfc}[1]{U_{#1}L_{\co{T}}}
\newcommand{\bfcfull}[2]{U_{#1}L_{#2}}
\newcommand{\solver}[1]{\texttt{#1}}

\usepackage{mathtools}

\usepackage[skins]{tcolorbox}

\usepackage{enumitem}
\setlist[itemize,3]{label=$\diamond$}

\newenvironment{shortitem}[1][]{\begin{itemize}[topsep=0pt, itemsep=0pt, parsep=0.5pt, leftmargin=*, #1]}{\end{itemize}}

\newcommand{\R}{\mathbb R}
\newcommand{\N}{\mathbb N}

\newcommand{\bOmega}{\bar{\Omega}}
\newcommand{\bu}{\bar{u}}
\newcommand{\bl}{\bar{l}}

\newcommand{\itercounter}{\tau}

\newcommand{\cola}{Q}

\newcommand{\xb}{\bar{\boldsymbol{x}}}
\newcommand{\Xb}{\bar{\boldsymbol{X}}}

\newcommand{\zbar}{\bar z}

\renewcommand{\stop}{\textsc{End}}

\newtheorem{theor}{Theorem}
\newtheorem{cor}{Corollary}
\newtheorem{lem}{Lemma}

\usepackage{todonotes}

\title{Bound tightening in lifted formulations: (sub)solver-dependent impact on performance in RLT-based algorithms}
\author[1,2]{Julio González-Díaz\thanks{Corresponding author: julio.gonzalez@usc.es}}
\author[2,3]{Brais González-Rodríguez}
\author[1]{Ignacio Gómez-Casares}
\affil[1]{Department of Statistics, Mathematical Analysis and Optimization and MODESTYA Research Group, University of Santiago de Compostela, 15782 Santiago de Compostela, Spain}
\affil[2]{CITMAga (Galician Center for Mathematical Research and Technology), 15782 Santiago de Compostela, Spain}
\affil[3]{Department of Statistics and Operational Research and SiDOR Research Group, University of Vigo.}

\date{\today}

\begin{document}

\maketitle

\begin{abstract}
In this paper we explore a relevant aspect of the interplay between two core elements of global optimization algorithms for nonconvex nonlinear programming problems. The first one is the reformulation of the original problem, which requires the introduction of auxiliary variables with the goal of defining convex relaxations that can be solved both reliably and efficiently on a node-by-node basis. The second one, bound tightening or, more generally, domain reduction, allows to reduce the search space to be explored by the branch-and-bound algorithm. We are interested in the performance implications of propagating the bounds of the original variables to the auxiliary ones in the lifted space: does this propagation reduce the overall size of the tree? does it improve the efficiency of solving the node relaxations? To better understand the above interplay, we focus on the reformulation-linearization technique for polynomial optimization. In this setting we present a theoretical result on the implicit bounds of the auxiliary variables in the RLT relaxations, which sets the stage for the ensuing computational study, whose goal is to assess to what extent the performance of an RLT-based algorithm may be affected by the decision to explicitly propagate the bounds on the original variables to the auxiliary ones.
\end{abstract}



\section{Introduction}

In recent years there has been a rise in the research in global optimization and, in particular, in the development of global optimization solvers for problems that are both nonlinear and nonconvex, for which local optimality conditions do not guarantee good quality solutions. State-of-the-art commercial solvers for MILP problems such as \solver{Gurobi} \cite{gurobi} and \solver{Xpress} \cite{xpress} can now solve MINLP problems to certified global optimality. \solver{Octeract} \cite{octeract} is another recent global solver, which has joined the pool of well established ones such as \solver{ANTIGONE} \cite{antigone}, \solver{BARON} \cite{baron2018}, \solver{Couenne} \cite{Belotti2009}, \solver{LindoGlobal} \cite{lindo}, and \solver{SCIP} \cite{scip}. 

The single most important element of the branch-and-bound algorithms behind the above solvers is probably the approach taken to define the linear and/or convex relaxations of the underlying nonconvex problem. These relaxations are then efficiently solved at the nodes of the tree, along which finer and finer partitions of the feasible region are defined. Devising relaxation strategies that are applicable to general MINLP problems is a challenging task, which has led to a rich literature with a wide variety of mathematical developments (refer to \cite{tawarmalani2013} for a comprehensive coverage). In a series of seminal papers (see, for instance, \cite{Ryoo:1996}, \cite{baron1996}, \cite{tawarmalani2004}, \cite{tawarmalani2005}), the research team behind \solver{BARON} demonstrated the potential of factorable programming \cite{mccormick1976} to produce linear relaxations whose solution is both computationally efficient and numerically robust.\footnote{Recently, in \cite{baron2018}, the same research team has developed, and integrated in \solver{BARON}, a framework that combines polyhedral and convex nonlinear relaxations. Importantly, they devised fail-safe techniques to control for numerical issues in the solution of nonlinear relaxations and revert to the more robust polyhedral based relaxations when needed.} Since its adoption in \cite{Ryoo:1996} and \cite{baron1996}, factorable programming reformulations and factorable programming relaxations (usually polyhedral) have been at the core of essentially any general MINLP global solver.

The growing interest in global optimization is also present in the development of specialized algorithms for particular classes on nonlinear problems. One of the most important such classes, because of the richness of its applications, is polynomial programming, where \solver{RLT-POS} \cite{Dalkiran2016} and, more recently, \solver{RAPOSa} \cite{Gonzalez-Rodriguez:2023} are two specialized solvers that ensure global optimality. These solvers build upon the reformulation-linearization technique, RLT, introduced in \cite{Sherali1992} for polynomial optimization problems. In this technique, an auxiliary variable is associated to each monomial present in the original problem, which naturally leads to linear relaxations that are then embedded into the branch-and-bound scheme.

Importantly, both of the aforementioned approaches, factorable programming and RLT, require introducing a large number of auxiliary variables to construct suitable reformulations of the original problem and, hence, the ensuing relaxations are defined in lifted spaces.  Since these lifted spaces are typically of a much larger dimension than the original problem, also the search space for the branch-and-bound algorithm may become much larger. This observation motivates another fundamental element behind the efficient design of global optimization algorithms: domain reduction and, more specifically, bound tightening. Bound-tightening \cites{Ryoo:1996,tawarmalani2004,Belotti2009,Puranik:2017} refers to a family of methods designed to reduce the search space by adjusting the bounds of the variables. In \cite{Puranik:2017}, the authors provide a comprehensive numerical study highlighting the significant impact of domain reduction techniques, especially bound tightening, on several leading global solvers for MINLP problems. More recently, \cite{Gonzalez-Rodriguez:2023} and \cite{gomezcasares2024} report a comparable impact for polynomial optimization problems using solver \solver{RAPOSa}. 

In this paper we are concerned with an apparently minor aspect of bound tightening in lifted spaces which, to the best of our knowledge, has not been studied by past literature: potential trade-offs when tightening the bounds of the auxiliary variables. Typically, for solvers based on factorable reformulations, the convergence of the resulting branch-and-bound scheme can be guaranteed by branching just on the original variables. \solver{Xpress} \cite{xpress}, for instance, gives priority to branching on the original variables, and only resorts to branching on variables in the lifted space in particular cases where they are anticipated to be particularly effective, whereas \solver{Couenne} \cite{Belotti2009} makes no distinction between original and auxiliary variables once the branch-and-bound starts. Given the tight connections between the original and the auxiliary variables in the reformulated problem, one might expect that the bounds on the original variables implicitly propagate to the auxiliary ones and so one might wonder to what extent it is necessary to explicitly define and subsequently tighten the bounds of the auxiliary variables. On the one hand, the extra tightening should result in smaller trees but, on the other, it is not clear what is the impact that making all these bounds explicit may have on the performance of the auxiliary solver used at the different nodes of the tree. In order to analyze this trade-off, we focus on the RLT technique for polynomial optimization. There are two main reasons for this choice. First, in the RLT technique, only the original variables are used for branching and, therefore, we are guaranteed that convergence does not get compromised if the bounds on the auxiliary variables are not tightened or even passed to the linear solver. Second, the linear RLT relaxations solved at the nodes of the tree are tightened with the so called bound-factor constraints. In this context, we formally prove that, when the bound-factor constraints are present, the bounds on the original variables imply easily computable bounds for the auxiliary ones. Thus, since explicitly tightening the auxiliary variables with these bounds does not change the feasible region of the relaxations, in theory, their addition should not impact the size of the resulting branch-and-bound tree, at least in settings in which relaxations have unique optimal solutions. In computational practice, however, the high degeneracy typical of lifted formulations means these relaxations frequently possess multiple optimal solutions. Consequently, explicitly enforcing these redundant bounds can shift the specific optimal solution returned by the LP solver, altering subsequent branching decisions and ultimately impacting the overall size of the search tree. Although to a lesser extent, even if the optimal solutions coincide, minor numeric differences in the returned variable values might also lead to different branching paths.

Driven by the above discussion, the main goal of this paper is to gain understanding of the impact of tightening the bounds of the auxiliary variables in spatial branch-and-bound algorithms. In order to do so, we conduct a relatively simple numerical analysis. We compare the performance of different configurations of an RLT-based solver: depending on how the bounds of the auxiliary variables are handled, depending on whether or not bound-tightening techniques are enabled, and also depending on whether or not the test sets are restricted to continuous problems. We find that the impact of explicitly incorporating the bounds of the auxiliary variables is significant and, somewhat surprisingly, it can go in any direction with no clear pattern behind it. More precisely, we find that:
\begin{itemize}
	\item The direction and the magnitude of the impact on performance varies significantly across the auxiliary linear solvers used to solve the node relaxations.
	\item The direction of the impact varies across test sets. Even after restricting attention to a specific linear solver, the direction of the impact may vary across test sets.
	\item The direction and the magnitude of the impact also depend on whether or not bound tightening techniques were enabled in the RLT-based solver and on whether or not the problems under consideration have integer variables. Further, the direction of the impact may again vary across solvers and test sets.
    \item Importantly, the magnitude of the impact seems to go beyond the expected computational noise stemming from the combinatorial nature of branch-and-bound algorithms. The observed variability is substantially larger than what is produced by neutral perturbations, such as reversing the variable and constraint orders.
\end{itemize}

The aforementioned impact is significant for different measures of performance such as solve time, gap, number of nodes of the tree, and solve time of the linear relaxations. Despite our efforts and the extended analysis reported in the manuscript, we have not been able to find any consistent pattern underlying this behavior. To advance this line of inquiry, we turn to statistical learning in an effort to shed light on the potential existence of underlying patterns. Should such patterns be present, they might be used to predict, on a given instance, what approach regarding the bounds of the auxiliary variables would perform better. To this end, we build upon the general learning framework developed in \cite{ghaddar2023}, with the evidence suggesting that the observed variability in the results cannot be attributed solely to randomness. In particular, the improvements obtained by following the trained model suggest that further progress could be achieved with learning techniques tailored to this setting. These techniques should be designed not only to enhance performance but, more importantly, to improve understanding of the underlying patterns.

In summary, the contribution of this paper lies in raising awareness of the potentially large impact of a seemingly minor decision: whether or not to specify bounds for auxiliary variables. Crucially, this impact may depend heavily on the auxiliary solver used to handle the relaxations. Gaining a deeper understanding of the mechanisms at play may open the door to performance gains of RLT-based algorithms and, more broadly, of global optimization solvers that rely on lifted formulations.

The rest of the paper is structured as follows. In Section~\ref{sec:framework} we present the general polynomial optimization framework, including the RLT-technique and some considerations regarding bound tightening. In Section~\ref{sec:proof} we present the theoretical result for the (implicit) bounds on the auxiliary variables of the RLT relaxations. In Section~\ref{sec:testing_learning} we present the testing and the learning environments. Section~\ref{sec:computational} contains the core computational experiments. We conclude in Section~\ref{sec:conclusions}.

\section{Framework for the analysis}\label{sec:framework}
\subsection{Spatial branch-and-bound algorithm}
We frame our analysis in the context of polynomial optimization problems, relying on the reformulation–linearization technique, introduced in \cite{Sherali1992} and further refined in more recent papers such as \cite{Sherali2012reduced}, \cite{Sherali2012cuts}, \cite{Dalkiran2016}, \cite{dalkiran2018}, \cite{Gonzalez-Rodriguez:2023}, \cite{ghaddar2023}, \cite{raposaconic}, and \cite{raposainteger}. We consider (continuous) polynomial optimization problems given by
\begin{equation}
\begin{aligned}
\text{minimize} & \quad \phi_0(\x)\\
\text{subject to}  & \quad \phi_r(\x)\geq \beta_r, & r=1,2,\ldots, m_1 \\
& \quad \phi_r(\x)=\beta_r, & r=m_1+1,\ldots,m\\
& \quad \x\in\Omega \subset \mathbb{R}^{\card{N}}\text{,}
\end{aligned}
\label{eq:PO}
\tag*{$PP(\Omega)$}
\end{equation}
where $N$ denotes the set of variables, each $\phi_r(\mathbf{x})$ is a polynomial of degree $\delta_r \in \mathbb{N}$, and the set $\Omega = \{ \x \in \mathbb{R}^{\card{N}}: 0 \leq l_j \leq x_j \leq u_j < \infty, \text{ with } l_j<u_j,\, \forall j \in N \} \subset \mathbb{R}^{\card{N}}$ is a hyperrectangle containing the feasible region. The degree of the problem is given by $\delta=\max_{r \in \{0,\ldots,R\}} \delta_r$.

Next, in order to formally define the linear relaxation of~\ref{eq:PO}, it is convenient to introduce the notion of multiset. Given a finite set $N$ and a map $\mu: N \to \mathbb{N} \cup \{ 0 \}$, a multiset is a pair $(N,\mu)$ where, for each $j\in N$, $\mu(j)$ denotes $j$'s multiplicity. The cardinality of a multiset is given by $\lvert (N,\mu) \rvert = \sum_{j \in N} \mu(j)$. Consider two multisets, $(N, \mu_1)$ and $(N, \mu_2)$. We define their sum as the multiset $(N, \mu)=(N, \mu_1) + (N, \mu_2)$ where, for each $j\in N$, $\mu(j) = \mu_1(j) + \mu_2(j)$ and we define their union as the set of all $j\in N$ for which $\mu(j)\neq 0$. Finally, given a set $N$ and $\lambda\in \mathbb N$, we denote by $\ndelta{N}{\lambda}$ the collection of multisets $\{ (N,\mu) \text{ s.t. } \lvert (N, \mu) \rvert = \lambda \}$. 

The \textit{reformulation-linearization technique}, RLT, solves problem~\ref{eq:PO} through a reformulation that can readily be relaxed as a linear problem, which is then embedded in a branch-and-bound scheme. This is achieved by replacing each monomial of degree greater than one with an auxiliary variable. For example, variable $X_{1224}$ would replace monomial $x_1x_2x_2x_4$. More generally, these auxiliary variables, called RLT variables, are defined, for a given multiset $J$ with $\lvert J \rvert>1$, by
\begin{equation}
  X_J = \prod_{j \in J}x_j.
  \label{eq:RLTidentity}
\end{equation}
Since the cardinality of a multiset coincides with the degree of the associated monomial, we use cardinality and degree interchangeably. In order to get tighter relaxations and ensure convergence of the resulting algorithm, additional constraints, called bound-factor constraints, are also added {(and linearized)}. For each $\lambda\in \mathbb N$ and each pair of multisets $J_1$ and $J_2$ such that $J_1 + J_2 \in \ndelta{N}{\lambda}$, the corresponding bound-factor constraint has degree $\lambda$ and is given by
\begin{equation}
F_{\Omega}^{\lambda}(J_1,J_2)=\prod_{j\in J_1}{(x_j-l_j)}\prod_{j\in J_2}{(u_j-x_j)}\geq 0.
\label{eq:BFC}
\end{equation}
Now, let $[\cdot]_L$ denote the linearization of a polynomial obtained by replacing its monomials with their associated RLT variables. Then, the linear relaxation of~\ref{eq:PO} is defined as
\begin{equation}
\begin{aligned}
\text{minimize} & \quad [\phi_0(\x)]_L & \\
\text{subject to}  & \quad [\phi_r(\x)]_L \geq \beta_r, & r=1,2,\ldots, m_1 \\
& \quad [\phi_r(\x)]_L =\beta_r, & r=m_1+1,\ldots,m\\
& \quad [F_{\Omega}^{\delta}(J_1,J_2)]_L \geq 0, & J_1 + J_2 \in \ndelta{N}{\delta}\\
& \quad \x\in\Omega \subset \mathbb{R}^{\card{N}}\text{.}\\
\end{aligned}
\label{eq:LP}
\tag*{$LP(\Omega)$}
\end{equation}

In Figure~\ref{fig:RLTalg} we present the general scheme of an RLT-based algorithm, which proceeds by defining a series of linear relaxations, $LP(\Omega^k)$, of polynomial optimization problems analogous to~\ref{eq:PO}, but with respect to different hyperrectangles $\Omega^k$. For each $j \in N$ and each $k \in \N$, let $l^k_j$ and $u^k_j$ denote the lower and upper bounds of variable $x_j$ in $LP(\Omega^k)$. Note that the bound-factor constraints of $LP(\Omega^k)$ depend on the $l^k_j$ and $u^k_j$ bounds and, therefore, they vary across nodes.  Once a linear relaxation $LP(\Omega^k)$ has been solved and an optimal solution $(\xb, \Xb)$ is available, $\theta_j^k$ denotes a measure of the violation of the RLT-defining identities, Eq.~\eqref{eq:RLTidentity}, in which variable $j \in N$ is involved.\footnote{Refer to Section~5.4 in \cite{Gonzalez-Rodriguez:2023} for additional details.}

\begin{figure}[!htbp]
\centering
\begin{tcolorbox}
\small
\textbf{Initialization.} Let $LB\vcentcolon=-\infty$, $UB\vcentcolon=+\infty$, and $\x^{best}\vcentcolon=\emptyset$. Let $\itercounter\vcentcolon=1$, $\Omega^1 \vcentcolon= \Omega$, $\cola\vcentcolon=\{1\}$, and $LB^1=-\infty$.

\textbf{Stage 1 \emph{(main)}.} Choose $k\in \cola$ such that $LB^k=\min_{s\in \cola} LB^s$. Let $\cola \vcentcolon= \cola \backslash \{ k\}$. Solve problem $LP(\Omega^k)$.
	\begin{shortitem}
		\item If $LP(\Omega^k)$ is infeasible, go to {\bf Stage~2}.
		\item If $LP(\Omega^k)$ is feasible, let $\zbar^k$ be the optimal value and let $(\xb^k, \Xb^k)$ be an optimal solution.
	\begin{shortitem}
		\item If $\zbar^k<UB$ and $\theta_j^k = 0$ for all $j \in N$:
		\begin{shortitem}
			\item \emph{Update upper bound.} $UB\vcentcolon=\zbar^k$. Let $\x^{best}\vcentcolon=\xb^k$.
			\item \emph{Prune.} Remove all $s\in \cola$ such that $LB^s \geq UB$. Go to {\bf Stage~2}.
		\end{shortitem} 
		\item If $\zbar^k < UB$ and $\theta_j^k > 0$ for some $j \in N$:
		\begin{shortitem}
			\item \emph{Branch.} Choose $p \in N$ such that $\theta_p^k = \max_{j \in N} \theta_j^k$. Branch at $\beta \vcentcolon= \bar x^k_p$, by defining $\Omega^{\itercounter+1} \vcentcolon= \Omega^k \cap \{ \x \in \R^{\vert N \vert} \text{ s.t. } l^k_p \leq x_p \leq \beta \}$ and $\Omega^{\itercounter+2} \vcentcolon= \Omega^k \cap \{ \x \in \R^{\vert N \vert} \text{ s.t. } \beta \leq x_p \leq u^k_p \}$.
	\item Update queue $\cola \vcentcolon= \cola \cup \{ \itercounter+1, \itercounter+2\}$. Let $LB^{\itercounter+1} = LB^{\itercounter+2} \vcentcolon= \zbar^k$. Let $\itercounter \vcentcolon= \itercounter + 2$. Go to {\bf Stage~2}.
		\end{shortitem}
	\item If $\zbar^k \geq  UB$. \emph{Prune branch.} Go to {\bf Stage~2}.
	\end{shortitem}
\end{shortitem}

\textbf{Stage 2 \emph{(control)}.}
\emph{Update lower bound.} $LB\vcentcolon=\min\{\min_{k\in \cola}{LB^k}, UB \}$.
\begin{shortitem}
	\item If $LB=UB$, \stop:
	\begin{shortitem}
		\item If $\x^{best}=\emptyset$, \ref{eq:PO} is infeasible.
		\item If $\x^{best}\neq\emptyset$, $\x^{best}$ is a global optimum and $UB$ is the optimal value of \ref{eq:PO}.
	\end{shortitem} 
	\item Otherwise, go to {\bf Stage~1}.
\end{shortitem}
\end{tcolorbox}
\caption{Basic scheme of an RLT-based algorithm.}
\label{fig:RLTalg}
\end{figure}

For the sake of the exposition, we focus for now on continuous problems, while noting that the computational experiments in Appendix~\ref{app:additional} also cover mixed-integer instances. The latter are addressed via the natural extension of the RLT technique developed in \cite{raposainteger}, where the LP relaxations $LB^k$ are replaced by their MILP counterparts.

\subsection{J-set filtering strategy}\label{sec:j-set}

To alleviate the exponential proliferation of bound-factor constraints in RLT-based schemes, \cite{Sherali2013} studied different families of filtered bound-factor constraints, alongside convergence results. One of them, called $J$-set strategy, displayed a remarkable trade-off between tightness and computational efficiency and, since then, the $J$-set approach has become a standard filtering strategy in modern RLT implementations. We present it formally below.

Let $\mathcal{M}$ denote the set of multisets corresponding to all nonlinear monomials present in \ref{eq:PO}. A first strategy discussed in \cite{Sherali2013} consists of generating just the bound-factor constraints associated with the elements of $\mathcal{M}$. Then, they propose the $J$-set routine, which avoids the generation of dominated bound-factor constraints. To this end, they define $\mathcal{J}$ as the set of maximal elements in $\mathcal{M}$ under multiset inclusion. Under the $J$-set routine, bound-factor constraints are generated exclusively for the elements of $\mathcal{J}$. Formally, for each $J\in \mathcal{J}$, the following bound-factor constraints are generated:
  \[
	[\prod_{j \in J_1} (x_j - l_j) \prod_{j \in J_2} (u_j - x_j)]_L \ge 0, \quad \forall (J_1 + J_2) = J.
	\]

\subsection{Bound tightening: OBBT and FBBT}\label{subsec:bt}

Although, by definition of~\ref{eq:PO}, all variables in $N$ have to be bounded, their bounds might be unnecessarily large, yielding a hyperrectangle $\Omega$ composed mainly of infeasible points or, in other words, leading to an unnecessarily large search space. Bound-tightening techniques are designed to tighten the bounds of the variables, which would reduce the size of $\Omega$ and, importantly, increase the tightness of the bound-factor constraints (and of their linear relaxations), speeding up the convergence of the branch-and-bound algorithm. The two most prominent methods are i)~optimality-based bound tightening, OBBT, in which bounds are tightened by solving a series of relaxations of minor variations of the original problem and ii)~feasibility-based bound tightening, FBBT, in which tighter bounds are deduced directly by exploiting the relations between the problem constraints and the bounds of the variables. Given that OBBT is computationally more intensive than FBBT, a standard practice is to perform OBBT at the root node only (or at specific nodes of the tree) and use FBBT at each and every node. This is the approach taken by the RLT-based algorithm used in our analysis:
\begin{description}
	\item[OBBT.] It is run only at the root node by solving two optimization problems for each variable $j\in N$. These problems consist of minimizing and maximizing $x_j$ subject to the feasible region of problem~\ref{eq:LP}.
	\item[FBBT.] It is run at all nodes of the tree by using constraint propagation techniques on~$PP(\Omega^k)$, following the interval arithmetic approach in \cite{Belotti:2012}.
\end{description}

\subsection{Bounds on RLT variables and their role in the RLT-based algorithm}

By definition, the RLT variables are (implicitly) bounded by the products of the bounds of the original variables present in them, \emph{i.e.}, for each multiset $J$, we have $\prod_{j\in J} l_j  \leq X_J\leq \prod_{j\in J} u_j$. Yet, since constraint~\eqref{eq:RLTidentity} is not present in~\ref{eq:LP} or in the subsequent $LP(\Omega^k)$ problems, it might be useful to tighten these linear relaxations by explicitly imposing these bounds on the RLT variables. Driven by this observation, the main goal of this paper is precisely to assess the impact of this tightening of the linear relaxations and understand the associated trade-offs. Hereafter, we use \RLTlooseB and \RLTtightB to refer to configurations of the RLT-based algorithm in which these bounds on RLT variables are omitted or specified, respectively, in the $LP(\Omega^k)$ relaxations solved by the algorithm. 

As discussed in the Introduction, moving from \RLTlooseB to \RLTtightB might lead to two main effects: i)~impact on the solving time of the linear relaxations due to the large number of additional bounds passed to the linear solver and ii)~impact on the size of the resulting branch-and-bound tree due to the additional tightness of the relaxations. We argue below that, from the theoretical standpoint, for the particular case of the RLT-based algorithms under study, no approach, \RLTlooseB or \RLTtightB, should be expected to systematically produce smaller trees.

Consider an arbitrary hyperrectangle $\bOmega= \{ \x \in \mathbb{R}^{\card{N}}: 0 \leq \bl_j \leq x_j \leq \bu_j < \infty, \, \forall j \in N \}$ and a degree $\delta \in \mathbb{N}$. The main result in Section~\ref{sec:proof}, Theorem~\ref{theor:RLTbounds}, states that, given a multiset $J\in N^\delta$, the bounds of the form  $\prod_{j\in J} \bl_j  \leq X_J\leq \prod_{j\in J} \bu_j$ are satisfied at any point that is feasible to all the linearizations of the bound-factor constraints $[F_{\bOmega}^{\delta}(J_1,J_2)]_L \geq 0, \  J_1 + J_2 \in \ndelta{N}{\delta}$.  Thus, since the bounds added to the RLT variables in \RLTtightB are redundant, if all subproblems solved by the branch-and-bound algorithm had a unique optimal solution, moving from \RLTlooseB to \RLTtightB would not change the resulting tree. Yet, in practice, even under uniqueness of the optimal solution of the linear relaxations, one might get different trees because of numerical precision when solving two equivalent relaxations (with and without explicit bounds). More importantly, the high dimensionality of~\ref{eq:LP}, induced precisely by the RLT variables, facilitates the existence of a plethora of optimal solutions at each node. Therefore, despite Theorem~\ref{theor:RLTbounds}, moving from \RLTlooseB to \RLTtightB might still impact not only the solving time of the relaxations, but also the size of the resulting tree. The direction of these differences in tree sizes due to the multiplicity of optimal solutions would be difficult to anticipate, since they would be the result of the combinatorial nature inherent to branch-and-bound algorithms. In particular, they should be expected to favor neither \RLTlooseB nor \RLTtightB. However, in Section~\ref{sec:computational} we show that there are significant differences in performance according to a wide variety of metrics (including the number of nodes), and moreover, we argue that these variations in performance cannot be solely attributed to the inherent variability in branch-and-bound trees due to their combinatorial nature.

It is also natural to wonder if the impact of making the RLT bounds explicit in  \RLTtightB may depend on whether or not bound tightening is performed by the RLT-based algorithm. Given that OBBT is run on problems with the same feasible region of~\ref{eq:LP}, Theorem~\ref{theor:RLTbounds} again implies that the outcome of the OBBT phase of the algorithm should not be significantly different between \RLTlooseB and \RLTtightB. Moreover, since FBBT is run directly on the nonlinear problem~\ref{eq:PO}, the RLT variables are not present and, hence, using \RLTlooseB or  \RLTtightB makes no difference for FBBT calls. Similarly, one may also want to study the impact of moving from continuous to mixed-integer problems. Since Theorem~\ref{theor:RLTbounds} still applies to the continuous relaxations of these mixed-integer problems, we get that the bounds of the form  $\prod_{j\in J} \bl_j  \leq X_J\leq \prod_{j\in J} \bu_j$ are also valid for them. Therefore, there is no reason to expect that the potential differences in performance between \RLTlooseB and \RLTtightB might depend on i)~whether or not bound tightening is enabled in the RLT-based algorithm or ii)~whether or not the instances under consideration have integer variables. Yet, the computational results reported in Appendix~\ref{app:additional} for these alternative experiments again show significant differences.

\section{Implied bounds on RLT variables}\label{sec:proof}
This section is fully devoted to proving that, given a linear relaxation of the form $LP(\Omega)$, the explicit addition of the bounds of the form $\prod_{j\in J} l_j  \leq X_J\leq \prod_{j\in J} u_j$ for the RLT variables does not change the feasible region, \emph{i.e.}, they are redundant. Indeed, they are redundant already if only the constraints given by the linearizations of the bound-factor constraints are considered.

Given the definition of the bound-factor constraints, the above result is hardly surprising, and can be obtained as a by-product of the theoretical analysis in \cite{Sherali1997} and \cite{dalkiran2018}, where the authors study the tightness of quadratic reformulations of~\ref{eq:PO}. Their proofs are based on binomial expansions, a standard tool when tackling dominance results for families of bound-factor constraints. In this section we present a direct proof based on linear programming duality that builds upon the construction of an explicit, closed-form dual solution. We hope this structural approach may serve as a useful tool for future research. Those readers just interested in the computational results may skip the rest of this section.

\subsection{Preliminaries}\label{subsec:prelim}
We start by presenting a couple of preliminary results related to the implications of the bound-factor constraints in~\ref{eq:LP}. 

\begin{lem}[Lemma~2 in \cite{Sherali1992} and Lemma~2-$J$ in \cite{barros2026}]\label{lem:ST1992}
Given a linear relaxation of the form $LP(\Omega)$, all the bound-factor constraints of degree smaller than $\delta$ are implied by the bound-factor constraints of degree $\delta$.
\end{lem}
\begin{proof}
Refer to \cite{Sherali1992} and \cite{barros2026}.\footnote{In \cite{barros2026}, a minor inaccuracy in the statement and proof of Lemma~2 in \cite{Sherali1992} is addressed.}
\end{proof}


\begin{cor}[Corollary 1.3 in \cite{Brais2022}]\label{cor:bounded} Given a linear relaxation of the form $LP(\Omega)$, its feasible region is bounded. In particular, all RLT variables are bounded.
\end{cor}
\begin{proof}
Refer to \cite{Brais2022}.
\end{proof}

Building upon the above result it is straightforward to show that, for quadratic problems, the RLT variables have, as lower (upper) bounds the products of the lower (upper) bounds of the variables present in them.

\begin{cor}\label{cor:bounds} Given a linear relaxation of the form $LP(\Omega)$, the bound-factor constraints of degree~2 imply that, for each monomial $J$ of degree~2, $\prod_{j\in J} l_j  \leq X_J\leq \prod_{j\in J} u_j$.
\end{cor}
\begin{proof}
Let $X_J$ be an RLT variable with $J=\{j_1,j_2\}$. Consider bound-factor constraint $F_{\Omega}^{2}(J,\emptyset)$, given by
\[
F_{\Omega}^{2}(J,\emptyset)=[(x_{j_1}-l_{j_1})(x_{j_2}-l_{j_2})]_L=X_J-l_{j_2}x_{j_1}-l_{j_1}x_{j_2}+l_{j_1}l_{j_2}\geq 0.
\]
Then, $X_j \geq l_{j_2}x_{j_1} + l_{j_1}x_{j_2} - l_{j_1}l_{j_2} \geq l_{j_2}l_{j_1} + l_{j_1}l_{j_2} - l_{j_1}l_{j_2} = l_{j_1}l_{j_2}$. For the upper bound, bound-factor constraint $F_{\Omega}^{2}(\{j_1\},\{j_2\})$, given by
\[
F_{\Omega}^{2}(\{j_1\},\{j_2\})=[(x_{j_1}-l_{j_1})(u_{j_2}-x_{j_2})]_L=-X_J+u_{j_2}x_{j_1}+l_{j_1}x_{j_2}-l_{j_1}u_{j_2}\geq 0.
\]
Then, $X_J\leq u_{j_2}x_{j_1} + l_{j_1}x_{j_2} - l_{j_1}u_{j_2} \leq u_{j_2}u_{j_1} + l_{j_1}u_{j_2} - l_{j_1}u_{j_2}= u_{j_2}u_{j_1}$.
\end{proof}

The goal of this section is to extend the above result for the RLT variables associated with monomials of any degree. More precisely, we want to prove the following theorem.
\begin{theor}\label{theor:RLTbounds} Given a linear relaxation of the form $LP(\Omega)$, the bound-factor constraints imply that, for each monomial $J$, $\prod_{j\in J} l_j  \leq X_J\leq \prod_{j\in J} u_j$. Moreover, the same result holds, under the $J$-set filtering strategy, for each monomial $J$ contained in a $J$-set of \ref{eq:PO}.
\end{theor}

We prove the above result by induction on the cardinality of $J$. The case $\card{J}=2$ follows from Corollary~\ref{cor:bounds}. Thus, we assume that the result holds when $\card{J}<\card{N}$ and prove it for the case $\card{J}=\card{N}$. For the sake of exposition, we do the proof for the case $J=N=\{1,2,\ldots,n\}$. The other cases, where $J$ can have repeated variables such as $\{1144\}$, would be completely analogous, just working with the repeated variables as if they were different from each other.

The proof consists of showing that the optimal solution of the linear problem given by
\begin{equation}
\begin{aligned}
\text{minimize} & \quad X_{N} & \\
\text{subject to} & \quad [F_{\Omega}^{\card{N}}(J_1,J_2)]_L \geq 0, & J_1 + J_2 \in \ndelta{N}{\card{N}},\ J_1\cup J_2=N\\
& \quad \prod_{j\in J} l_j  \leq X_J, & \text{for all } J \text{ with } \card{J}<\card{N},  \\
\end{aligned}
\label{eq:LPproof}
\end{equation}
is precisely $\prod_{j\in N} l_j$. Thus, Theorem~\ref{theor:RLTbounds} holds irrespectively of the constraints of~\ref{eq:PO}. Further, note that not all the bound-factor constraints of degree $\card{N}$ are considered. By restricting to those for which $J_1\cup J_2=N$, we are disregarding those in which $J_1$ or $J_2$ have some repeated variables. Thus, since we will show that $\prod_{j\in N} l_j$ is a valid lower bound for $X_{N}$ in a relaxed version of $LP(\Omega)$, the bound also remains valid for $LP(\Omega)$. Note that the induction hypothesis is already used in the formulation of the above linear problem, through the inclusion of the lower bounds on variables of degree smaller than $\card{N}$. Finally, the extension to the $J$-set filtering strategy is straightforward, since $N$ in Eq.~\eqref{eq:LPproof} stands for any arbitrary monomial; thus, the arguments directly apply to any $J$-set of~\ref{eq:PO} and, by Lemma~\ref{lem:ST1992}, also to its subsets.

It is straightforward to check that, by the definition of the bound-factor constraints, the solution given, for each $R\subseteq N$, by $x_R=\prod_{i \in R} l_i$, is a feasible solution with cost $\prod_{i \in N} l_i$. The core of the proof consists of showing that there is a feasible solution of the dual problem with the same objective. Then, by the weak duality theorem, we would have proved that $\prod_{i \in N}l_i$ is the optimal value of the primal problem, which is the desired result.

In order to establish that $X_{N} \leq \prod_{j\in N} u_j$ it would suffice to work with a maximization problem instead of a minimization one and consider the upper bounds instead of the lower bounds. Then, the proof would follow similar arguments to those of the minimization case.\footnote{Alternatively, one can also change the problem of maximizing $X_N$ into the problem of minimizing $-X_N$. Then perform the change of variable $Y_R=-X_R$ and reformulate with respect to the $Y_R$ variables to obtain a problem analogous to the one in Eq.~\eqref{eq:LPproof}.} 

Bound-factor constraints of the form $F_{\Omega}^{\lambda}(T,\co{T})$, with $\co{T}= N\backslash T$, are denoted by $\bfc{T}$, \emph{i.e.}, variables $i\in T$ appear as $(u_i-x_i)$ and variables $j \in \co{T}$ as $(x_j-l_j)$. For the sake of exposition we use, interchangeably, $x_i$, $X_{\{i\}}$, and $X_i$ to refer to the original variables.

\subsection{\texorpdfstring{Proof illustration when $|N|=3$}{Proof illustration when |N|=3}}\label{subsec:n3}
To provide some initial intuition for the proof of Theorem~\ref{theor:RLTbounds}, we start discussing the case $\card{N}=3$, in which we can explicitly present the linear problem in Eq.~\eqref{eq:LPproof}. The problem consists of minimizing RLT variable $X_{123}$ with respect to the following constraints:

{\small
\[
\renewcommand{\arraystretch}{1.1}
\begin{array}{c|ccccccc|ccc}
	 & X_1 & X_2 & X_3 & X_{12} & X_{13} & X_{23}& X_{123} & & \geq RHS& \\
	\hline
	X^{\text{LB}}_1 & 1 & & & & & & & & l_1 & \\
	X^{\text{LB}}_2 & & 1& & & & & & & l_2 & \\
	X^{\text{LB}}_3 & & & 1 & & & & & & l_3 & \\
	X^{\text{LB}}_{12} & & & & 1 & & & & & l_1l_2 & \\
	X^{\text{LB}}_{13} & & & & & 1& & & & l_1l_3 & \\
	X^{\text{LB}}_{23} & & & & & & 1& & & l_2l_3 & \\
	\bfcfull{\emptyset}{\{123\}} & l_2l_3 & l_1l_3 & l_1l_2 & -l_3 & -l_2 & -l_1& 1& & l_1l_2l_3 & \\
	\bfcfull{\{1\}}{\{23\}} & -l_2l_3 & -l_3u_1 & -l_2u_1 & l_3 & l_2& u_1& -1& & -l_2l_3u_1 & \\
	\bfcfull{\{2\}}{\{13\}} & -l_3u_2 & -l_1l_3 & -l_1u_2 & l_3 & u_2& l_1& -1& & -l_1l_3u_2 & \\
	\bfcfull{\{3\}}{\{12\}} & -l_2u_3 & -l_1u_3 & -l_1l_2 & u_3 & l_2& l_1& -1& & -l_1l_2u_3 & \\
	\bfcfull{\{12\}}{\{3\}} & l_3u_2 & l_3u_1 & u_1u_2 & -l_3 & -u_2& -u_1& 1& & l_3u_1u_2	& \\
	\bfcfull{\{13\}}{\{2\}} & l_2u_3 & u_1u_3 & l_2u_1 & -u_3 & -l_2& -u_1& 1& & l_2u_1u_3	& \\
	\bfcfull{\{23\}}{\{1\}} & u_2u_3 & l_1u_3 & l_1u_2 & -u_3 & -u_2& -l_1& 1& & l_1u_2u_3	& \\
	\bfcfull{\{123\}}{\{\emptyset\}} & -u_2u_3 & -u_1u_3 & -u_1u_2 & u_3 & u_2& u_1& -1& 	& -u_1u_2u_3	& \\[0.1cm]
\end{array}
\]
}
We now present a feasible solution of the dual problem with objective function $\prod_{i \in \{1,2,3\}} l_i$, which, as discussed above, suffices to prove that $\prod_{i \in \{1,2,3\}} l_i$ is the optimal value.  In the dual problem, the constraint matrix is the transpose of the one above, the right-hand side has all coefficients equal to zero except for the one associated with variable $X_{123}$, which takes value~1, and the objective function coefficients are given by the values in column RHS above. Further, given that there are no negativity constraints in the primal, the constraints in the dual problem are equality constraints and, similarly, since the primal constraints are of the form ``$\geq$'', the dual variables have to be nonnegative.

Consider the dual solution given below:
{\small
\[
\renewcommand{\arraystretch}{1.15}
\begin{array}{c|c|c|c|c|c}
\text{Constraint} & X^{\text{LB}}_3 & X^{\text{LB}}_{12} & \bfcfull{\emptyset}{\{123\}}& \bfcfull{\{1\}}{\{23\}}& \bfcfull{\{2\}}{\{13\}} \\
\hline
\text{Dual value} & l_1 l_2 & l_3 & 1+ \frac{l_1}{u_1-l_1} + \frac{l_2}{u_2-l_2} & \frac{l_1}{u_1-l_1} & \frac{l_2}{u_2-l_2} \\
\end{array},
\]}with the dual value of all other constraints being zero. It can be readily verified that this solution is feasible and that the objective function takes value $\prod_{i \in \{1,2,3\}} l_i$. For the sake of illustration, we present the calculations for the constraint associated with variable $X_1$ and for the objective function:
\paragraph{Constraint associated with $X_1$:}
    {\small
    \[
        \begin{array}{r}
        \big(1+ \frac{l_1}{u_1-l_1} + \frac{l_2}{u_2-l_2}\big)\cdot (l_2 l_3) + \frac{l_1}{u_1-l_1} \cdot (-l_2 l_3)+ \frac{l_2}{u_2-l_2} \cdot (-l_3 u_2)= \\[0.2cm]
        = l_2 l_3 + \frac{l_1l_2l_3}{u_1-l_1} + \frac{(l_2)^2 l_3}{u_2-l_2}-  \frac{l_1l_2l_3}{u_1-l_1} - \frac{l_2l_3u_2}{u_2-l_2} = l_2 l_3 + (l_2 l_3) \frac{l_2-u_2}{u_2-l_2} = l_2 l_3 - l_2 l_3 = 0.
        \end{array}
    \]}
\paragraph{Objective function:}
    {\small
    \[
        \begin{array}{r}
        (l_1l_2)\cdot l_3 + l_3 \cdot (l_1 l_2) + \big(1+ \frac{l_1}{u_1-l_1} + \frac{l_2}{u_2-l_2}\big)\cdot (l_1 l_2 l_3) + \frac{l_1}{u_1-l_1} \cdot (-l_2 l_3 u_1)+ \frac{l_2}{u_2-l_2} \cdot (-l_1 l_3 u_2)= \\[0.2cm]
        3 l_1 l_2 l_3 + l_1 l_2 l_3 \cdot \big(\frac{l_1}{u_1-l_1} + \frac{l_2}{u_2-l_2} - \frac{u_1}{u_1-l_1} -\frac{u_2}{u_2-l_2}  \big)= 3 l_1 l_2 l_3 - 2 l_1 l_2 l_3 = l_1 l_2 l_3.
        \end{array}
    \]}

\subsection{General case}\label{subsec:proof}

In this section we show that the construction we presented above for the case $\card{N}=3$ can be generalized to any cardinality. As discussed for the case $\card{N}=3$, 

\subsubsection{General expressions for the coefficients of the primal problem.}

\paragraph{Objective function coefficients.} Since we are minimizing $X_N$, all coefficients are zero except the one associated with $X_N$, which is~1.

\paragraph{Right-hand side bounds.} Given $R\subsetneq N$, the RLT variable $X_R$ is subject to the bound constraint $X_R\geq \prod_{i\in R} l_i$ (induction hypothesis).\footnote{The upper-bound constraints would be needed to derive the upper bound on $x_N$, but not for the lower bound.}

\paragraph{Right-hand side coefficients.} Given $T\subseteq N$, the right-hand side of constraint $\bfc{T}$ is given by
	\[\bfc{T}(RHS):=-\Big(\prod_{i\in T} u_i \prod_{j\in \co{T}} (-l_j)\Big)=(-1)^{\card{\co{T}}+1}\prod_{i\in T} u_i \prod_{j\in \co{T}} l_j.
	\]

\paragraph{Coefficients of the constraints' matrix.} Given $T\subseteq N$ and an RLT variable $X_R$, with $R\subseteq N$, its coefficient in constraint $\bfc{T}$ is given by
	\[
	\bfc{T}(R):=(-1)^{\card{T \cap R}} \prod_{i\in T\backslash R} u_i \prod_{j\in \co{T}\backslash R} (-l_j)=(-1)^{\card{T \cap R}+\card{\co{T}\backslash R}} \prod_{i\in T\backslash R} u_i \prod_{j\in \co{T}\backslash R} l_j.
	\]
The following result presents an equivalent expression for $\bfc{T}(R)$. 

\begin{lem}\label{lem:minus1}
Given $T\subseteq N$ and an RLT variable $X_R$, with $R\subseteq N$, its coefficient $\bfc{T}(R)$ can be equivalently written as
	\[
	\bfc{T}(R)=(-1)^{n-\card{R}+\card{T}} \prod_{i\in T\backslash R} u_i \prod_{j\in \co{T}\backslash R} l_j.
	\]
\end{lem}
\begin{proof} We need to show that, given $T\subseteq N$ and $R\subseteq N$, the numbers $\card{T \cap R}+\card{\co{T}\backslash R}$ and $n-\card{R}+\card{T}$ have the same parity or, equivalently, that their sum is an even number. First, note that
\[
\card{\co{T}\backslash R}=\card{\co{T}\cap \co{R}}=\card{\overline{T\cup R}}=n-\card{T\cup R}=n-\card{T}-\card{R}+\card{T\cap R}.
\]
Then, $\card{T \cap R}+\card{\co{T}\backslash R}=n-\card{T}-\card{R}+2\card{T\cap R}$ and we have
\[
\card{T \cap R}+\card{\co{T}\backslash R}+n-\card{R}+\card{T}=2(n-\card{R}+\card{T\cap R}). \qedhere
\]
\end{proof}

The expression in Lemma~\ref{lem:minus1} implies that the sign of the power of $-1$  in $\bfc{T}(R)$ can be computed just from the cardinality of $T$ and $R$, with no need to worry about the sets $T\cap R$ and $\co{T}\backslash R$. Moreover, when $R=\emptyset$ we get
\[
n-\card{R}+\card{T}=n+\card{T}=n+n-\card{\co{T}}=2n-\card{\co{T}}.
\]
Since $(-1)^{2n-\card{\co{T}}}=(-1)^{\card{\co{T}}}$, Lemma~\ref{lem:minus1} also implies that $\bfc{T}(RHS)=-\bfc{T}(\emptyset)$ and, hence, the expression for the right-hand side coefficients is a particular case of the formula for the coefficients of the constraint's matrix (but with the opposite sign because it has been moved to the right-hand side of the constraint).

\subsubsection{General expressions for the candidate solution of the dual problem.}

The dual solution is constructed taking variable~$n$ as reference, since we just need one of the infinitely many dual optimal solutions (stemming from the redundancies in the bound-factor constraints).

\paragraph{Reduced costs for the RLT variables (\emph{i.e.}, duals of bound constraints).} Only the reduced costs associated with $X_{\{n\}}$ and $X_{N\backslash \{n\}}$ are different from zero. They are given by
		\[ RC(X_{\{n\}})=\prod_{i \in N\backslash \{n\}}l_i \qquad \text{and} \qquad RC(X_{N\backslash \{n\}})=l_n.
		\]

\paragraph{Dual values for $\bfc{T}$ constraints.} Given $T\subseteq N$, let
	\[D(T):=\left\{\begin{array}{cl}
		0, & \text{if } n \in T \text{ or } \card{T}\geq n-1\\[0.1cm]
		1, & \text{if } T=\emptyset \\[0.1cm]
		\displaystyle \prod_{i\in T} \frac{l_i}{u_i-l_i}, & \text{otherwise}.
	\end{array}\right.
		\]
Then, we define the dual value of constraint $\bfc{T}$ as
			\[
			\bfc{T}(Dual) := \sum_{\hat{T} \colon T\subseteq \hat{T}} D(\hat{T}) =  D(T) + \sum_{\substack{\hat{T} \colon T\subsetneq \hat{T} \\ n\notin \hat{T},\; \card{\hat{T}}\leq n-2}} D(\hat{T}) .\]
Note that, by definition, $\bfc{T}(Dual)$ is zero if $n\in T$ or $\card{T}\geq n-1$, since the sums involved in the definition of $\bfc{T}(Dual)$ only contain $D(\hat T)$ addends with $T \subseteq \hat T$.

We want to show the following properties of the dual solution we have just defined:
\begin{itemize}
	\item \textbf{Dual feasibility.} For each $R\subsetneq N$, $R\neq \emptyset$, the inner product of our dual solution and the vector of the coefficients of the RLT variable $X_R$ in the constraint's matrix is zero. Given that all these dual constraints are equality constraints with right hand side zero, this would establish the feasibility of the dual solution with respect to these constraints.
	\item\textbf{Dual feasibility.} Similarly, for $R=N$, the corresponding inner product has to be~1 (the only right-hand side different from 0 in the dual, since in the primal we are minimizing $X_N$).
	\item\textbf{Dual objective.} The inner product of the dual solution and the RHS vector equals $\prod_{i\in N} l_i$, which would deliver the desired bound on $X_N$ by the weak duality theorem. Moreover, this theorem also implies the optimality of both this dual solution and the primal one in which each $X_R$ is given by $\prod_{i \in R} l_i$.
\end{itemize}

\subsubsection{\texorpdfstring{Dual feasibility: $R\neq \{n\}$, $R\neq N\backslash \{n\}$, and $R\neq N$}{Dual feasibility: R neq \{n\}, R neq N minus \{n\}, and R neq N}}\label{sec:mainfeas}

Given an RLT variable $R\subsetneq N$ with $R\neq \{n\}$ and $R\neq N\backslash \{n\}$, we have that the reduced cost associated with the bound constraint for $X_R$ is zero. Thus, to prove that the dual constraint associated with $R$ is satisfied at our dual solution, we have to show that
\[
\sum_{T\subseteq N} \bfc{T}(R)\cdot \bfc{T}(Dual)=\sum_{\substack{T\colon T\subseteq N \\ n\notin T,\; \card{T}\leq n-2}} \bfc{T}(R)\cdot \bfc{T}(Dual)=0.
\]
Equivalently, we have to show that
\begin{equation}
\sum_{\substack{T\colon T\subseteq N \\ n\notin T,\; \card{T}\leq n-2}} (-1)^{n-\card{R}+\card{T}} \prod_{i\in T\backslash R} u_i \prod_{j\in \co{T}\backslash R} l_j \cdot \sum_{\hat{T} \colon T\subseteq \hat{T}} D(\hat{T}) = 0.
\label{eq:mainRcons}
\end{equation}

Let $W\subseteq N$, with $n\notin W$ and $\card{W}\leq n-2$. We have that, in Eq.~\eqref{eq:mainRcons}, $D(W)$ appears once in the inner summation of each $T\subseteq W$. Thus, if we isolate the terms involving $D(W)$ we get:
\begin{equation}
	\sum_{T\subseteq W} \bfc{T}(R) \cdot D(W)=\sum_{T\subseteq W} (-1)^{n-\card{R}+\card{T}} \prod_{i\in T\backslash R} u_i \prod_{j\in \co{T}\backslash R} l_j \cdot D(W).
\label{eq:mainCancel}
\end{equation}
We distinguish two cases:
\paragraph{Case $W\cap R\neq\emptyset$.} Let $j\in W\cap R$ and let $T\subseteq W$ with $j\in T$. Then, one can readily verify that
\[
\bfc{T}(R)=-\bfcfull{T\backslash \{j\}}{\co{T}\cup \{j\}}(R),
\]
since they involve the same products of $u_i$ and $l_i$ coefficients and their powers are such that $n-\card{R}+\card{T}=n-\card{R}+\card{T\backslash \{j\}}+1$. Hence, the corresponding terms in Eq.~\eqref{eq:mainCancel}, $\bfc{T}(R)\cdot D(W)$ and $\bfcfull{T\backslash \{j\}}{\co{T}\cup \{j\}}(R)\cdot D(W)$, cancel out. Thus, whenever $W\cap R\neq\emptyset$, the overall coefficient of $D(W)$ is zero.

\paragraph{Case $W\cap R=\emptyset$.} Now, $\co{T}\backslash R=\co{W}\backslash R \cup (W\backslash T)$ and Eq.~\eqref{eq:mainCancel} can be equivalently written as
\begin{equation}
	\sum_{T\subseteq W} \bfc{T}(R) \cdot D(W)= (-1)^{n-\card{W}-\card{R}} \prod_{k\in \co{W}\backslash R} l_k \cdot \Big(\sum_{T \subseteq W}(-1)^{\card{W}+\card{T}} \prod_{i \in T} u_i \prod_{j\in W\backslash T} l_j  \Big) \cdot D(W).
\label{eq:uili}
\end{equation}
Since $\card{W}+\card{T}$ and $\card{W}-\card{T}$ have the same parity, the sum in the above parenthesis reduces to
\[
\sum_{T \subseteq W}(-1)^{\card{W}-\card{T}} \prod_{i \in T} u_i \prod_{j\in W\backslash T} l_j = \sum_{T \subseteq W} \Big( \prod_{i \in T} u_i \prod_{j\in W\backslash T} (-l_j) \Big) = \prod_{i \in W} (u_i-l_i).
\]
Therefore, Eq.~\eqref{eq:uili} can be equivalently written as
\[
	\sum_{T\subseteq W} \bfc{T}(R) \cdot D(W)= (-1)^{n-\card{W}-\card{R}} \prod_{k\in \co{W}\backslash R} l_k \prod_{i \in W} (u_i-l_i) \cdot D(W).
\]
%

Now, combining cases $W\cap R\neq \emptyset$ and $W\cap R=\emptyset$ above, Eq.~\eqref{eq:mainRcons} can be equivalently written as
\[
\sum_{\substack{W\colon W\subseteq N,\, W\cap R=\emptyset \\ n\notin W,\; \card{W}\leq n-2}} (-1)^{n-\card{W}-\card{R}}\cdot  \prod_{k\in \co{W}\backslash R} l_k \cdot \prod_{i \in W} (u_i-l_i) \cdot D(W)=0.
\]
Since $D(W)= \prod_{i\in W} \frac{l_i}{u_i-l_i}$, the left-hand side of the above equation reduces to
\begin{equation}
\sum_{\substack{W\colon W\subseteq N,\, W\cap R=\emptyset \\ n\notin W,\; \card{W}\leq n-2}} (-1)^{n-\card{W}-\card{R}} \cdot \prod_{k\in \co{W}\backslash R} l_k \cdot \prod_{i \in W} l_i = \sum_{\substack{W\colon W\subseteq N,\, W\cap R=\emptyset \\ n\notin W,\; \card{W}\leq n-2}} (-1)^{n-\card{W}-\card{R}}\cdot \prod_{i \in N\backslash R} l_i.
\label{eq:noDW}
\end{equation}
We have that $n\notin W$, $R\neq \{n\}$, and $R\neq \emptyset$. Then, $W\cap R=\emptyset$ implies that $R\cup \{n\} \subseteq \bar W$ and the condition $\card{W}\leq n-2$ is always satisfied. Thus, the above expression can be written as:
\begin{equation}
\prod_{i \in N\backslash R} l_i \cdot\sum_{W\colon W\subseteq N\backslash (R \cup \{n\})} (-1)^{n-\card{W}-\card{R}}.
\label{eq:no_n-2}
\end{equation}
Since $n$ and $R$ are fixed, the sign in each addend of the above summation just depends on $\card{W}$. Let $S=N\backslash (R \cup \{n\})$. Since $R\neq N\backslash \{n\}$ and $R\neq N$, we have that $\card{S}\geq 1$. The summation covers all the subsets of $S$ and the number of such subsets with cardinality $s$ is given by the combinatorial number $\card{S}$ over $s$. Therefore, Eq.~\eqref{eq:no_n-2} can be equivalently rewritten as:
\begin{equation}
\prod_{i \in N\backslash R} l_i \cdot \sum_{s=0}^{\card{S}} \binom{\card{S}}{s} (-1)^{n-s-\card{R}}. 
\label{eq:combinatorial}
\end{equation}
The above sum corresponds to the sum of the numbers of row $\card{S}$ in Pascal's triangle,\footnote{The first row in Pascal's triangle is usually numbered as row~$0$, since it corresponds to the terms of the expansion of a binomial with power zero (and with the cardinality of the subsets of sets of cardinality zero).} with alternating signs, which is known to be zero whenever $\card{S}\geq 1$.\footnote{Each element of the {$n$-th} row of the triangle contributes to two consecutive elements of row $n+1$, one of them being odd and the other one being even. Hence, the sum of the odd elements equals the sum of the even ones.}

Therefore, we have established the feasibility of our candidate dual solution with respect to the dual constraint associated with the RLT variable $R$, with $R\neq \{n\}$ and $R\neq N\backslash \{n\}$.

\subsubsection{\texorpdfstring{Dual feasibility: $R=\{n\}$}{Dual feasibility: R=\{n\}}}

The proof is analogous to the general case in Section~\ref{sec:mainfeas}, with minor adjustments. First, the dual constraint for $R=\{n\}$ contains the reduced cost $RC(X_{\{n\}})=\prod_{i \in N\backslash \{n\}}l_i$.
We have to show that
\[
RC(X_{\{n\}})+\sum_{T\subseteq N} \bfc{T}(R)\cdot \bfc{T}(Dual)=\prod_{i \in N\backslash \{n\}}l_i+\sum_{\substack{T\colon T\subseteq N \\ n\notin T,\; \card{T}\leq n-2}} \bfc{T}(R)\cdot \bfc{T}(Dual)=0.
\]
All the derivations in Section~\ref{sec:mainfeas} for the second addend of the above equation remain valid up to Eq.~\eqref{eq:noDW}, which is no longer equivalent to Eq.~\eqref{eq:no_n-2}. Since $\card{R}=1$, the condition $W\cap R=\emptyset$ does not imply that $\card{W}\leq n-2$. Thus, if we look at the sum in Eq.~\eqref{eq:no_n-2}, the set $\{W\colon W\subseteq N\backslash (R \cup \{n\})\}$ becomes $\{W\colon W\subseteq N\backslash \{n\}\}$, which contains $N\backslash \{n\}$. Thus, $N\backslash \{n\}$ has to be removed to recover the equivalence with Eq.~\eqref{eq:noDW}. Once this is taken into consideration by subtracting the corresponding addend of the sum and, moreover, by incorporating the reduced cost $RC(X_{\{n\}})$, we get
\begin{equation*}
\prod_{i \in N\backslash \{n\}}l_i+\prod_{i \in N\backslash \{n\}} l_i \cdot\sum_{W\colon W\subseteq N\backslash \{n\}} (-1)^{n-\card{W}-1}-\prod_{i \in N\backslash \{n\}} l_i \cdot (-1)^{n-\card{N\backslash \{n\}}-1}.
\end{equation*}
Since $n-\card{N\backslash \{n\}}-1=0$, the first and the third terms cancel out and we are back to the general expression of Eq.~\eqref{eq:no_n-2}. Thus, the combinatorial argument used there applies again and $\sum_{W\colon W\subseteq N\backslash \{n\}} (-1)^{n-\card{W}-1}=0$. Hence, the dual constraint associated with $R=\{n\}$ is also satisfied. 

\subsubsection{\texorpdfstring{Dual feasibility: $R=N\backslash \{n\}$}{Dual feasibility: R=N minus \{n\}}}

Again, the proof follows the same arguments of the general case in Section~\ref{sec:mainfeas}, with minor adjustments. First, the dual constraint for $R=N\backslash \{n\}$ contains the reduced cost $RC(X_{N\backslash \{n\}})=l_n$.
We have to show that
\[
RC(X_{N\backslash \{n\}})+\sum_{T\subseteq N} \bfc{T}(R)\cdot \bfc{T}(Dual)=l_n+\sum_{\substack{T\colon T\subseteq N \\ n\notin T,\; \card{T}\leq n-2}} \bfc{T}(R)\cdot \bfc{T}(Dual)=0.
\]
All the derivations in Section~\ref{sec:mainfeas} for the second addend of the above equation remain valid up to Eq.~\eqref{eq:combinatorial}. Now, $S=N\backslash (R \cup \{n\})=\emptyset$ and, once the reduced cost $RC(X_{N\backslash \{n\}})=l_n$ is incorporated, Eq.~\eqref{eq:combinatorial} becomes
\[
l_n + \prod_{i \in \{n\}} l_i \cdot \sum_{s=0}^{0} \binom{0}{0} (-1)^{n-0-\card{N\backslash \{n\}}}=l_n-l_n=0,
\]
and the dual constraint associated with $R=N\backslash \{n\}$ is also satisfied.

\subsubsection{\texorpdfstring{Dual feasibility: $R= N$}{Dual feasibility: R= N}}
Once again, the proof follows the same arguments of the general case, with minor adjustments. First, recall that the primal problem requires to minimize $X_N$ and, hence, the right-hand side of the dual constraint associated with $R= N$ takes value~1. Thus, we have to prove that 
\[
\sum_{T\subseteq N} \bfc{T}(R)\cdot \bfc{T}(Dual)=\sum_{\substack{T\colon T\subseteq N \\ n\notin T,\; \card{T}\leq n-2}} \bfc{T}(R)\cdot \bfc{T}(Dual)=1.
\]
Similarly to the case $R=N\backslash \{n\}$, all the derivations developed in Section~\ref{sec:mainfeas} for the second addend of the above equation remain valid up to Eq.~\eqref{eq:combinatorial}. Now, $S=N\backslash (R \cup \{n\})=\emptyset$ and Eq.~\eqref{eq:combinatorial} becomes
\[
\prod_{i \in \emptyset} l_i \cdot \sum_{s=0}^{0} \binom{0}{0} (-1)^{n-0-\card{N}}=1\cdot (-1)^0=1,
\]
and the dual constraint associated with $R=N$ is satisfied. 

\subsection{Dual objective}
We have to show that the dual objective associated with our dual solution is precisely $\prod_{i\in N} l_i$. Since the dual objective corresponds with the inner product of our dual solution with the $\bfc{T}(RHS)$ coefficients and, by Lemma~\ref{lem:minus1}, $\bfc{T}(RHS)=-\bfc{T}(\emptyset)$, we can also build upon the analysis in Section~\ref{sec:mainfeas}. First, we need to incorporate the right-hand side coefficients and reduced costs associated with $R=\{n\}$ and $R=N\backslash \{n\}$: $RC(X_{\{n\}})=\prod_{i \in N\backslash \{n\}} l_i$ and $RC(X_{N\backslash \{n\}})=l_n$. More precisely, we have to show that
\begin{equation}
\begin{array}{rcl}
\displaystyle l_n\cdot RC(X_{\{n\}}) + \prod_{i \in N\backslash \{n\}} l_i\cdot RC(X_{N\backslash \{n\}}) - \sum_{T\subseteq N} \bfc{T}(\emptyset)\cdot \bfc{T}(Dual) &=&\\[0.6cm]
\displaystyle \prod_{i\in N} l_i + \prod_{i\in N} l_i - \sum_{\substack{T\colon T\subseteq N \\ n\notin T,\; \card{T}\leq n-2}} \bfc{T}(\emptyset)\cdot \bfc{T}(Dual) &=&\prod_{i\in N} l_i.
\label{eq:dualobj}
\end{array}
\end{equation}
Similarly to the case $R=\{n\}$, all derivations for the expression of the third summation above remain valid up to Eq.~\eqref{eq:noDW}, which is no longer equivalent to Eq.~\eqref{eq:no_n-2}. Since $\card{R}=0$, the condition $W\cap R=\emptyset$ does not imply that $\card{W}\leq n-2$. Hence, if we look at the sum in Eq.~\eqref{eq:no_n-2}, the set $\{W\colon W\subseteq N\backslash (R \cup \{n\})\}$ becomes again $\{W\colon W\subseteq N\backslash \{n\}\}$, which contains $N\backslash \{n\}$. Thus, $N\backslash \{n\}$ has to be removed to recover the equivalence with Eq.~\eqref{eq:noDW}. Taking this into consideration by subtracting the corresponding addend of the sum, we get
\begin{equation*}
\prod_{i \in N} l_i \cdot\sum_{W\colon W\subseteq N\backslash \{n\}} (-1)^{n-\card{W}}-\prod_{i \in N\backslash \{n\}} l_i \cdot (-1)^{n-\card{N\backslash \{n\}}}=\prod_{i \in N} l_i \cdot0- \prod_{i \in N} l_i \cdot (-1)^1=\prod_{i \in N} l_i.
\end{equation*}
If we substitute the summation in Eq.~\eqref{eq:dualobj} with the above value, $\prod_{i \in N} l_i$, we get
\[
\prod_{i\in N} l_i + \prod_{i\in N} l_i - \prod_{i \in N} l_i = \prod_{i\in N} l_i,
\]
which concludes the proof.

\section{Testing environment and learning methodology}\label{sec:testing_learning}

\subsection{Testing environment}\label{sec:testing}
For the numerical results presented in this paper, we use the polynomial optimization solver RAPOSa~4.4.1 \cite{Gonzalez-Rodriguez:2023}, whose core is an RLT-based algorithm. For the computational experiments, we use instances from three well-known test sets. The first one, DS, is taken from \cite{Dalkiran2016} and consists of 180 instances of randomly generated continuous polynomial optimization problems of different degrees, number of variables, and density. The second test set comes from MINLPLib \cite{MINLPLib}, a reference library of mixed-integer nonlinear programming instances. We have selected from MINLPLib those instances that are polynomial optimization problems with box-constrained variables resulting in 189 continuous instances and 214 mixed-integer instances. Finally, the last test set comes from QPLIB \cite{qplib}, which is a library with quadratic instances in which we have selected those instances with box-constrained variables, ending up with 67 continuous instances and 141 mixed-integer instances.\footnote{RAPOSa solver can be downloaded from \url{https://raposa.usc.es/download/} and the instances used in this work are available at \url{https://raposa.usc.es/instances/}.}


All the executions have been run on the supercomputer Finisterrae~III, at Galicia Supercomputing Centre (CESGA). We used computational nodes powered with two thirty-two-core Intel Xeon Ice Lake 8352Y CPUs with 256GB of RAM and 1TB SSD. All the executions were run with a time limit of 1 hour and the stopping criterion was set at threshold $0.001$ for either the relative gap or the absolute gap. 

To compare the performance of the different approaches under consideration, we use the following metrics:
\begin{itemize}
 \item \textbf{Unsolved}. Number of instances not solved to certified optimality (relative or absolute gap below $0.001$).
 \item \textbf{Gap}. Geometric mean of the optimality gap obtained by each approach. We exclude instances for which i)~at least one approach did not return an optimality gap and ii)~all the approaches solved it within the time limit.
 \item \textbf{Time}. Geometric mean of the running time of each approach. We exclude instances for which i)~every approach solved it in less than 5 seconds and ii)~no approach solved it within the time limit.
 \item \textbf{Pace}. Geometric mean of pace$^{LB}$, as introduced in \cite{ghaddar2023}. It represents the number of seconds needed to improve the lower bound of the algorithm by one unit (the pace at which the lower bound improved). We exclude instances solved by every approach in less than 5 seconds. The main motivation behind this performance measure is that it allows to compare the performance on all instances together, whereas Time and Gap fail to do so. Time is not informative when comparing performance between instances not solved by any approach (all of them reach the time limit). These instances can be compared using Gap which, in turn, is not informative in instances solved by all approaches (all of them close the gap), and where Time could be used. Pace, on the other hand, is informative regardless of the number of approaches that might have solved each of the different instances. 
 \item \textbf{Nodes}. Geometric mean of the number of nodes explored in the branch-and-bound tree. We only consider the instances solved by all the approaches within the time limit.
  \item \textbf{L-Time}. Geometric mean of the computational time per node used by the auxiliary linear solver to solve the linear relaxations generated along the branch-and-bound tree. We only consider the instances solved by all the approaches within the time limit.

\end{itemize}

Note that, for all the previous metrics, the lower the better. This facilitates the discussion of the computational results.

\subsection{Learning methodology}\label{sec:learning_meth}

As noted in the Introduction, part of our numerical analysis investigates the extent to which learning techniques can reveal patterns underlying our main results and, in particular, predict whether \RLTlooseB or \RLTtightB will perform better on a given instance. We now provide a brief overview of the chosen learning approach.

We build upon the learning framework developed and thoroughly analyzed in \cites{ghaddar2023, Ghaddar2025}, and that has been recently applied in \cite{raposaconic} and \cite{gomezcasares2024}. In \cites{ghaddar2023}, the authors employ learning techniques to enhance the performance of the RLT-based solver \texttt{RAPOSa} by training a model to select among different branching rules. The reported gains were substantial, with the learning-based approach achieving improvements of up to $25\%$ compared to the best individual branching rule. We have chosen to adopt the relatively simple approach of \cites{ghaddar2023}, without introducing any customization or adaptation to our specific setting. As will be discussed later, more sophisticated techniques may plausibly offer further gains in understanding or performance. However, we have refrained from pursuing this direction, as it could divert attention from the main focus of the paper, namely the significant impact of the choice of \RLTlooseB or \RLTtightB may have on performance.

Table~\ref{table:features} presents the list of core input variables (features), from which additional transformations and normalizations are derived, yielding a final set of 129 variables.\footnote{VIG and CMIG stand for two graphs that can be associated to any given polynomial optimization problem: \textit{variables intersection graph} and \textit{constraints-monomials intersection graph} (see \cite{ghaddar2023}).} Then, as concisely summarized in \cite{raposaconic}: ``The other main ingredient is the chosen performance indicator. We consider $\NLBpace$, a normalized version of $\LBpace$ with values in $[0,1]$. It is computed, for each version of the algorithm, by dividing the best (smallest) pace among all versions to be compared by the pace of the current one. Thus, the goal of the learning is to predict the performance ($\NLBpace$) of each configuration on a new instance based on a regression analysis of its performance on the training instances. Then, the configuration with the best predicted performance for the given instance is chosen.''

\begin{table}[!htbp]
\centering
{\footnotesize
\renewcommand{\arraystretch}{0.6}
\setlength{\tabcolsep}{3pt}
    \begin{tabular}{ll}
      \toprule
      \multirow{5}{*}{Variables}   &  No. of variables, variance of the density of the variables \\
         & Average/median/variance of the ranges of the variables  \\
         & Average/variance of the no. of appearances of each variable\\
         & Pct. of variables not present in any monomial with degree greater than one \\
         & Pct. of variables not present in any monomial with degree greater than two \\
        \midrule
      Constraints &  No. of constraints, Pct. of equality/linear/quadratic constraints\\
    \midrule
    \multirow{3}{*}{Monomials}   &  No. of monomials\\
      & Pct. of linear/quadratic monomials, Pct. of linear/quadratic RLT variables\\
      & Average pct. of monomials in each constraint and in the objective function\\
      \midrule
      Coefficients   &  Average/variance of the coefficients\\
     \midrule
     \multirow{3}{*}{Other}   &  Degree and density of \ref{eq:PO}\\
     & No. of variables divided by no. of constraints/degree \\
     & No. of RLT variables/monomials divided by no. of constraints\\
     \midrule
    Graphs  & Density, modularity, treewidth, and transitivity of VIG and CMIG\\
    \bottomrule
    \end{tabular}}
\caption{Features used for the learning.}
\label{table:features}
\end{table}

For the learning we rely on quantile regression models, since, as discussed in \cite{ghaddar2023}, ``the presence of outliers and of an asymmetric behavior of $\NLBpace$ (negative skewness) makes quantile regression more suitable than conventional regression models (based on the conditional mean)''. Random forests \cite{Breiman2001} are ensemble methods that aggregate multiple decision trees into a single prediction, while quantile regression forests \cite{Meinshausen2006} generalize them by estimating the conditional distribution of the response variable using all observations in each leaf rather than their average. A notable advantage is that results can be reported for the entire dataset via out-of-bag predictions, without relying on explicit training and test splits. The statistical analysis is conducted in~\texttt{R} \cite{Rlang} with the \texttt{ranger} library~\cite{rangerR}.


\section{Main computational experiments}\label{sec:computational}

\subsection{Continuous instances without bound tightening}

In this section we compare the performance of different configurations of the RLT-based algorithm RAPOSa, all of them using the $J$-set filtering strategy. The focus is on the differences in performance between approaches \RLTlooseB and \RLTtightB. Further, in order to assess whether the differences between these two approaches might be dependent on the auxiliary linear solver, all instances are solved with the linear solvers available from Google OR-Tools \cite{ortools}. More precisely, we use \solver{Gurobi} \cite{gurobi}, \solver{Clp} \cite{clp}, \solver{CPLEX} \cite{cplex}, Google's proprietary LP solver \solver{Glop}, and \solver{Xpress} \cite{xpress}. Section~\ref{app:additional} presents additional experiments designed to assess the potential interactions of \RLTlooseB and \RLTtightB with bound tightening and the presence of integer variables.

\begin{table}[!htbp]
\centering
\footnotesize
\renewcommand{\arraystretch}{0.88}
\begin{tabular}{|l|l|r|r|r|r|r|r|}
 \hline
 \multicolumn{2}{|c|}{} & \multicolumn{2}{c|}{DS} & \multicolumn{2}{c|}{MINLPLib} &\multicolumn{2}{c|}{QPLIB} \\
\cline{3-8}
 \multicolumn{2}{|c|}{} & Instances & Variation (\%) & Instances & Variation (\%)& Instances & Variation (\%) \\
 \hline
\multirow{6}{*}{\solver{Gurobi}} & Unsolved & 180 & $+$30.77  & 186 & \textbf{$-$3.23}   & 66 & \phantom{$+$}0.00           \\
 & Gap & 17 & $+$46.91  & 49 & \textbf{$-$27.06}  & 56 & \textbf{$-$1.79}                     \\
 & Time & 110 & $+$72.74  & 47 & $+$14.19   & 3 & \textbf{$-$0.29}                    \\
 & Pace & 123 & $+$63.24  & 105 & \textbf{$-$10.12} & 65 & \textbf{$-$2.06}                     \\
 & Nodes &163 & \textbf{$-$0.29}  & 122 & \textbf{$-$2.08}   & 4& $+$0.22                      \\
 & L-Time & 163& $+$178.77  & 122& $+$607.98  & 4& $+$54.96                   \\
\hline
\multirow{6}{*}{\solver{Clp}} & Unsolved & 177& \textbf{$-$13.89}  & 187  & \textbf{$-$4.92}   & 66 & \phantom{$+$}0.00          \\
 & Gap &37 & \textbf{$-$60.92}  & 45& \textbf{$-$15.08}  &55 & \textbf{$-$1.98}                      \\
 & Time &85 & \textbf{$-$48.52} & 33 & \textbf{$-$5.80}  &  2& \textbf{$-$8.12}                      \\
 & Pace & 115& \textbf{$-$40.15}  & 91& \textbf{$-$6.30}  &  65& \textbf{$-$0.50}                      \\
 & Nodes & 140& $+$9.93   & 126 & \textbf{$-$0.76}  & 3 & $+$0.29                     \\
 & L-Time & 140& \textbf{$-$45.53}  &126 & \textbf{$-$4.67}  & 3 & \textbf{$-$13.68}                     \\
\hline
\multirow{6}{*}{\solver{Cplex}} & Unsolved & 179& \phantom{$+$}0.00 &  184  & $+$1.79  &  63& \phantom{$+$}0.00 \\
 & Gap & 14& $+$13.91   & 40& $+$0.52   &  52& \textbf{$-$1.39}                    \\
 & Time & 100& $+$11.76   & 31& $+$1.31   &  3& \textbf{$-$3.91}                    \\
 & Pace & 114& $+$10.32  & 87 & \textbf{$-$0.05}  & 62& \textbf{$-$2.58}                   \\
 & Nodes & 165& $+$0.29   & 127 & \textbf{$-$3.77}  & 4& \textbf{$-$2.39}                     \\
 & L-Time & 165& $+$7.57   &  127& \textbf{$-$3.95}  &  4& \textbf{$-$8.82}                     \\
\hline
\multirow{6}{*}{\solver{Glop}} & Unsolved & 177& \textbf{$-$33.33}   & 186 & $+$8.33   & 67& \textbf{$-$1.56}                    \\
 & Gap & 37& \textbf{$-$74.43}  & 47& $+$65.95  &  56& $+$8.23                     \\
 & Time & 92& \textbf{$-$51.70}  &36 & $+$64.90  & 3 & \textbf{$-$16.17}                    \\
 & Pace &115& \textbf{$-$45.82}  & 96& $+$123.90 & 66 & \textbf{$-$100.00}                   \\
 & Nodes &140 & $+$10.81   & 121& \textbf{$-$5.21}  & 3 & $+$3.44                     \\
 & L-Time &140 & \textbf{$-$49.00}  & 121& \textbf{$-$17.35}  & 3& \textbf{$-$12.00}                     \\
\hline
\multirow{6}{*}{\solver{Xpress}} & Unsolved &180 & \textbf{$-$19.15}   & 185 & \textbf{$-$1.32}    &67 & \phantom{$+$}0.00          \\
 & Gap & 51& \textbf{$-$42.34} &  61& \textbf{$-$14.10}  & 58& $+$1.93                      \\
 & Time & 127& $+$2.29   & 57 & $+$30.11   & 2& \textbf{$-$20.25}                   \\
 & Pace & 161& $+$1.06   &  132& \textbf{$-$18.57}  & 67& \textbf{$-$0.30}                     \\
 & Nodes & 129& $+$0.29   &  109& $+$8.52    &2 & \textbf{$-$24.14}                   \\
 & L-Time & 129& $+$24.30   & 109& $+$14.14   &2 & $+$5.40                    \\
\hline 
\end{tabular}
\caption{Results without bound tightening in continuous instances.}
\label{tab:withoutbt}
\end{table}

Table~\ref{tab:withoutbt} summarizes the results for the executions of \RLTlooseB and \RLTtightB for the continuous instances and without bound tightening. The results are separated by test set and by linear solver. We report the results using \RLTlooseB as the reference approach. Accordingly, the numbers in the table represent the relative change in performance when switching from \RLTlooseB to \RLTtightB: positive values indicate a performance deterioration for \RLTtightB, whereas negative values (shown in bold) indicate that \RLTtightB outperformed \RLTlooseB. We have chosen to report the results in this manner to avoid direct comparisons between linear solvers, which could distract from the primary goal of the study. To this end, the results in Table~\ref{tab:withoutbt} just allow to compare the relative impact of the tightening of the RLT bounds for the different solvers, but no comparison of their absolute performance is possible. Finally, column ``instances'' contains the number of instances used to compute each performance measure in each test set, after applying the exclusions described in Section~\ref{sec:testing}. In particular, the number of instances in row ``Unsolved'' is the total number of instances used for that test set (some instances that turned out to be problematic for some solver were removed for that solver).

In short, Table~\ref{tab:withoutbt} shows that anything can happen: nothing resembling a consistent pattern emerges. With \solver{Gurobi}, \RLTlooseB is preferable for DS, but results are somewhat mixed for MINLPLib and QPLIB. Yet, for both test sets, \RLTtightB increases the solve time of the linear relaxation quite significantly. The situation practically gets reversed with \solver{Glop}, for which \RLTtightB clearly outperforms \RLTlooseB for DS, but gets clearly outperformed in MINLP. The results for \solver{Clp}, \solver{Xpress}, and \solver{CPLEX} just confirm that the impact of tightening the bounds of the RLT variables can go in any direction which, moreover, may depend on the test set and on the specific performance metric. The only additional insight may be that \solver{Clp} seems to perform better using \RLTtightB approach in the three test sets.

These results suggest that the way in which variable bounds are handled internally by the different solvers has a big impact on performance. Yet, we have not been able to find any pattern that may explain the disparate performance results for \RLTlooseB and \RLTtightB in Table~\ref{tab:withoutbt}.

\subsection{Exploring underlying drivers}

The first question one may ask by looking at the results in Table~\ref{tab:withoutbt} is: if the linear relaxations of \RLTlooseB and \RLTtightB are mathematically equivalent optimization problems (Theorem~\ref{theor:RLTbounds}), why do the solved instances exhibit different node counts? To investigate this, we examined the branch-and-bound trees of both formulations and found that they diverge very early. The underlying reason is the high multiplicity of optimal solutions in the RLT relaxations: different optimal solutions induce different violations of the RLT-defining identities in Eq.~\eqref{eq:RLTidentity}, which in turn lead to distinct branching decisions. Although in principle this separation should not systematically advantage either configuration, the observed fluctuations in performance across metrics and solvers raise questions about the presumed neutrality of this effect.

In order to gain some understanding of the mechanisms driving these fluctuations, in this section we develop a deeper analysis of \solver{Gurobi}'s results. The reason for this choice is that \solver{Gurobi} is the default linear solver in  \solver{RAPOSa} and, hence, it's interactions with RAPOSa's core RLT-based algorithm are better understood. Moreover,  we restrict attention to instances that have been solved by both \RLTlooseB and \RLTtightB, which enables cleaner comparisons. Given that both approaches solve only 4 instances in QPLIB library, we exclude it from the analysis in this section.

Now, recall that metrics Nodes and L-Time are computed on instances solved by both approaches within the time limit. The results in Table~\ref{tab:withoutbt} for Nodes suggest that the size of the resulting tree is not as affected by the choice of \RLTlooseB or \RLTtightB as other metrics. In particular, as already discussed, \RLTtightB seems to deteriorate the performance of \solver{Gurobi} at solving the LP relaxations. In Figure~\ref{fig:rel_nodes_times_gurobi} we disaggregate the results behind Nodes and L-Time, presenting density plots of the relative differences in these metrics instance by instance. More precisely, for each instance we compute the following relative differences:\footnote{For the sake of completeness, in Appendix~\ref{app:additional} we present similar plots for the different solvers, and also for the additional experiments with bound tightening enabled in \solver{RAPOSa} and those with mixed-integer instances.}
\begin{itemize}
    \item Number of nodes: $\frac{\text{``nodes \RLTtightB''}\,-\,\text{``nodes \RLTlooseB''}}{\text{``nodes \RLTlooseB''}}$.
    \item Mean linear solve time: $\frac{\text{``mean-lin-time \RLTtightB''}\,-\,\text{``mean-lin-time \RLTlooseB''}}{\text{``mean-lin-time \RLTlooseB''}}$
\end{itemize}

\begin{figure}[!htbp]
    \centering
    \begin{subfigure}{0.48\textwidth}
        \includegraphics[width=\textwidth]{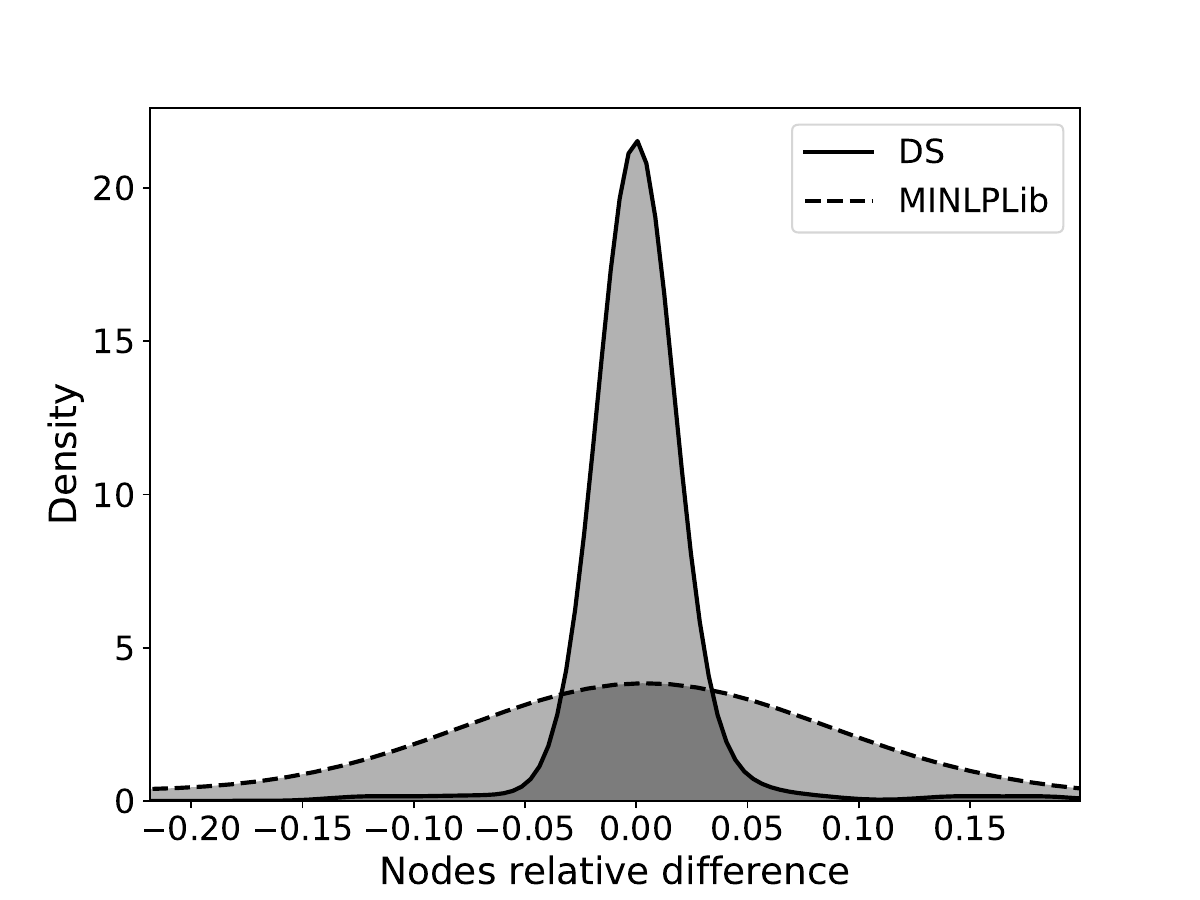}
        \caption{Relative differences in the number of nodes.}
        \label{fig:rel_nodes_gurobi}
    \end{subfigure}
    \begin{subfigure}{0.48\textwidth}
        \includegraphics[width=\textwidth]{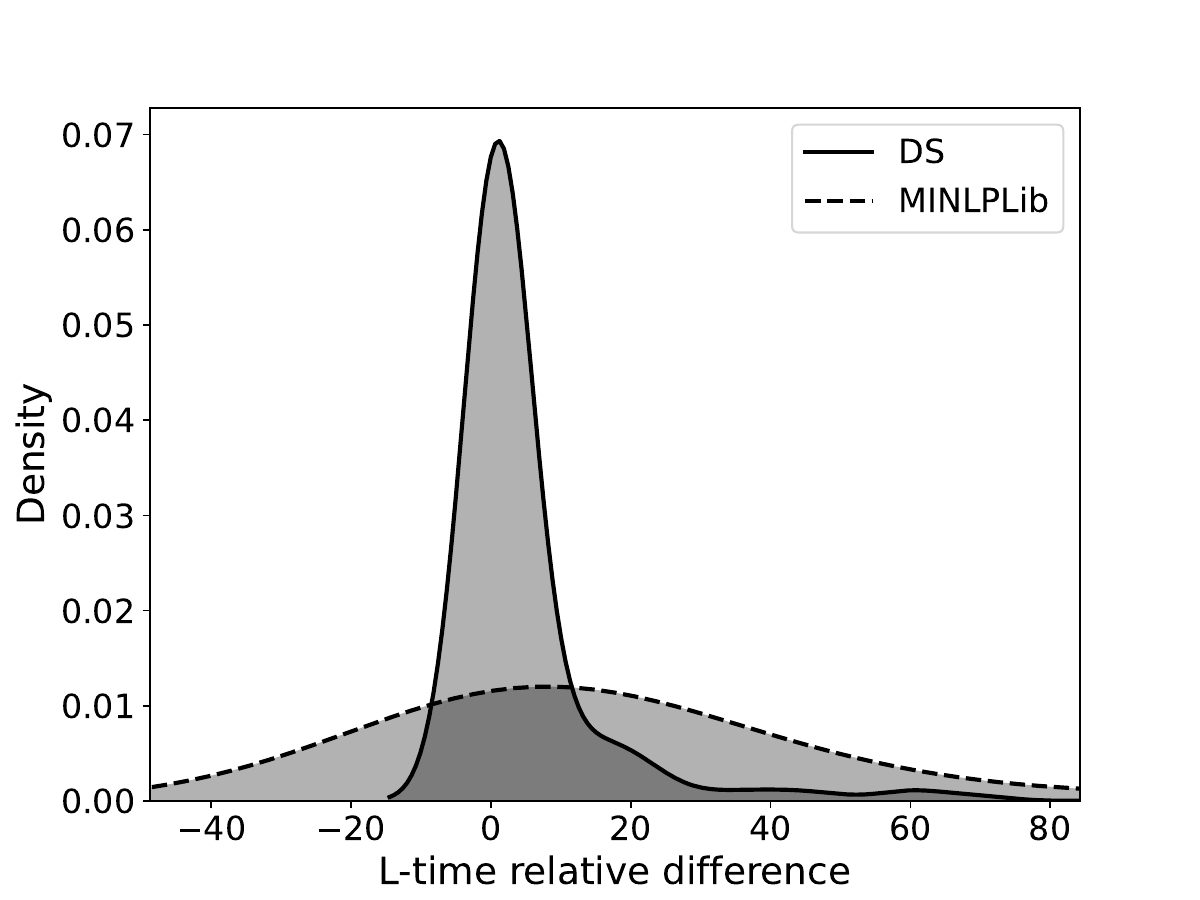}
        \caption{Relative differences in mean linear solve times.}
        \label{fig:rel_lintimes_gurobi}
    \end{subfigure}    
    \caption{Densities of relative differences between \RLTlooseB and \RLTtightB using \solver{Gurobi}.}
    \label{fig:rel_nodes_times_gurobi}
\end{figure}

Figure~\ref{fig:rel_nodes_times_gurobi}(\subref{fig:rel_nodes_gurobi}) shows that not only the mean values represented in Nodes are close to zero, but also the variability is quite small, suggesting that neither \RLTlooseB nor \RLTtightB tend to systematically deliver smaller trees. This observation is consistent across solvers and also applies to the additional computational experiments in Appendix~\ref{app:additional}, where Nodes is the metric that exhibits less variability.

Figure~\ref{fig:rel_nodes_times_gurobi}(\subref{fig:rel_lintimes_gurobi}) tells a different story. It shows that, when looking at the average values of the solve times of the linear relaxations associated to each instance, there is much more variability than before in the relative differences between \RLTlooseB nor \RLTtightB and, more importantly, there is an important asymmetry towards positive values, confirming that linear solve times tend to be larger for \RLTtightB as clearly suggested by the values $+178.77$ and $+607.98$ in Table~\ref{tab:withoutbt}. To the best of our understanding, the two most natural explanations for this behavior are:
\begin{itemize}
    \item Explicitly including the bounds on the RLT variables degrades \solver{Gurobi}'s performance.
    \item The branching decisions taken \RLTtightB lead to more difficult LP relaxations.
\end{itemize}

Note that the mean solve times reported in Figure~\ref{fig:rel_nodes_times_gurobi}(\subref{fig:rel_lintimes_gurobi}) are based on different trees, and therefore it is not straightforward to determine which explanation is more plausible. To address this, we designed the following experiment: we take the branch-and-bound tree produced by \RLTlooseB and, at each node, compute the solve times of the linear relaxations with and without the bounds on the RLT variables. In this way, the resulting differences, reported in Figure~\ref{fig:rel_times_fixed_tree}, are obtained by solving exactly the same subproblems under both approaches. Differently from what we observed in Figure~\ref{fig:rel_nodes_times_gurobi}(\subref{fig:rel_lintimes_gurobi}), the distribution is now skewed towards negative values, \emph{i.e.}, the explicit inclusion of the RLT bounds yields faster solve times. Specifically, the L-Time values in Table~\ref{tab:withoutbt}, which were $+178.77\%$ in DS and $+607.98\%$ in MINLPLib, become $-38.56\%$ in DS and $-70.35\%$ in MINLPLib. More importantly, the values along the $x$-axis are now confined to the range $[-1.5, 1.5]$, whereas before they spanned from $-40$ to $+80$.

\begin{figure}[!htbp]
    \centering
    \includegraphics[width=0.48\textwidth]{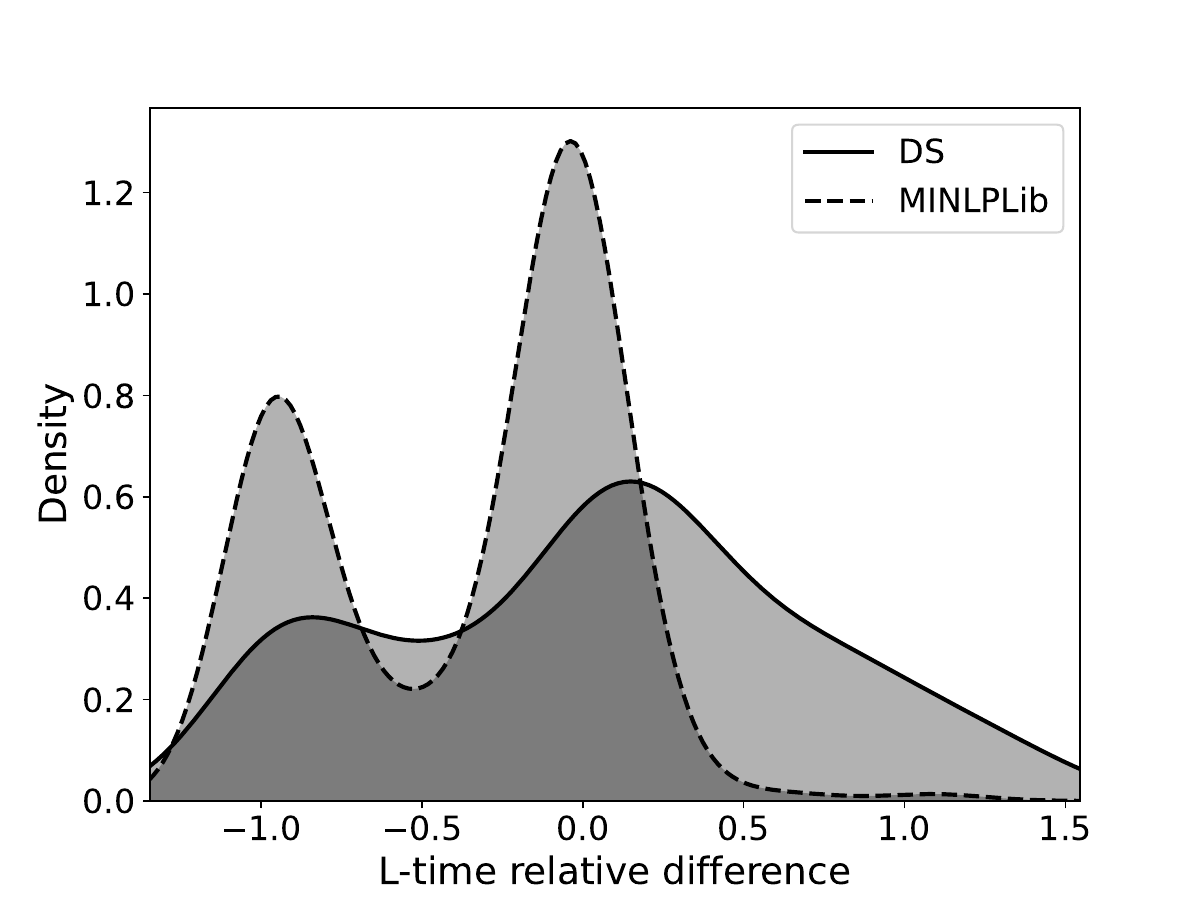}
    \caption{Relative differences in linear times on a fixed tree.}
    \label{fig:rel_times_fixed_tree}
\end{figure}

These findings strongly support the explanation that the branching decisions taken by \RLTtightB lead to more difficult LP relaxations. In particular, when Gurobi is used as the linear solver, we observe that, although \RLTlooseB and \RLTtightB generate branch-and-bound trees of comparable size, \RLTtightB tends to lead to linear relaxations that are harder to solve for Gurobi. 

Next, one may ask whether this phenomenon persists across solvers or is specific to \solver{Gurobi}. The results in Table~\ref{tab:withoutbt}, together with those in Appendix~\ref{app:additional}, indicate that its presence varies across solvers and test sets. Hence, there is nothing inherent to the \RLTtightB configuration that systematically biases branching toward more difficult LP subproblems.

Therefore, we may ultimately have to settle for the view that these results are driven by the combinatorial nature of branch-and-bound, rather than by any structural property of the formulations. Yet, even under this assumption, the impact could be concentrated in certain subclasses of problems, or the combinatorial nature could affect the performance of \RLTlooseB and \RLTtightB differently in ways that are hidden yet somewhat systematic. However, even after an in-depth analysis of the disaggregate results, we found no evidence supporting this possibility. In the following section, we pursue this avenue further by exploring whether statistical learning techniques can provide evidence of underlying patterns.

\subsection{Potential for learning}\label{sec:learning_res}

Since no clear pattern could be identified behind the results of the preceding section, we now turn to statistical learning techniques, aiming to assess whether or not the observed variability can be attributed solely to the combinatorial randomness of branch-and-bound algorithms. If not, this would open the possibility of predicting, for a given instance, which approach, \RLTlooseB or \RLTtightB, is likely to perform better. To this end, we build upon the ML framework described in Section~\ref{sec:learning_meth}. To maintain focus, we start restricting attention once more to \solver{Gurobi}.

Table~\ref{tab:learningnobt} presents the results of the learning which, importantly,  is done jointly for the three test sets, although we disaggregate them in the results. As explained in Section~\ref{sec:learning_meth}, the machine learning model is trained to predict $\NLBpace$, a normalized version of $\LBpace$. For the results reported in Table~\ref{tab:learningnobt}, we therefore use as reference the configuration that showed superior overall performance in terms of Pace, which is \RLTlooseB, owing to its overwhelming advantage on the DS test set. The ML columns contain the improvement in performance obtained by the learning model, while the Oracle columns represent the performance achieved when, for each instance, we choose the best configuration according to Pace, \emph{i.e.}, an ideal configuration that corresponds to having an oracle indicating the best approach for each instance. We use Oracle as the reference against which to assess the quality of the learning.

\begin{table}[!htbp]
\centering
\footnotesize
\begin{tabular}{|l|r|r|r|r|r|r|}
\hline
                                   & \multicolumn{2}{c|}{DS}    & \multicolumn{2}{c|}{MINLPLib} & \multicolumn{2}{c|}{QPLIB} \\
\cline{2-7}
                                   & ML (\%) & Oracle (\%)      & ML (\%) & Oracle (\%)      & ML (\%) & Oracle (\%)       \\
\hline
                        Unsolved   & \phantom{$+$}0.00   & \phantom{$+$}0.00     & \textbf{$-$4.84}    & \textbf{$-$6.45}             & \phantom{$+$}0.00   & \phantom{$+$}0.00    \\
                        Gap        & \phantom{$+$}0.00   & \phantom{$+$}0.00     & \textbf{$-$27.76}   & \textbf{$-$28.48}            &  \textbf{$-$2.00}            & \textbf{$-$2.47}              \\
                        Time       & \phantom{$+$}0.00   & \textbf{$-$0.22}      &  \textbf{$-$8.61}   & \textbf{$-$14.23}            &  \textbf{$-$16.06}           & \textbf{$-$16.06}             \\
                        Pace       & \phantom{$+$}0.00   & \textbf{$-$0.20}      &  \textbf{$-$27.34}  & \textbf{$-$30.03}            &  \textbf{$-$3.04}            &  \textbf{$-$3.99}             \\
                        Nodes      & \phantom{$+$}0.00   & \phantom{$+$}0.00     &  \textbf{$-$2.79}   & \textbf{$-$5.08}             &  \textbf{$-$0.12}            &  \textbf{$-$0.12}             \\
                        L-Time     & \phantom{$+$}0.00   & \textbf{$-$0.16}      &  $+$0.56   & $+$1.25            &  \textbf{$-$13.43}           &  \textbf{$-$13.43}            \\
\hline 
\end{tabular}
\caption{Learning without bound tightening in continuous instances using \solver{Gurobi}.}
\label{tab:learningnobt}
\end{table}

Looking at Table~\ref{tab:learningnobt}, we see that in DS the Oracle barely improves upon \RLTlooseB, which indicates that \RLTlooseB is the best choice in almost every problem. In this test set, ML always selects \RLTlooseB, which explains why the ML column is 0 for all metrics. For MINLPLib, in terms of Pace, Oracle improves 30.03\%, while the learning achieves 27.34\%, which is very close to the optimal improvement. Similarly, in QPLIB, Oracle improves 3.99\%, and ML 3.04\%.

These results provide strong evidence that the variability we see in Table~\ref{tab:withoutbt} is not simply due to randomness, but that there is some underlying pattern that can, in fact, be learned. Hopefully, further investigation will provide insights into the nature of these patterns. In Table~\ref{tab:learningsummarycont} we present, for each linear solver, how much is there to be learned (Oracle) and how far the learning gets (ML). In the parentheses we put the best performing version for that solver according to Pace.


\begin{table}[!htbp]
\centering
\footnotesize
\begin{tabular}{|l|r|r|r|r|r|}
\hline
& \solver{Gurobi} & \solver{Clp}& \solver{Cplex}& \solver{Glop} & \solver{Xpress} \\[-0.1cm]
& \tiny{(\RLTlooseB)} & \tiny{(\RLTtightB)} & \tiny{(\RLTlooseB)} & \tiny{(\RLTtightB)} & \tiny{(\RLTtightB)} \\
\hline
ML (\%)	& \textbf{-11.42} & +0.38 &	\textbf{-1.46} &	\textbf{-32.00} &	\textbf{-3.88} \\
Oracle (\%)	    & \textbf{-12.88} &	\textbf{-2.74} &	\textbf{-2.48} &	\textbf{-35.17} &	\textbf{-11.15} \\
\% Learned	& 88.66 &	--- &	58.87 &	90.99 &	34.80 \\
\hline 
\end{tabular}
\caption{Learning summary in continuous instances.}
\label{tab:learningsummarycont}
\end{table}

It is worth noting the excellent results achieved by the learning model for both \solver{Gurobi} and \solver{Glop}, where ML’s performance comes very close to Oracle. Though to a lesser extent, there are also meaningful improvements for \solver{Cplex} and \solver{Xpress}. On the other hand, for \solver{Clp} the learning fails, as its performance according to Pace falls slightly behind that of \RLTtightB. Yet, this outcome is not entirely unexpected, since the scope for improvement in \solver{Clp} is limited (Oracle achieved only a 2.74\% gain) making learning considerably harder (the same could be said for \solver{Cplex}, where ML was still successful).

Overall, we believe that these results support the view that randomness alone cannot explain the disparate outcomes observed in Table~\ref{tab:withoutbt}, although we cannot rule out the possibility that these ``statistical patterns'' are spurious. Thus, to further substantiate our conclusion, we perform one final computational experiment. We have considered two modifications to the default configuration of \solver{RAPOSa}, whose impact on performance should be entirely random, thereby severely limiting the potential for learning. The first modification reverses the order of the variables prior to solving the problem, while the second reverses the order of the constraints. By construction, both modifications define optimization problems that are mathematically equivalent to the original problem, analogously to what happened when switching from \RLTlooseB to \RLTtightB.

\begin{table}[!htbp]
\centering
\footnotesize
\begin{tabular}{|l|ccc|ccc|}
\hline
       & \multicolumn{3}{c|}{Rev. Variables} & \multicolumn{3}{c|}{Rev. Constraints} \\
\hline
       & DS & MINLPLib & QPLIB & DS & MINLPLib & QPLIB \\
\hline
Unsolved & \phantom{$+$}0.00 & \phantom{$+$}0.00 & \phantom{$+$}0.00 & \phantom{$+$}0.00 & \phantom{$+$}0.00 & \textbf{$-$1.64} \\
Gap      & $+$0.88 & $+$10.59 & \textbf{$-$2.34} & $+$0.31 & $+$4.45 & \textbf{$-$0.94} \\
Time     & \textbf{$-$4.28} & \textbf{$-$10.48} & \textbf{$-$1.49} & \textbf{$-$3.25} & \textbf{$-$6.60} & \textbf{$-$1.69} \\
Pace     & \textbf{$-$3.65} & \textbf{$-$5.43} & $+$0.51 & \textbf{$-$7.16} & \textbf{$-$1.75} & $+$0.19 \\
Nodes    & \textbf{$-$0.98} & \textbf{$-$2.40} & $+$3.77 & $+$0.02 & \textbf{$-$3.74} & $+$5.11 \\
L-Time   & \textbf{$-$0.74} & \textbf{$-$4.63} & \textbf{$-$2.80} & \textbf{$-$0.56} & \textbf{$-$4.81} & $+$1.04 \\
\hline
\end{tabular}
\caption{Results of reversing the order of the variables and constraints.}
\label{tab:reversevarscons}
\end{table}

In Table~\ref{tab:reversevarscons}, we present results analogous to those in Table~\ref{tab:withoutbt}, showing the variations in performance resulting from reversing the order of variables and constraints (under the \RLTlooseB configuration). A first observation is that, as expected, depending on the library and the performance metric, the modifications can lead to either improvements or degradations. An important distinction from the results in Table~\ref{tab:withoutbt} is, however, that the general variability now appears to be smaller, which again hints at the impact of \RLTtightB not being entirely random. Next, we present the learning results associated to these executions.

\begin{table}[!htbp]
\centering
\footnotesize
\begin{tabular}{|l|r|r|r|r|}
\hline
                                   &\multicolumn{2}{c|}{Rev. Variables}    & \multicolumn{2}{c|}{Rev. Constraints}\\
\hline
                                   & ML (\%) & Oracle (\%)       & ML (\%) & Oracle (\%)        \\
\hline
Unsolved & \phantom{$+$}0.00 &	\phantom{$+$}0.00 &		\phantom{$+$}0.00 &	\phantom{$+$}0.00 \\
Gap & \textbf{$-$1.48} &	\textbf{$-$3.07} &	\textbf{$-$0.03} &	\textbf{$-$2.20} \\
Time &	\textbf{$-$0.18} &	\textbf{$-$4.44} & $+$1.85 &	\textbf{$-$2.82} \\
Pace  &	\textbf{$-$0.39} &	\textbf{$-$3.73} &	 $+$1.07 &	\textbf{$-$2.71} \\
Nodes  &	\textbf{$-$0.12} &	\textbf{$-$1.78} &	$+$1.14 &	\textbf{$-$0.59} \\
L-Time &	\textbf{$-$0.11} &	\textbf{$-$0.77} &	\phantom{$+$}0.00 &	\textbf{$-$0.63} \\
\hline 
\end{tabular}
\caption{Learning of reversing the order of the variables and constraints.}
\label{tab:learningreversevarscons}
\end{table}

It turns out that both ``reversed'' configurations are slightly superior to the baseline and are therefore used as the reference for the results in  Table~\ref{tab:learningreversevarscons}. Unlike what we saw in Table~\ref{tab:learningsummarycont}, where the ML column was able to deliver very strong performance for \solver{Gurobi} (close to the value in the Oracle column), this is now far from being the case. The potential for learning is limited, since the Pace improvements of Oracle in Rev. Variables and Rev. Constraints are just 3.73\% and 2.71\%, respectively. More importantly, the effective learning is negligible: 10.5\% for Rev. Variables and there is a significant degradation in performance of ML with respect to the configuration Rev. Constraints. We interpret this as a new piece of evidence supporting the hypothesis that the impact of the explicit inclusion of the bounds on the RLT variables is not random, indicating that there remains scope for additional research into the underlying structural patterns.

\subsubsection*{Interpretability.} We conclude this section by studying one last question within our learning framework. Given the nature of the learning techniques used in our analysis, (quantile) random forests, we can assign importance scores to the features and use them to identify which ones play the most significant role in the predictions, thereby shedding some light on the structural patterns that may underlie the promising results achieved by ML. Recall that our framework requires running one regression for each configuration, \RLTlooseB and \RLTtightB. Furthermore, replicating the same learning analysis for the additional computational experiments reported in Appendix~\ref{app:additional} yields a total of $4 \times 2$ regressions from which to analyze feature importance (with and without bound tightening, continuous and mixed-integer variables, \RLTlooseB and \RLTtightB).

\begin{figure}[!htbp]
    \centering
    \includegraphics[trim=10 0 0 0, width=\textwidth]{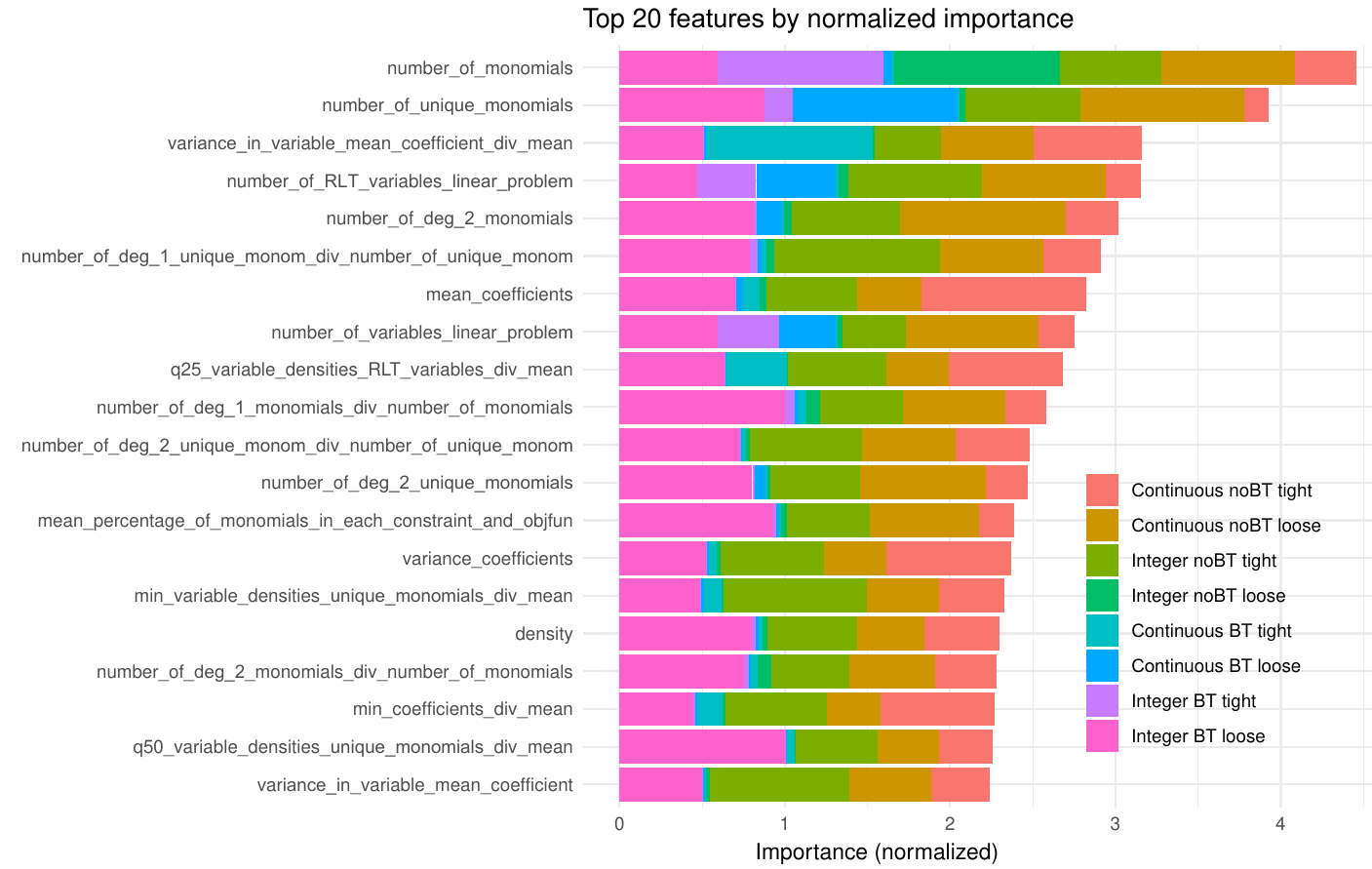}
    \caption{Importance of the features in the learning.}
    \label{fig:features}
\end{figure}

In Figure~\ref{fig:features} we present the features that are most important across the eight regressions, together with their relative importance in each individual regression. The most notable observation is the prominence of absolute and relative features related to the number of monomials and RLT variables.\footnote{Although most names are self-explanatory, some require clarification. \lstinline!number\_of\_unique\_monomials! denotes the count of monomials without repetitions.  \lstinline!number\_of\_RLT\_variables\_linear\_problem! excludes degree-1 variables (it may not coincide with the number of monomials of degree greater than one in the original problem, since bound-factor constraints entail the addition of new RLT variables). \lstinline!number\_of\_deg\_1\_unique\_monom! corresponds to the variables of the original problem that appear in degree-1 monomials.} Nevertheless, features capturing the magnitude of their coefficients and the density of the problem also play a relevant role. The set of features reported in Figure~\ref{fig:features} clearly exhibits a high degree of multicollinearity. This is not problematic for our purposes, since the main goal was to assess the potential for learning, but it does complicate the task of identifying individual effects.

In summary, we believe that future research in this particular direction should focus on more specialized learning frameworks. These frameworks ought to be designed not only for predictive performance, but also for interpretability and insight into underlying patterns. To this end, it would be useful to merge related features and introduce new ones tailored to the setting at hand.

\section{Conclusions}\label{sec:conclusions}

This paper has shown that the explicit inclusion of bounds on auxiliary variables in RLT relaxations, while theoretically redundant, can substantially affect solver performance. Across different auxiliary linear solvers and test sets, we observed improvements and degradations in solve time, relaxation time, optimality gaps, and branch-and-bound tree sizes, with no single pattern emerging. Our analysis suggests that the multiplicity of optimal solutions in the relaxations leads to divergent branching decisions, which in turn can make subproblems easier or harder. Statistical learning experiments further indicate that the observed variability cannot be explained by randomness alone, as predictive models often achieved performance close to an oracle benchmark. These findings suggest, as a promising avenue for future work, the development of learning frameworks tailored not only to performance prediction but also to interpretability, with the aim of uncovering the structural drivers behind the observed variability. 

Finally, the effects documented in this study are unlikely to be limited to RLT relaxations. Similar issues may arise in the lifted formulations behind general global optimization algorithms and solvers. We hope that this work contributes to a deeper understanding of the intricate interplay between combinatorial search, algorithm design, and problem structure, and encourages further research into how subtle modeling choices can have far-reaching computational consequences.

\section*{Acknowledgments}
This work is part of the R\&D projects PID2021-124030NB-C32, PID2025-170122NB-C31, and PID2025-170122NB-C33 funded by ERDF/EU and MICIU/AEI/10.13039/501100011033/. This research was also funded by Grupos de Referencia Competitiva ED431C-2021/24 and ED431C 2024/26 from the Consellería de Cultura, Educación e Universidades, Xunta de Galicia. Brais González-Rodríguez acknowledges the support from MICIU, through grant BG23/00155. Ignacio Gómez-Casares acknowledges the support from the Spanish Ministry of Education through FPU grant 20/01555. 

DeepL, Gemini, and GPT were used during the preparation of this manuscript for the sole purpose of language improvement.

\paragraph{Conflict of interest.} The authors declare no competing interests.

\bibliographystyle{apalike}
\bibliography{references}

\appendix

\section{Complementary computational experiments}\label{app:additional}

This Appendix contains the results associated with a complementary set of experiments, whose main goal was to show that the results discussed in Section~\ref{sec:computational} are not specific to the executions reported there, but that similar results can be achieved in similar experiments. 

In particular, in this Appendix we report results on the following experiments:
\begin{itemize}
    \item \textbf{Continuous instances without bound tightening.} We complement the results reported in Section~\ref{sec:computational} with detailed densities of Nodes and L-Time for solvers other than \solver{Gurobi}.
    \item \textbf{Continuous instances with bound tightening.} \solver{RAPOSa} is run with both OBBT and FBBT domain reduction techniques enabled, as described in Section~\ref{subsec:bt}.
    \item \textbf{Mixed-integer instances without bound tightening.} The continuous instances of the test sets are now replaced by mixed-integer ones, defined as described in Section~\ref{sec:testing}. In order to solve these instances, the general RLT-based algorithm in Figure~\ref{fig:RLTalg} is adapted to handle integer variables as discussed in \cite{raposainteger}, where the LP relaxations $LB^k$ are replaced by their MILP counterparts.
    \item \textbf{Mixed-integer instances with bound tightening.} \solver{RAPOSa} is run, on the mixed-integer instances, with both OBBT and FBBT domain reduction techniques enabled.
\end{itemize}

Overall, the results look significantly different from those in Table~\ref{tab:withoutbt} but, at the same time, they are qualitatively identical, since anything can happen in terms of the relative performance of \RLTlooseB and \RLTtightB. 

\subsection{Continuous instances without bound tightening}

\subsubsection{Main results}

Already reported in the body of the paper in Table~\ref{tab:withoutbt}.

\subsubsection{Densities for Nodes and L-Time. Results for \solver{Gurobi} solver}

Already reported in the body of the paper in Figure~\ref{fig:rel_nodes_times_gurobi}.

\subsubsection{Densities for Nodes and L-Time. Results for the rest of the solvers}

\begin{figure}[H]
    \centering
    \begin{subfigure}{0.4\textwidth}
        \includegraphics[trim=0 0 0 28, clip, width=\textwidth]{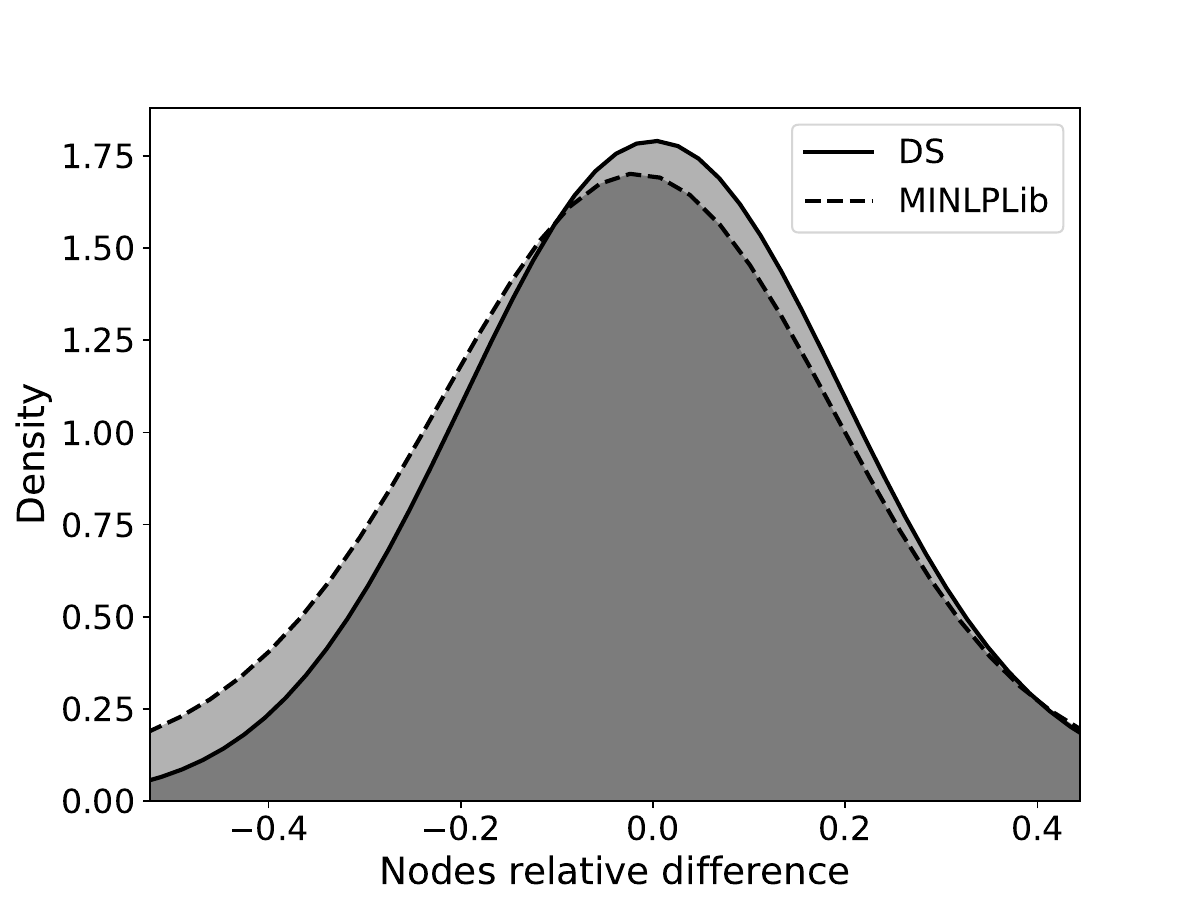}
        \caption{Nodes, \solver{Clp}.}
    \end{subfigure}    
    \begin{subfigure}{0.4\textwidth}
        \includegraphics[trim=0 0 0 28, clip, width=\textwidth]{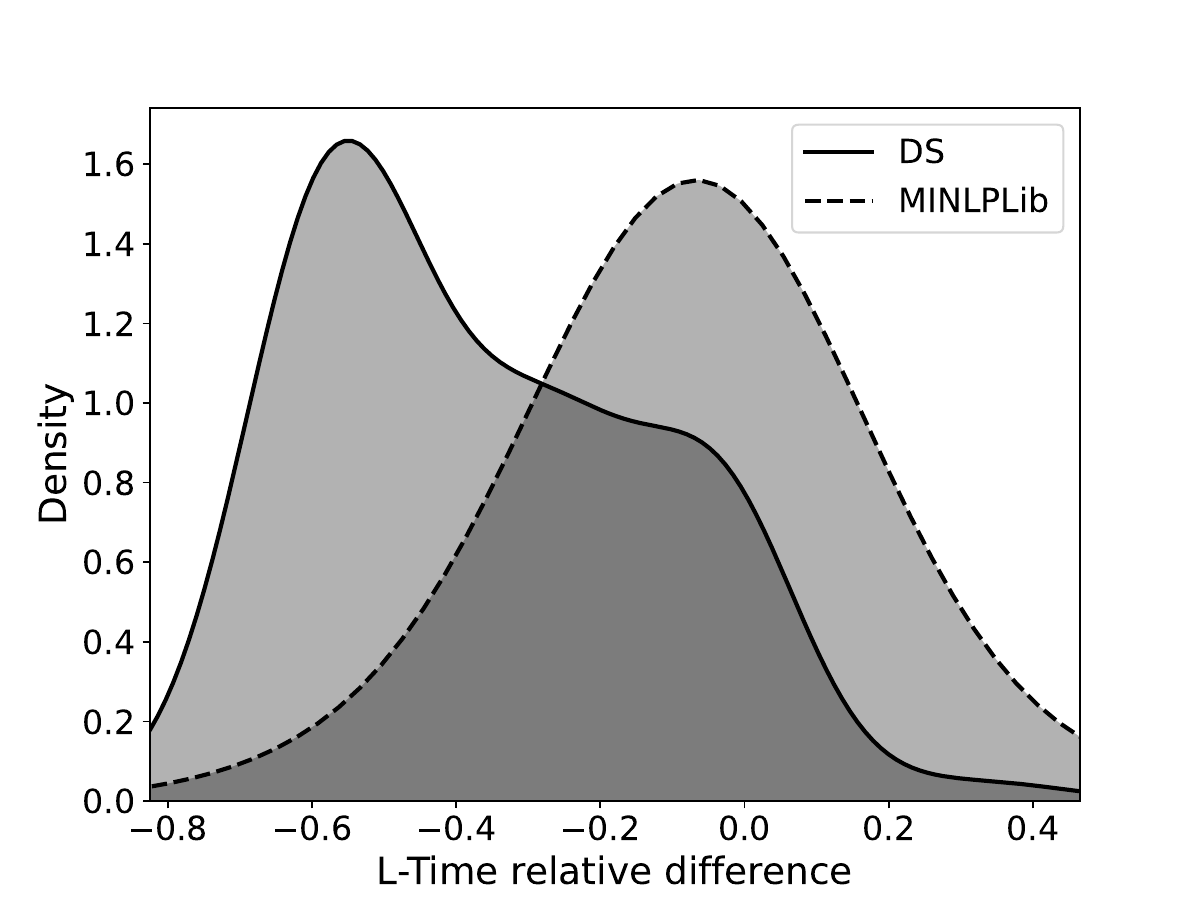}
        \caption{L-Time, \solver{Clp}.}
    \end{subfigure}    
    
    \begin{subfigure}{0.4\textwidth}
        \includegraphics[trim=0 0 0 28, clip, width=\textwidth]{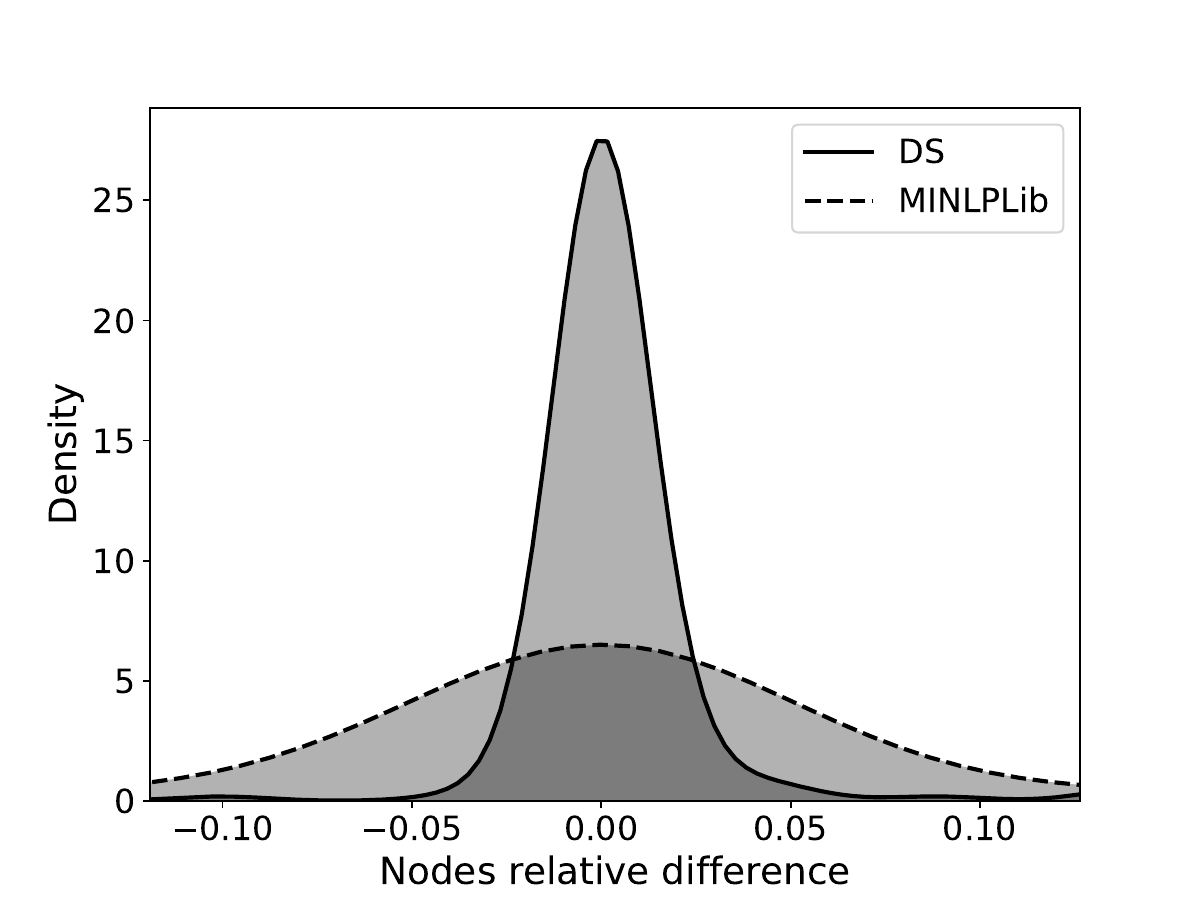}
        \caption{Nodes, \solver{Cplex}.}
    \end{subfigure}    
    \begin{subfigure}{0.4\textwidth}
        \includegraphics[trim=0 0 0 28, clip, width=\textwidth]{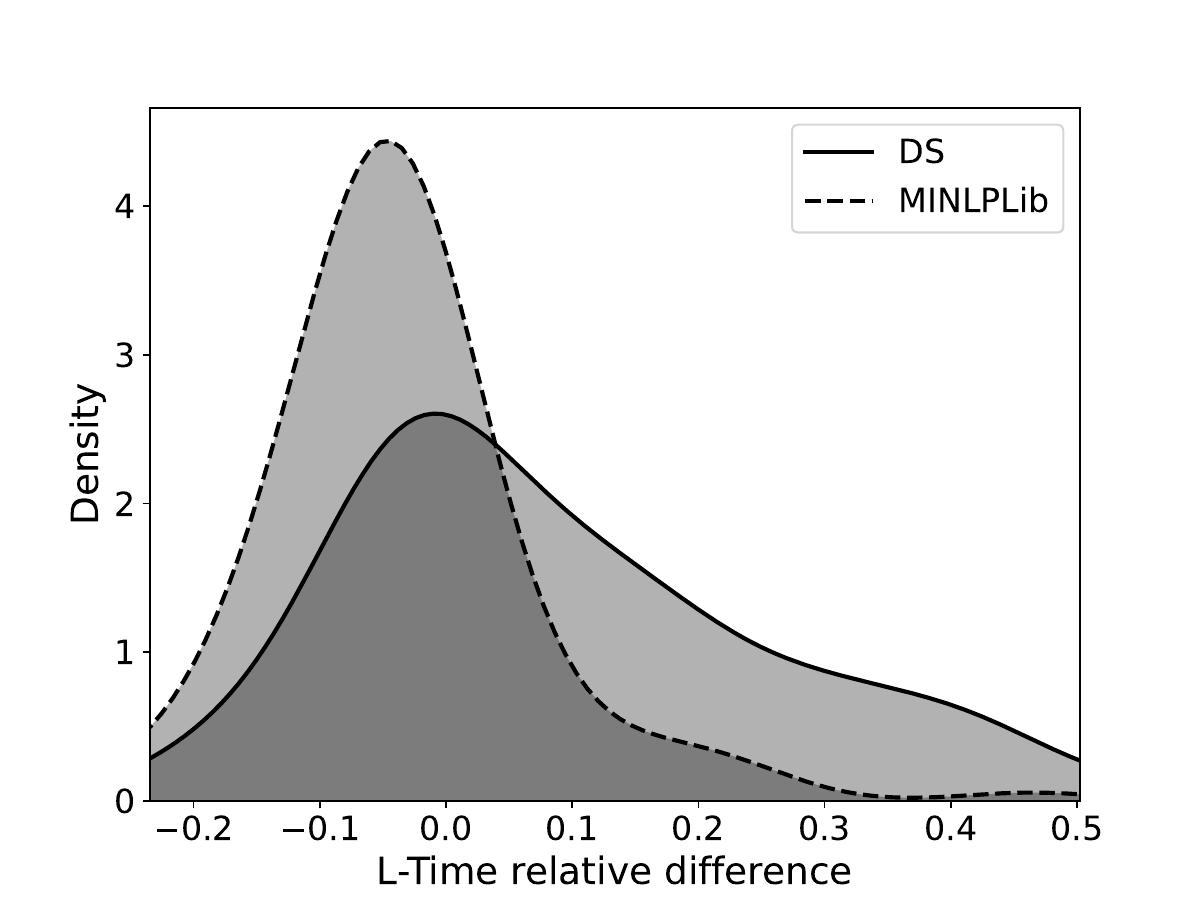}
        \caption{L-Time, \solver{Cplex}.}
    \end{subfigure}    
    
    \begin{subfigure}{0.4\textwidth}
        \includegraphics[trim=0 0 0 28, clip, width=\textwidth]{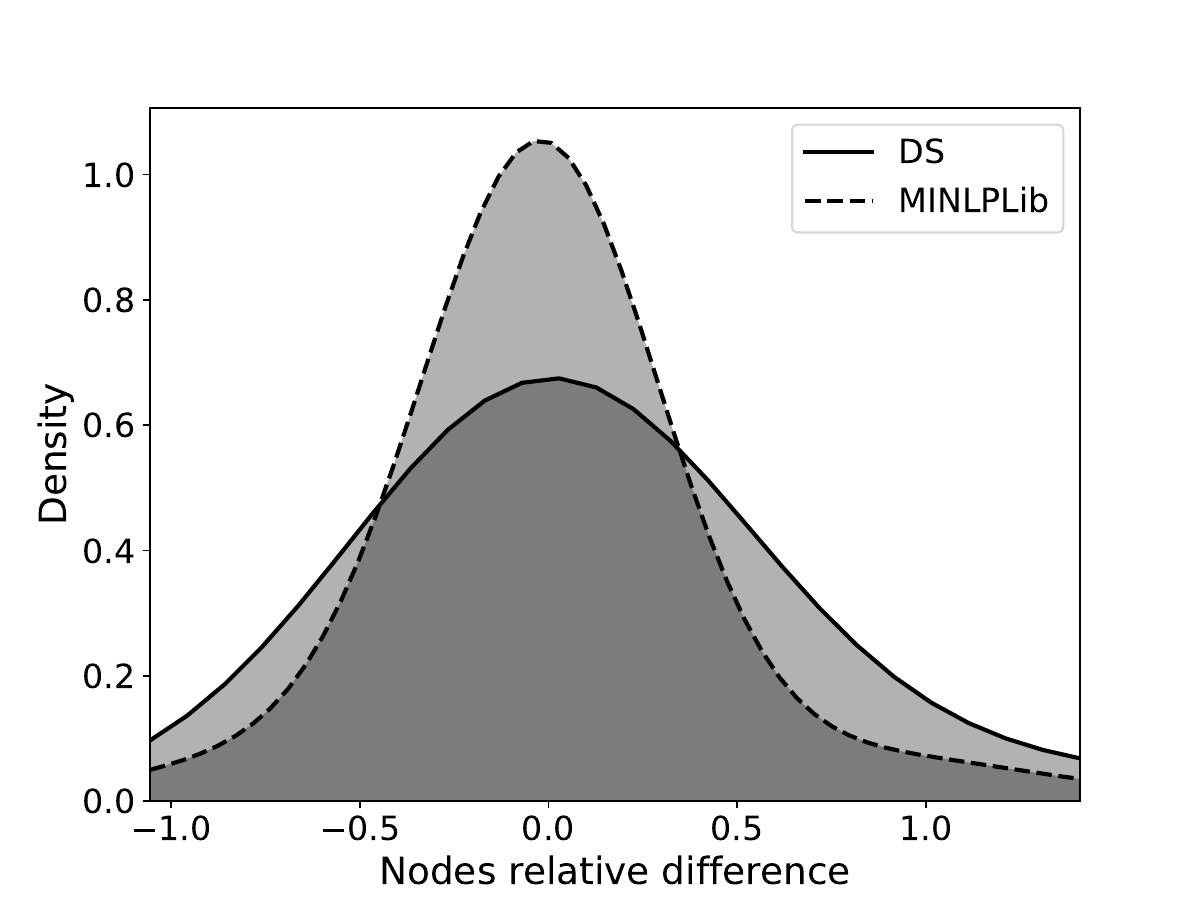}
        \caption{Nodes, \solver{Glop}.}
    \end{subfigure}    
    \begin{subfigure}{0.4\textwidth}
        \includegraphics[trim=0 0 0 28, clip, width=\textwidth]{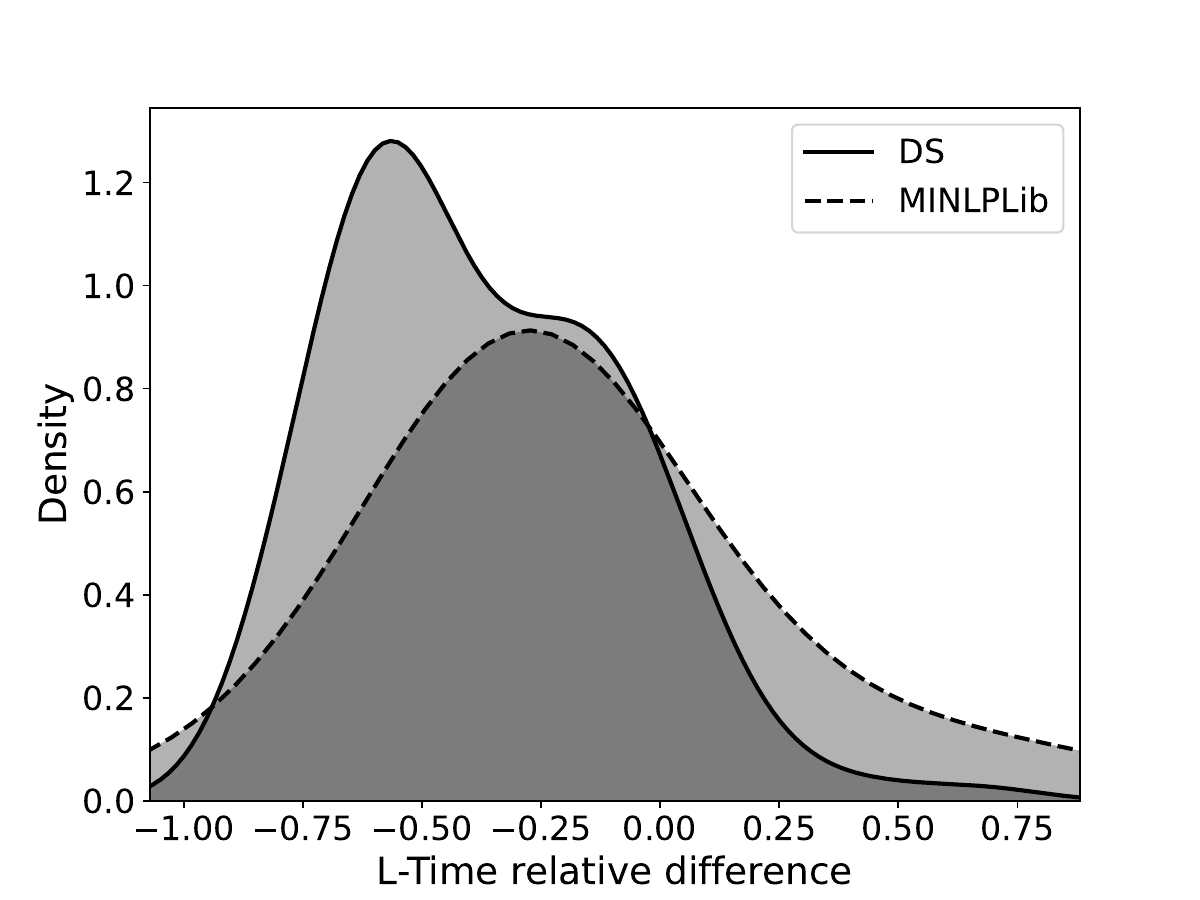}
        \caption{L-Time, \solver{Glop}.}
    \end{subfigure}    
    
    \begin{subfigure}{0.4\textwidth}
        \includegraphics[trim=0 0 0 28, clip, width=\textwidth]{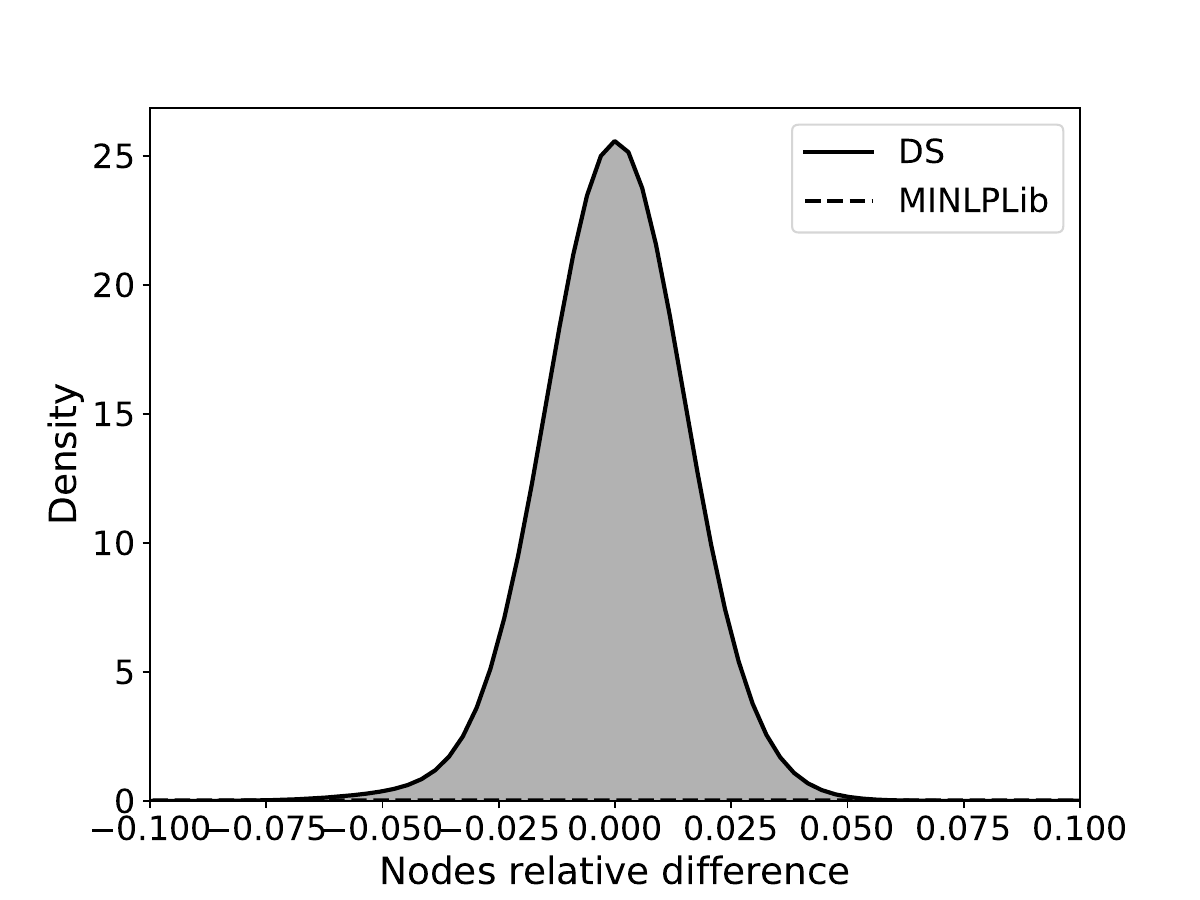}
        \caption{Nodes, \solver{Xpress}.}
    \end{subfigure}    
    \begin{subfigure}{0.4\textwidth}
        \includegraphics[trim=0 0 0 28, clip, width=\textwidth]{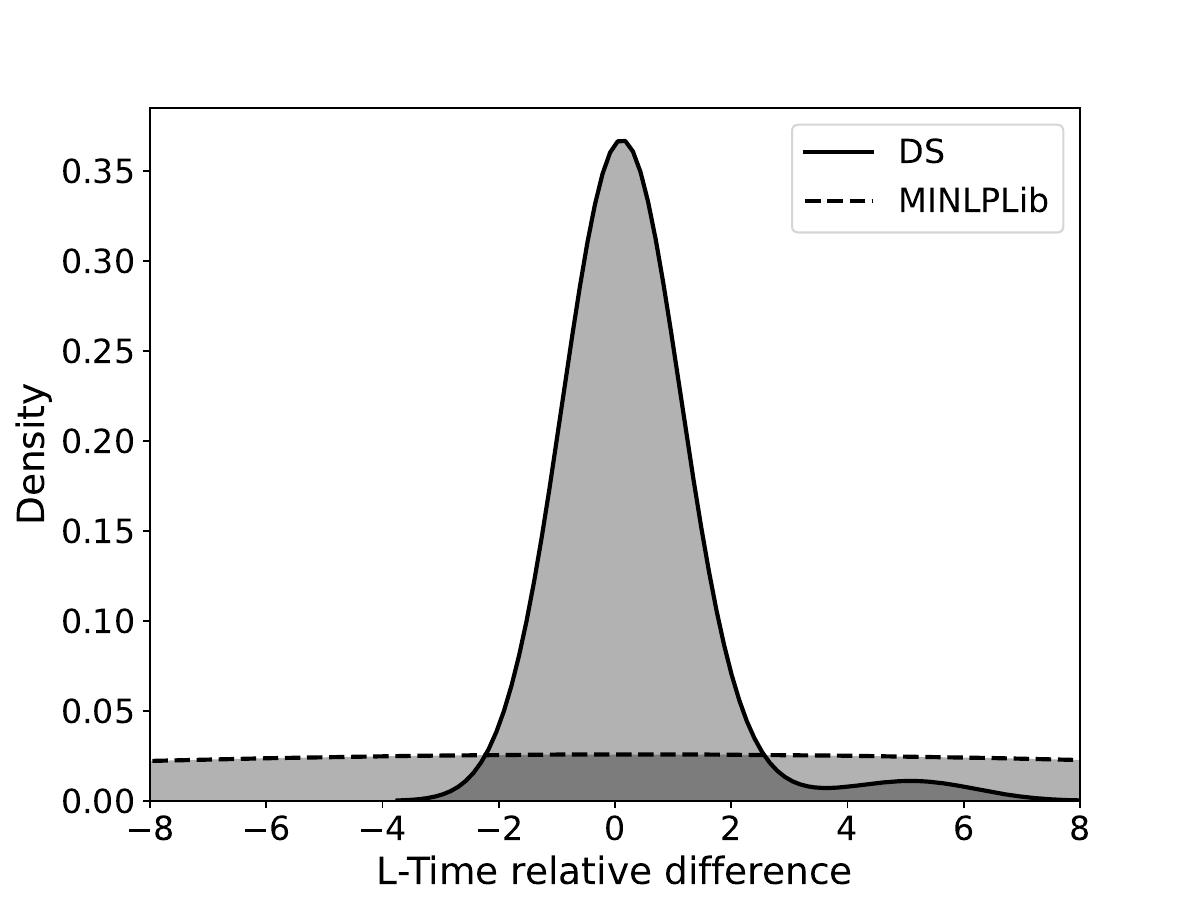}
        \caption{L-Time, \solver{Xpress}.}
    \end{subfigure}
    \caption{Relative differences in Nodes and L-Time between \RLTlooseB and \RLTtightB without bound tightening in continuous instances (solvers other than \solver{Gurobi}).}
    \label{fig:density_other_cont_no_bt}
\end{figure}

\subsection{Continuous instances with bound tightening}

\subsubsection{Main results}

\begin{table}[H]
\centering
\footnotesize
\renewcommand{\arraystretch}{0.88}
\begin{tabular}{|l|l|r|r|r|r|r|r|}
\hline
                       &            & \multicolumn{2}{c|}{DS}    & \multicolumn{2}{c|}{MINLPLib} & \multicolumn{2}{c|}{QPLIB} \\
\hline
                       &            & Instances & Variation (\%)       & Instances & Variation (\%)         & Instances & Variation (\%)       \\
\hline
\multirow{6}{*}{\solver{Gurobi}} 
	& Unsolved & 178 &	$+$18.75	& 186 & \textbf{$-$1.82}	& 66 &	$+$1.67 \\
	& Gap & 19 &	$+$26.09	& 43 &	\textbf{$-$7.30}	& 55 &	$+$1.59 \\
	& Time & 109 &	$+$34.80	&  46 &	$+$38.15	&  6 &	$+$3.37 \\
	& Pace & 125 &	$+$29.89	&  99 &	$+$10.51	&  66 &	\textbf{$-$1.87} \\
	& Nodes & 159 &	\textbf{$-$0.13}	&  130 & \textbf{$-$0.73}	&  5 &	\textbf{$-$11.40} \\
	& L-Time & 159 & $+$3.59	& 130 &	$+$1.04	&  5 &	\textbf{$-$8.29} \\
\hline
\multirow{6}{*}{\solver{Clp}} 
	& Unsolved & 174 & \textbf{$-$8.82}	& 187 &	$+$5.56	& 67 &	\phantom{$+$}0.00 \\
	& Gap & 35 &	\textbf{$-$44.39}	&  40 &	$+$10.60	& 53 &	\textbf{$-$0.35} \\
	& Time & 85 &	\textbf{$-$40.45}	&  38 &	$+$1.43	&  5 &	\textbf{$-$10.47} \\
	& Pace & 115 &	\textbf{$-$32.95}	&  92 &		\textbf{$-$20.55}	&  67 &	$+$0.98 \\
	& Nodes & 139 &	$+$0.60	&  130 &		\textbf{$-$0.30}	&  5 &	\textbf{$-$6.12} \\
	& L-Time & 139 &	\textbf{$-$34.36}	& 130 &	$+$0.71	&  5 &	\textbf{$-$4.63} \\
\hline
\multirow{6}{*}{\solver{Cplex}} 
	& Unsolved & 174 &	\phantom{$+$}0.00	& 184 &	$+$2.00	& 66 &	$+$1.67 \\
	& Gap & 16 &	$+$5.28	&  38 &	$+$1.36	&  53 &	\textbf{$-$0.70} \\
	& Time & 97 &	$+$2.19	&  39 &		$+$3.42	&  6 &		$+$1.21 \\
	& Pace & 113 &	$+$1.94	&  89 &		$+$2.13	&  66 &	$+$10.50 \\
	& Nodes & 158 &	\textbf{$-$0.29}	&  133 &		\textbf{$-$1.07}	&  5 &		\textbf{$-$8.71} \\
	& L-Time & 158 &	\textbf{$-$0.07}	&  133 &		\textbf{$-$0.27}	&  5 &	\textbf{$-$5.16} \\
\hline
\multirow{6}{*}{\solver{Glop}} 
	& Unsolved & 176 &	\textbf{$-$21.21}	& 187 &	$+$7.55	& 67 &		\phantom{$+$}0.00 \\
	& Gap & 36 &	\textbf{$-$58.34}	&  39 &	$+$47.32	&  53 &		\textbf{$-$4.71} \\
	& Time & 95 &	\textbf{$-$41.18}	&  42 &		$+$51.60	&  5 &		\textbf{$-$20.23} \\
	& Pace & 118 &	\textbf{$-$29.46}	&  95 &		$+$14.92	&  67 &		\textbf{$-$0.28} \\
	& Nodes & 140 &	$+$1.95	&  130 &		$+$3.42	&  3 &		\textbf{$-$23.89} \\
	& L-Time & 140 &	\textbf{$-$45.59}	&  130 &		\textbf{$-$14.13}	& L3 &	\textbf{$-$15.37} \\
\hline
\multirow{6}{*}{\solver{Xpress}} 
	& Unsolved & 179 &	$+$3.03	& 183 &	$+$1.47	& 65 &	\phantom{$+$}0.00 \\
	& Gap & 34 &	$+$18.80	&  55 &	$+$4.48	&  56 &		$+$0.43 \\
	& Time & 141 &	\textbf{$-$5.28}	&  89 &		\textbf{$-$5.53}	&  2 &		\textbf{$-$8.02} \\
	& Pace & 174 &	\textbf{$-$4.08}	&  155 &		\textbf{$-$9.02}	&  65 &		\textbf{$-$2.72} \\
	& Nodes & 145 &	$+$0.09	&  112 &		\textbf{$-$0.22}	&  2 &		\textbf{$-$11.47} \\
	& L-Time & 145 &	$+$2.85	&  112 &	\textbf{$-$1.29}	&  2 &	$+$0.79 \\
\hline 
\end{tabular}
\caption{Results with bound tightening in continuous instances.}
\label{tab:withbt}
\end{table}

\subsubsection{Densities for Nodes and L-Time. Results for \solver{Gurobi} solver}

\begin{figure}[H]
    \centering
    \begin{subfigure}{0.4\textwidth}
        \includegraphics[trim=0 0 0 28, clip, width=\textwidth]{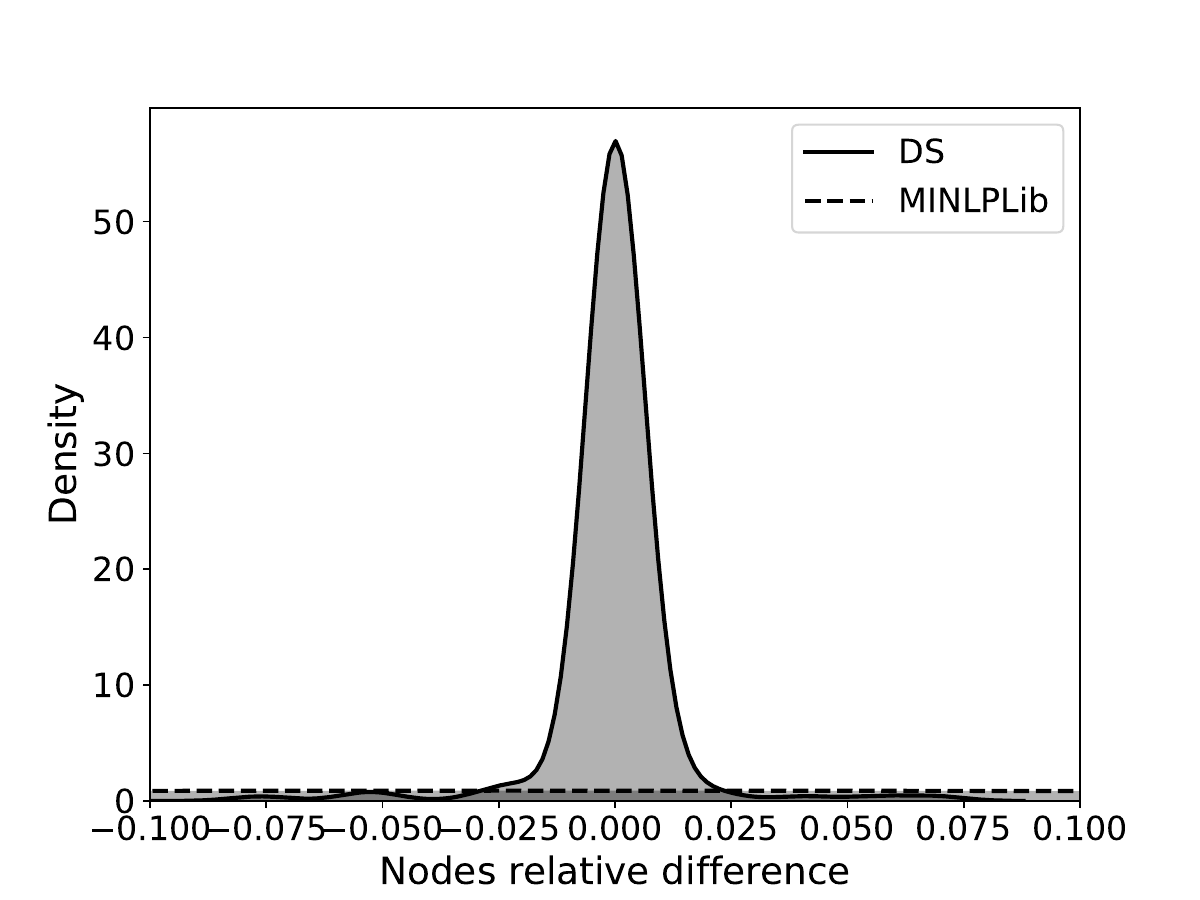}
        \caption{Nodes, \solver{Gurobi}.}
    \end{subfigure}    
    \begin{subfigure}{0.4\textwidth}
        \includegraphics[trim=0 0 0 28, clip, width=\textwidth]{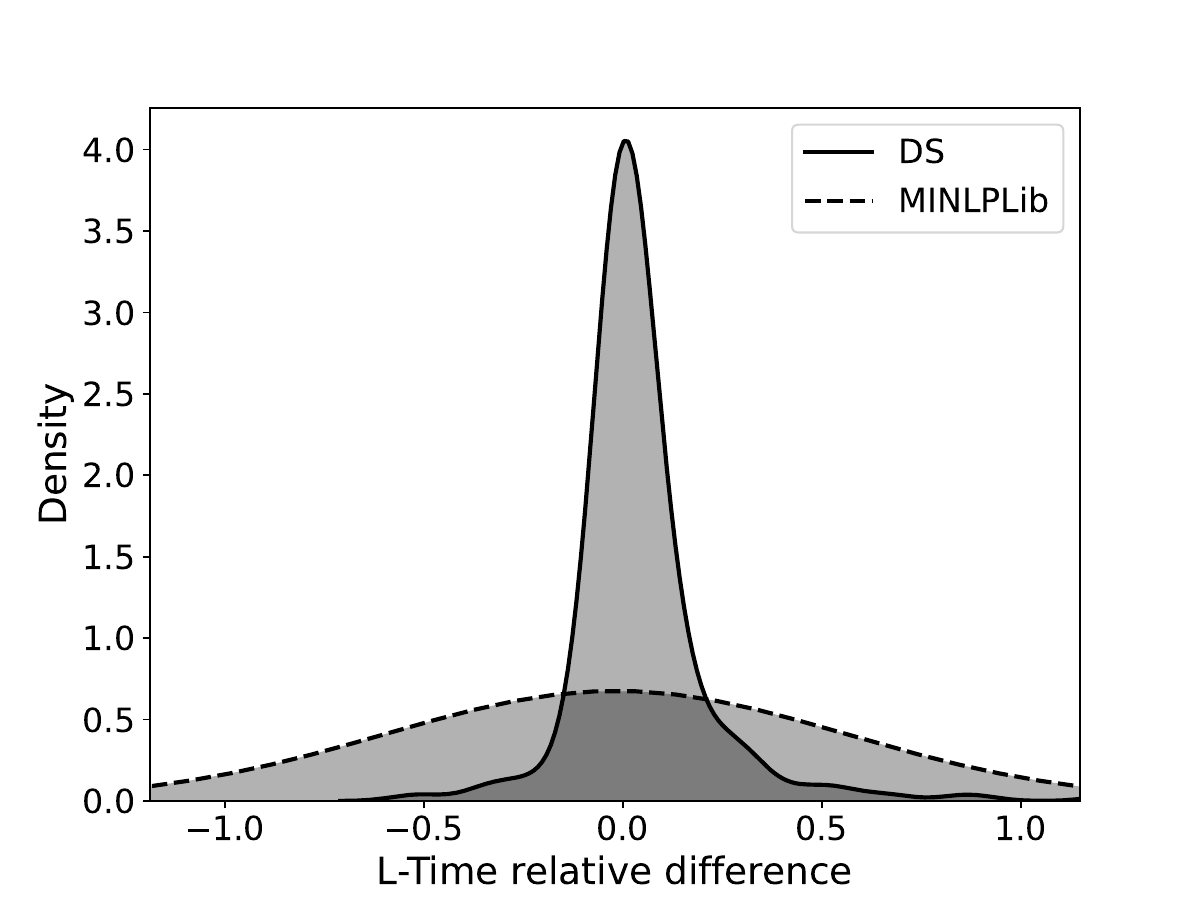}
        \caption{L-Time, \solver{Gurobi}.}
    \end{subfigure}    
    \caption{Relative differences in Nodes and L-Time between \RLTlooseB and \RLTtightB with bound tightening in continuous instances (\solver{Gurobi} solver).}
\end{figure}

\subsubsection{Densities for Nodes and L-Time. Results for the rest of the solvers}

\begin{figure}[H]
    \centering
    \begin{subfigure}{0.4\textwidth}
        \includegraphics[trim=0 0 0 28, clip, width=\textwidth]{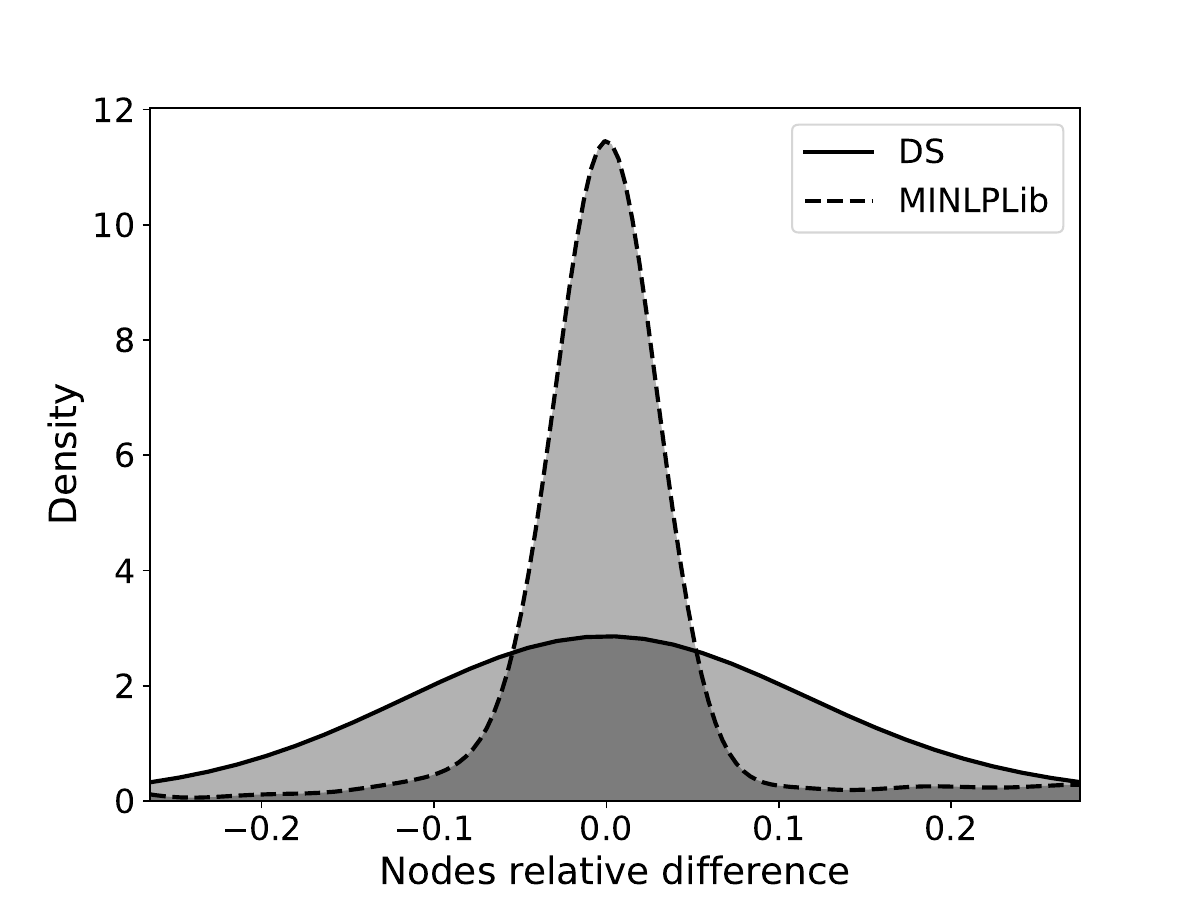}
        \caption{Nodes, \solver{Clp}.}
    \end{subfigure}    
    \begin{subfigure}{0.4\textwidth}
        \includegraphics[trim=0 0 0 28, clip, width=\textwidth]{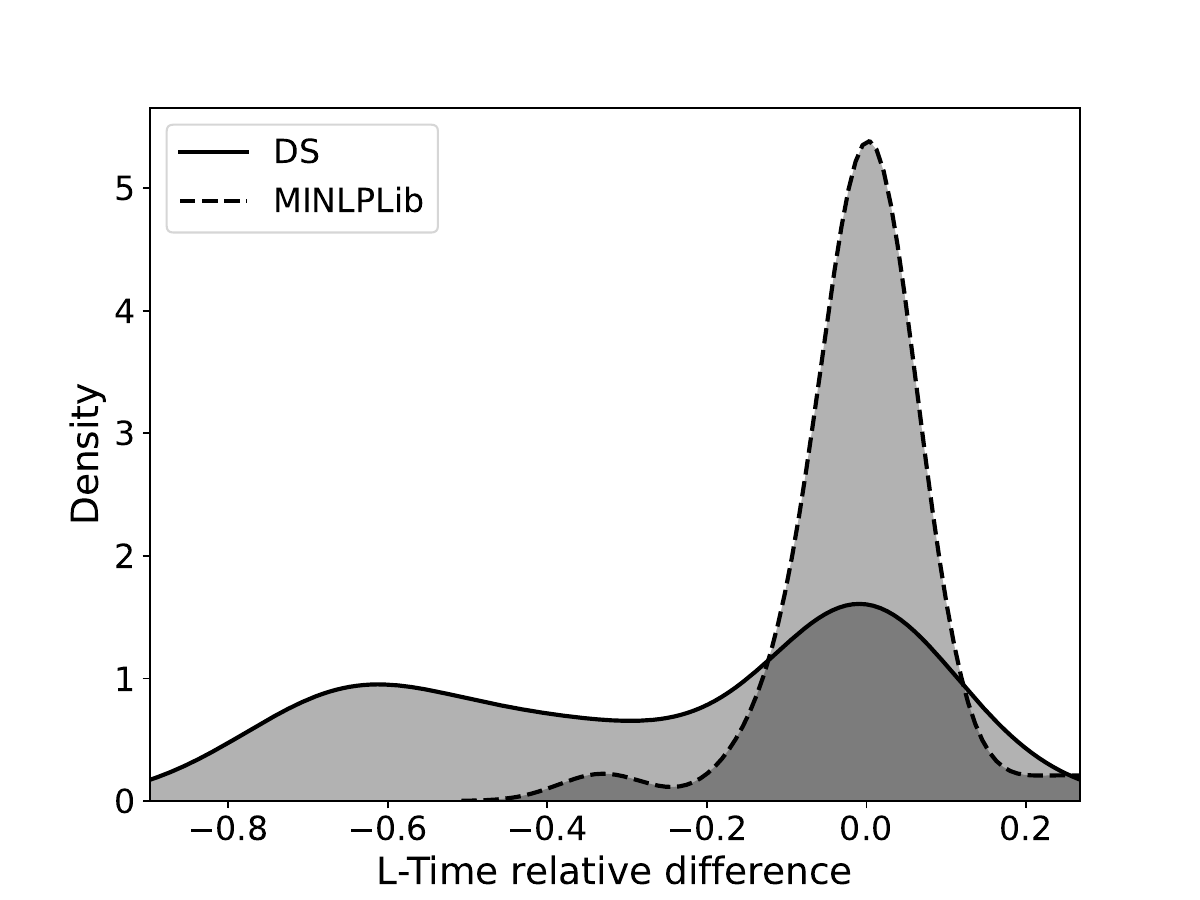}
        \caption{L-Time, \solver{Clp}.}
    \end{subfigure}    
    
    \begin{subfigure}{0.4\textwidth}
        \includegraphics[trim=0 0 0 28, clip, width=\textwidth]{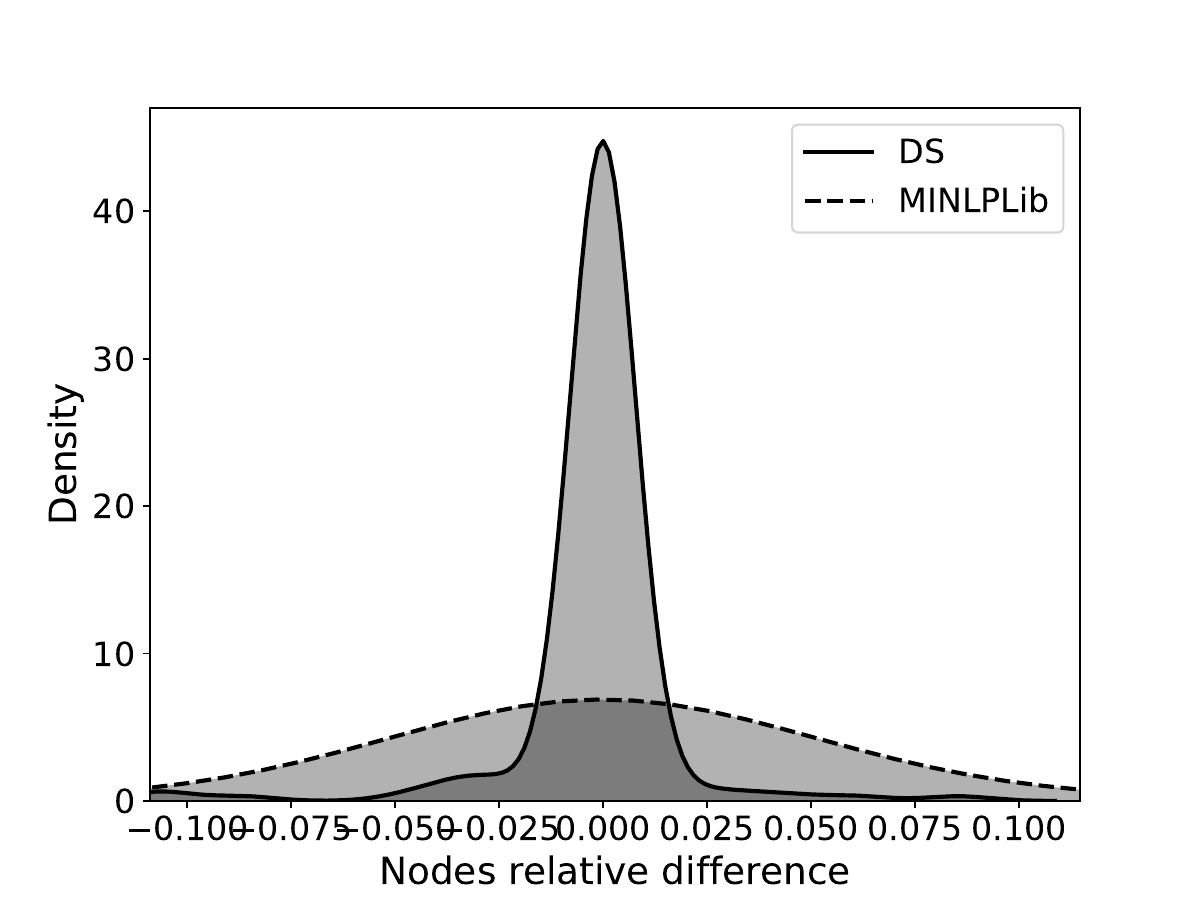}
        \caption{Nodes, \solver{Cplex}.}
    \end{subfigure}    
    \begin{subfigure}{0.4\textwidth}
        \includegraphics[trim=0 0 0 28, clip, width=\textwidth]{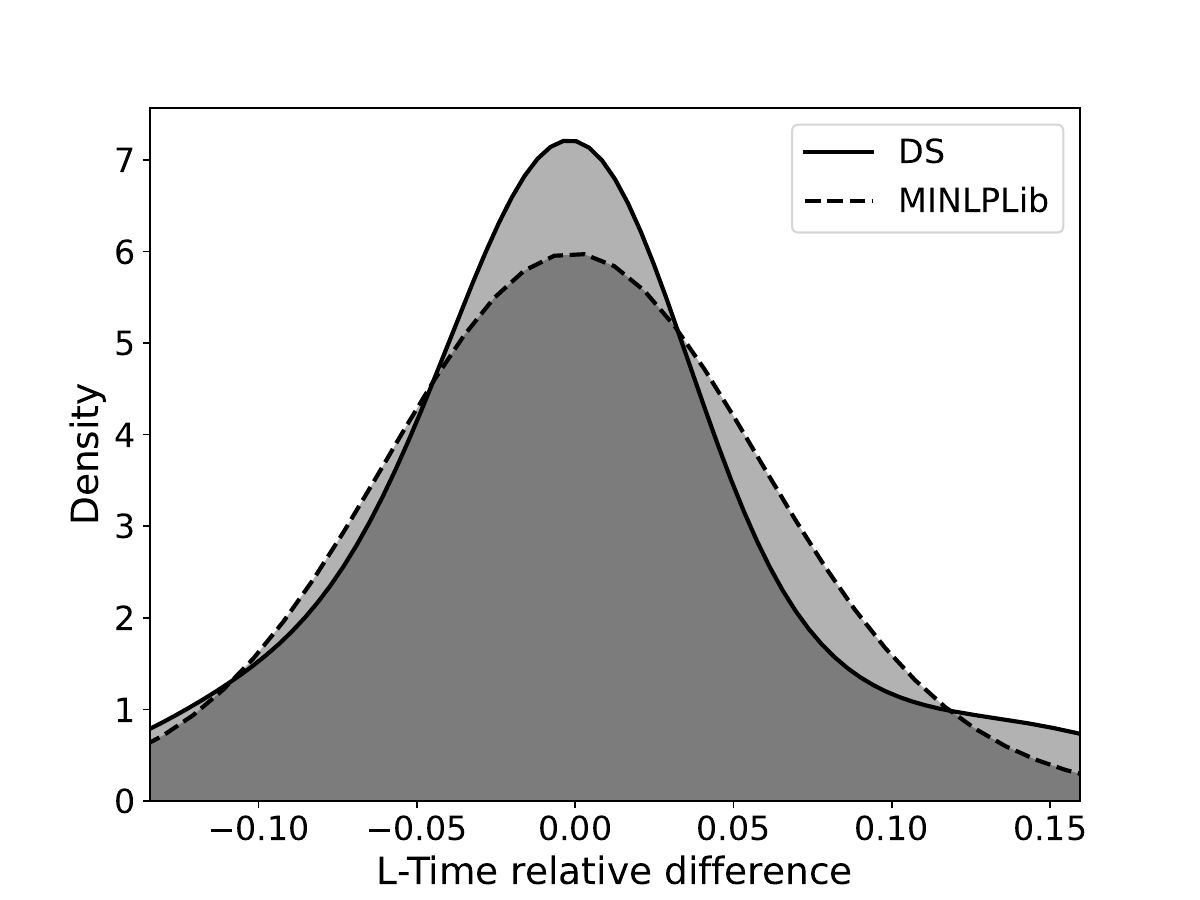}
        \caption{L-Time, \solver{Cplex}.}
    \end{subfigure}    
    
    \begin{subfigure}{0.4\textwidth}
        \includegraphics[trim=0 0 0 28, clip, width=\textwidth]{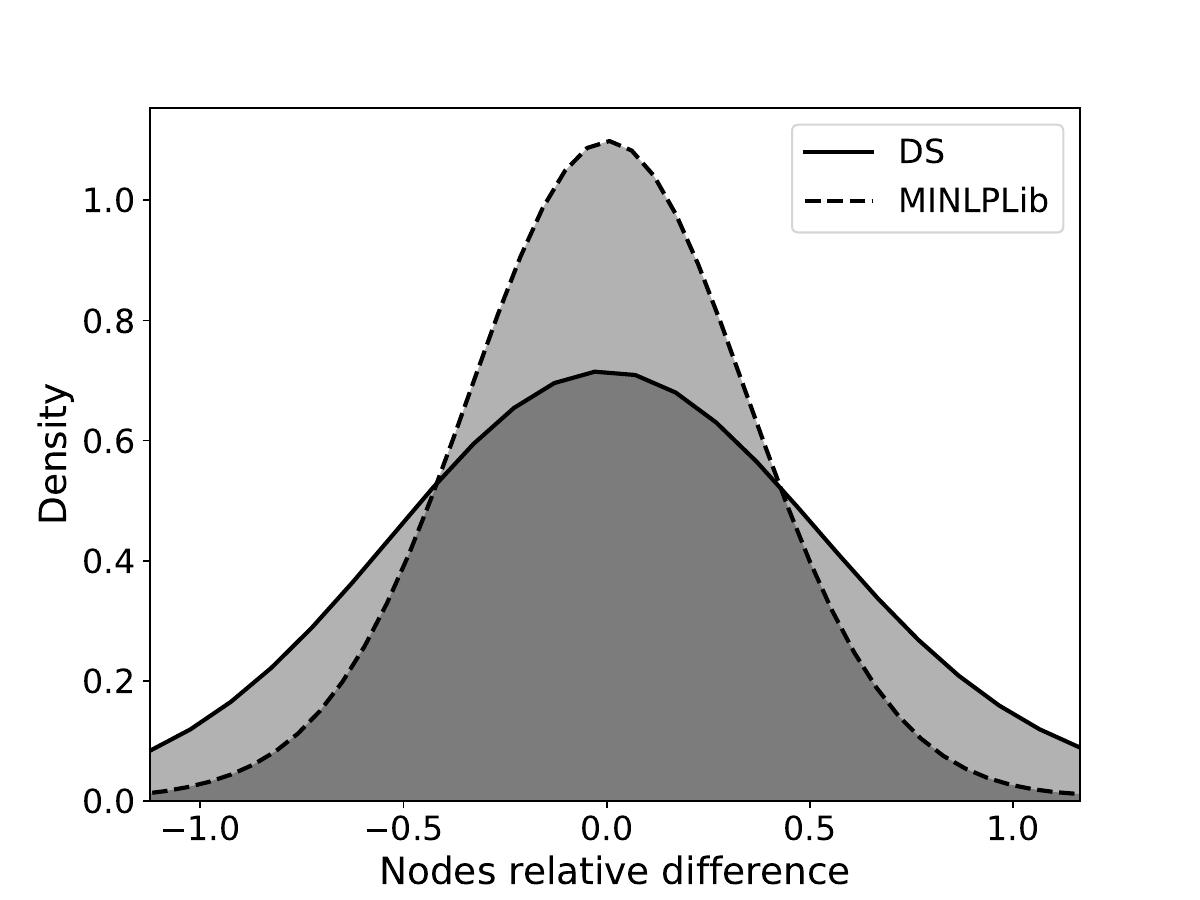}
        \caption{Nodes, \solver{Glop}.}
    \end{subfigure}    
    \begin{subfigure}{0.4\textwidth}
        \includegraphics[trim=0 0 0 28, clip, width=\textwidth]{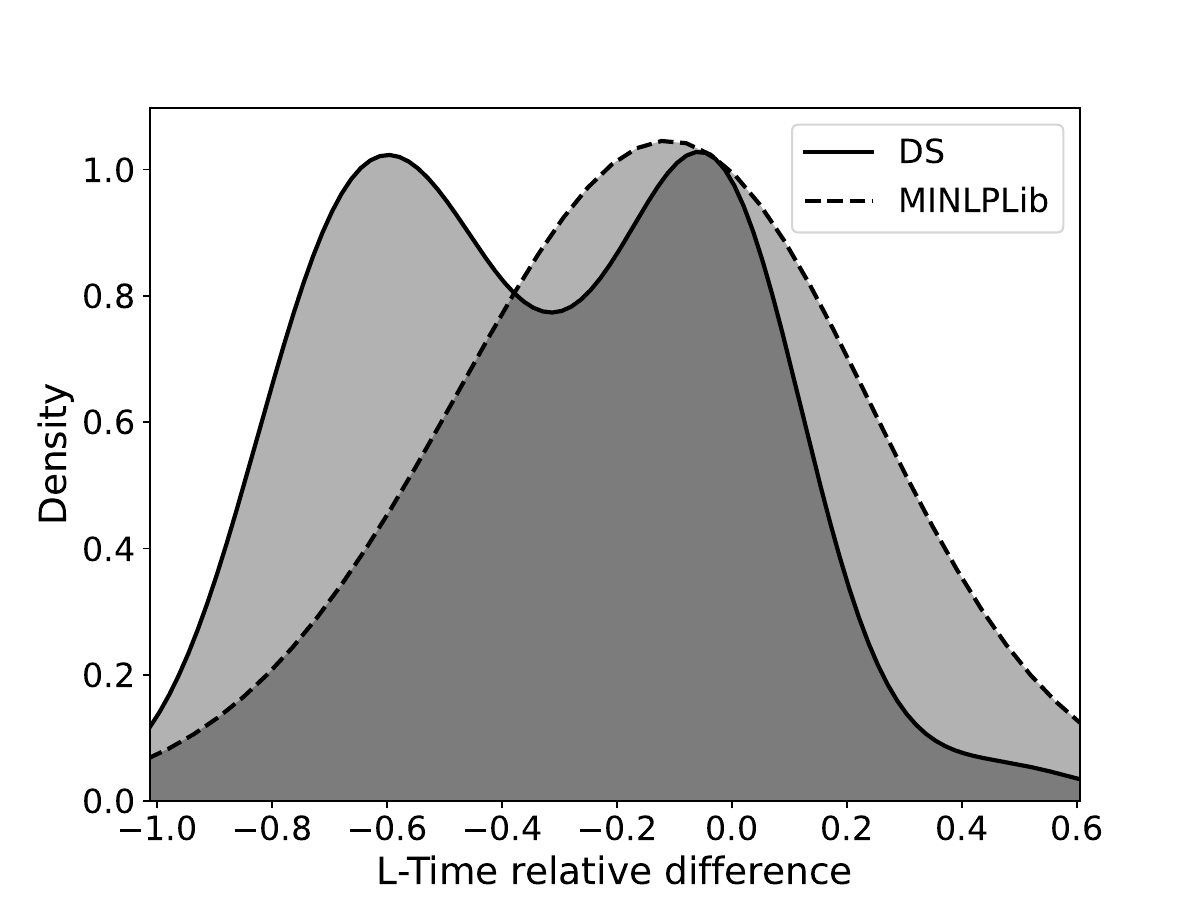}
        \caption{L-Time, \solver{Glop}.}
    \end{subfigure}    
    
    \begin{subfigure}{0.4\textwidth}
        \includegraphics[trim=0 0 0 28, clip, width=\textwidth]{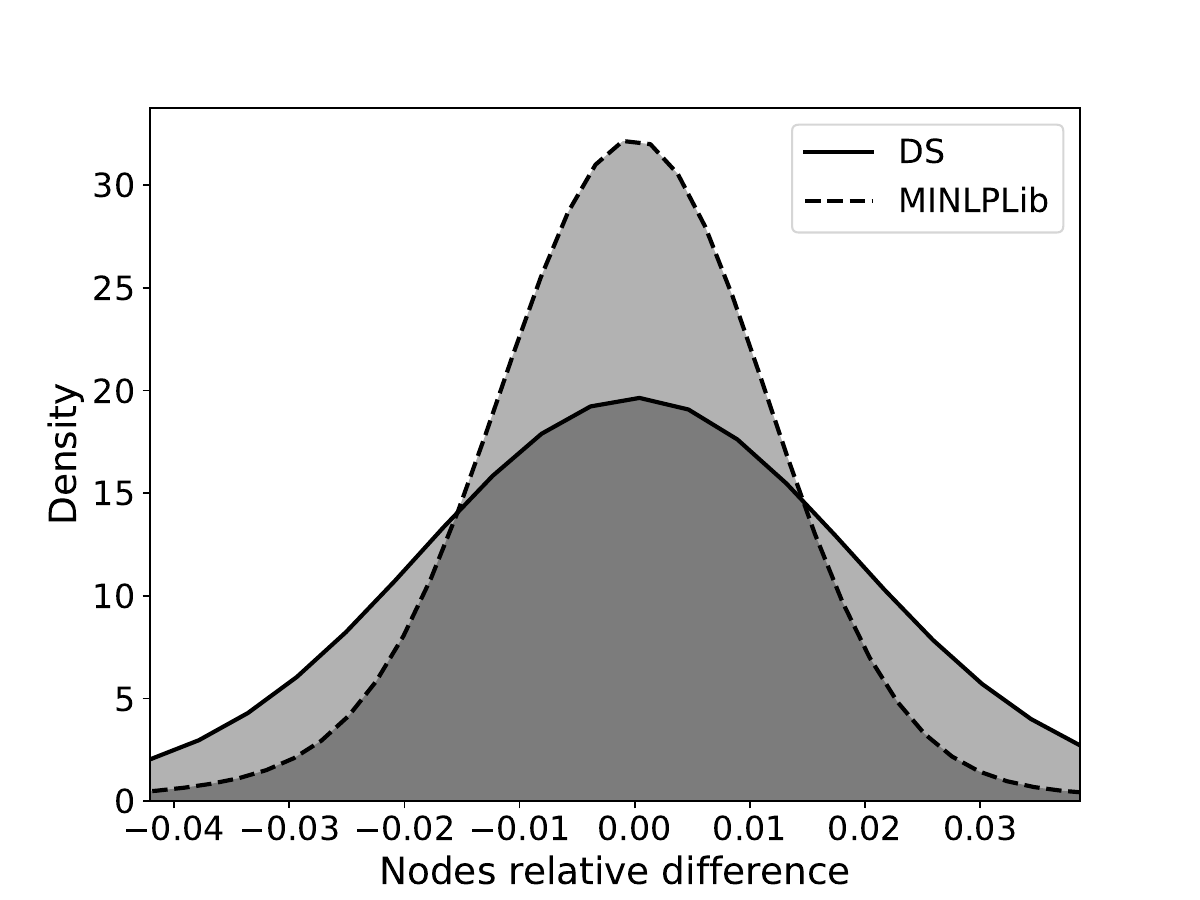}
        \caption{Nodes, \solver{Xpress}.}
    \end{subfigure}    
    \begin{subfigure}{0.4\textwidth}
        \includegraphics[trim=0 0 0 28, clip, width=\textwidth]{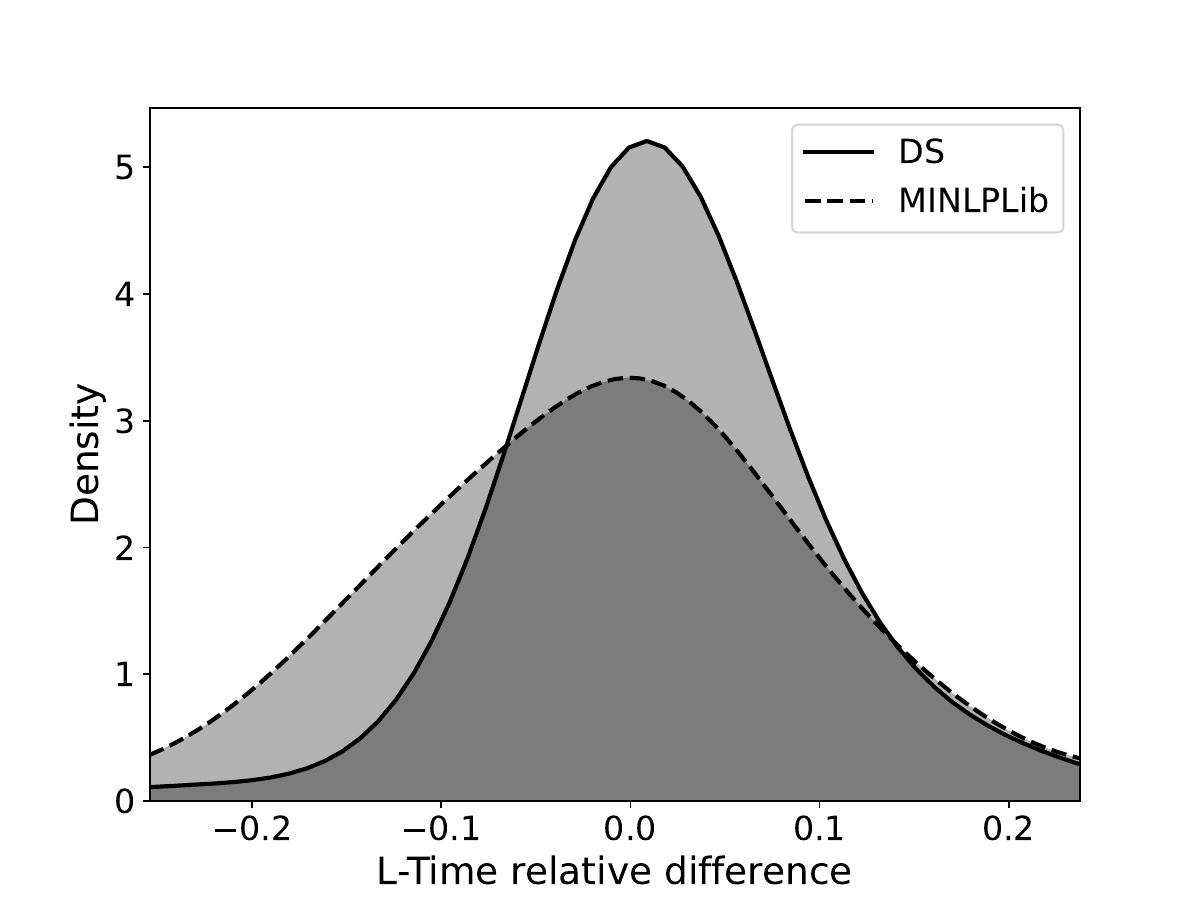}
        \caption{L-Time, \solver{Xpress}.}
    \end{subfigure}
    
    \caption{Relative differences in Nodes and L-Time between \RLTlooseB and \RLTtightB with bound tightening in continuous instances (solvers other than \solver{Gurobi}).}
\end{figure}

\subsection{Mixed-integer instances without bound tightening}

\subsubsection{Main results}

\begin{table}[H]
\centering
\footnotesize
\renewcommand{\arraystretch}{0.88}
\begin{tabular}{|l|l|r|r|r|r|r|r|}
\hline
                       &            & \multicolumn{2}{c|}{MINLPLib-integer}& \multicolumn{2}{c|}{QPLIB-integer} \\
\hline
                       &            & Instances & Variation (\%)        & Instances & Variation (\%)       \\
\hline
\multirow{6}{*}{\solver{Gurobi}} 
	& Unsolved & 212 &	+1.52	& 141 &	\textbf{$-$1.61} \\
	& Gap & 60 &	+18.10	&  57 &		\textbf{$-$2.90}\\
	& Time & 71 &	+23.97	&  77 &		+4.26\\
	& Pace & 137 &	+10.38	&  138 &		+2.49\\
	& Nodes & 145 &	\textbf{$-$1.07}	&  79 &	+0.45\\
	& L-Time & 145 &	+291.44 & 79 &		+9.80\\
\hline
\multirow{6}{*}{\solver{Cbc}} 
	& Unsolved & 209 &		\phantom{$+$}0.00	& 135 &	+0.88\\
	& Gap & 91 &	+2.19	&  110 &		+9.00\\
	& Time & 60 &		+8.05	&  22 &		+1.74\\
	& Pace & 159 &		\textbf{$-$13.44}	&  135 &		\textbf{$-$15.21}\\
	& Nodes & 108 &	\textbf{$-$4.88}	&  19 &		+1.29\\
	& L-Time & 108 &		+11.28	& 19 &		\textbf{$-$3.82}\\
\hline
\multirow{6}{*}{\solver{Cplex}} 
	& Unsolved & 211 &	\phantom{$+$}0.00	& 141 &	\textbf{$-$1.47}\\
	& Gap & 62 &		+4.23	&  63 &		\textbf{$-$5.69}\\
	& Time & 68 &		\textbf{$-$4.83}	&  69 &	\textbf{$-$2.77}\\
	& Pace & 137 &		\textbf{$-$6.71}	&  136 &		\textbf{$-$1.34}\\
	& Nodes & 142 &		\textbf{$-$3.81}	&  73 &		\textbf{$-$4.59}\\
	& L-Time & 142 &		+2.23	& 73 &		+3.49\\
\hline
\multirow{6}{*}{\solver{Xpress}} 
	& Unsolved & 212 &		\phantom{$+$}0.00	& 141 &		\phantom{$+$}0.00\\
	& Gap & 74 &		\textbf{$-$7.00}	&  72 &		\textbf{$-$4.08}\\
	& Time & 82 &		+11.86	&  62 &		\textbf{$-$1.26}\\
	& Pace & 163 &	+4.82	&  141 &		+0.27\\
	& Nodes & 129 &	+6.92	&  62 &		+0.18\\
	& L-Time & 129 &		\textbf{$-$4.08}	& 62 &	\textbf{$-$1.74}\\
\hline 
\end{tabular}
\caption{Results without bound tightening in mixed-integer instances.}
\label{tab:withoutbtmixedinteger}
\end{table}

\subsubsection{Densities for Nodes and L-Time. Results for \solver{Gurobi} solver}

\begin{figure}[H]
    \centering
    \begin{subfigure}{0.4\textwidth}
        \includegraphics[trim=0 0 0 28, clip, width=\textwidth]{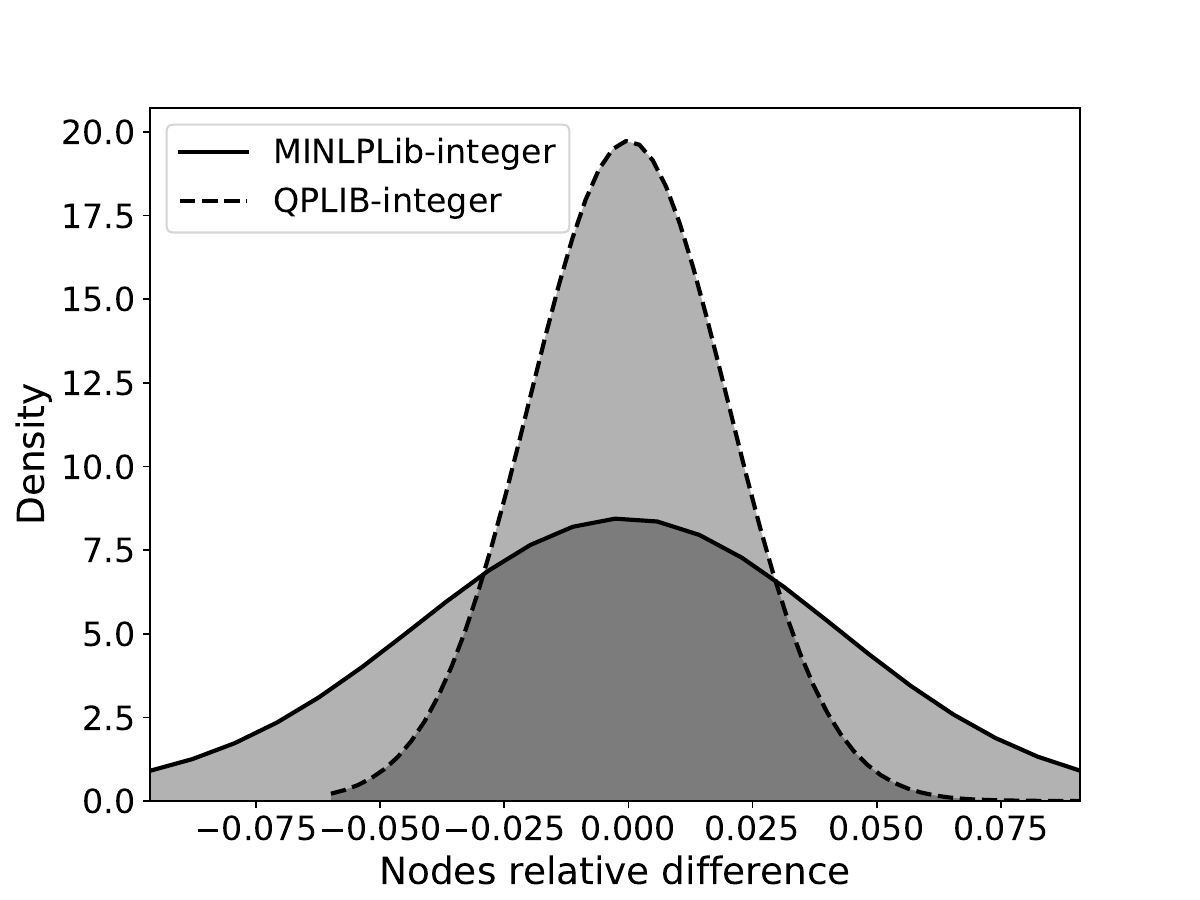}
        \caption{Nodes, \solver{Gurobi}.}
    \end{subfigure}    
    \begin{subfigure}{0.4\textwidth}
        \includegraphics[trim=0 0 0 28, clip, width=\textwidth]{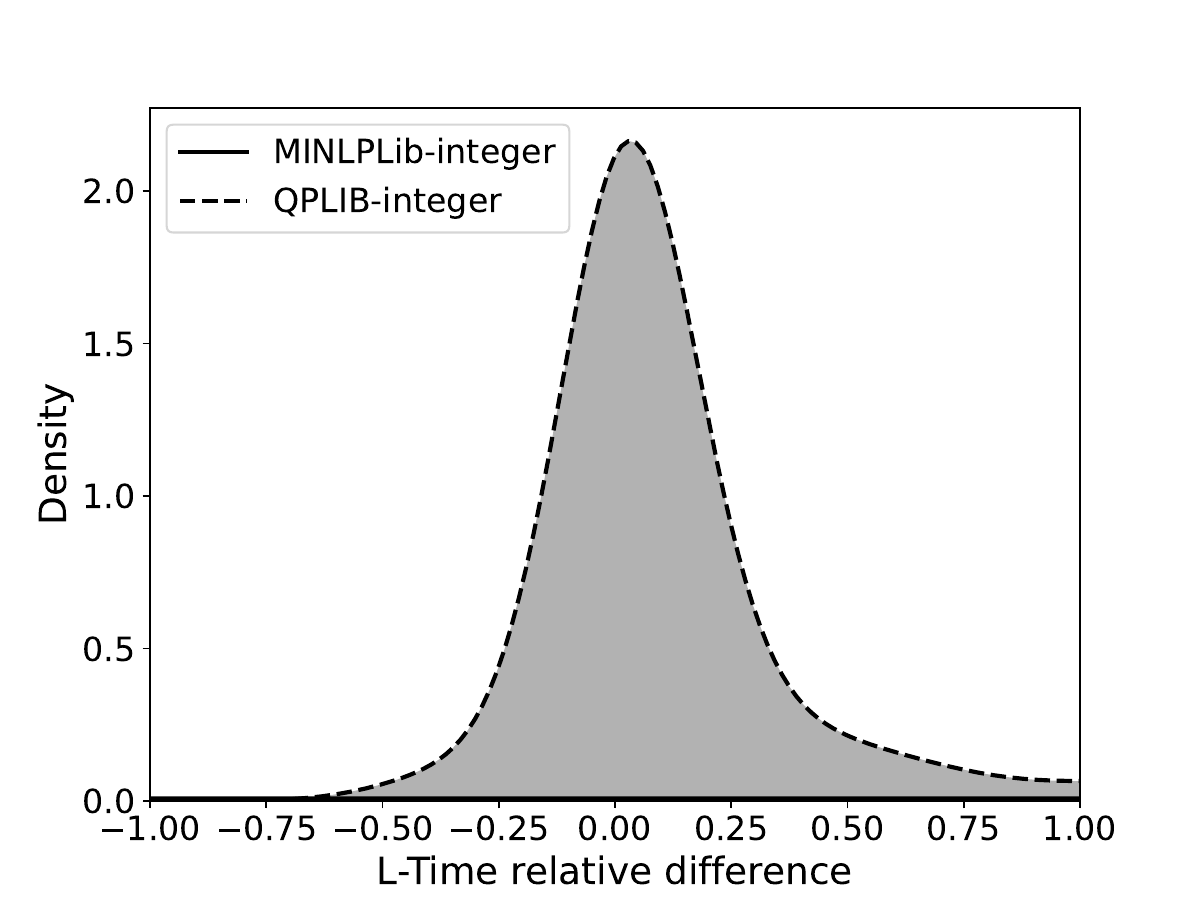}
        \caption{L-Time, \solver{Gurobi}.}
    \end{subfigure}    
    \caption{Relative differences in Nodes and L-Time between \RLTlooseB and \RLTtightB without bound tightening in mixed-integer instances (\solver{Gurobi} solver).}
\end{figure}

\subsubsection{Densities for Nodes and L-Time. Results for the rest of the solvers}

\begin{figure}[H]
    \centering
    \begin{subfigure}{0.4\textwidth}
        \includegraphics[trim=0 0 0 28, clip, width=\textwidth]{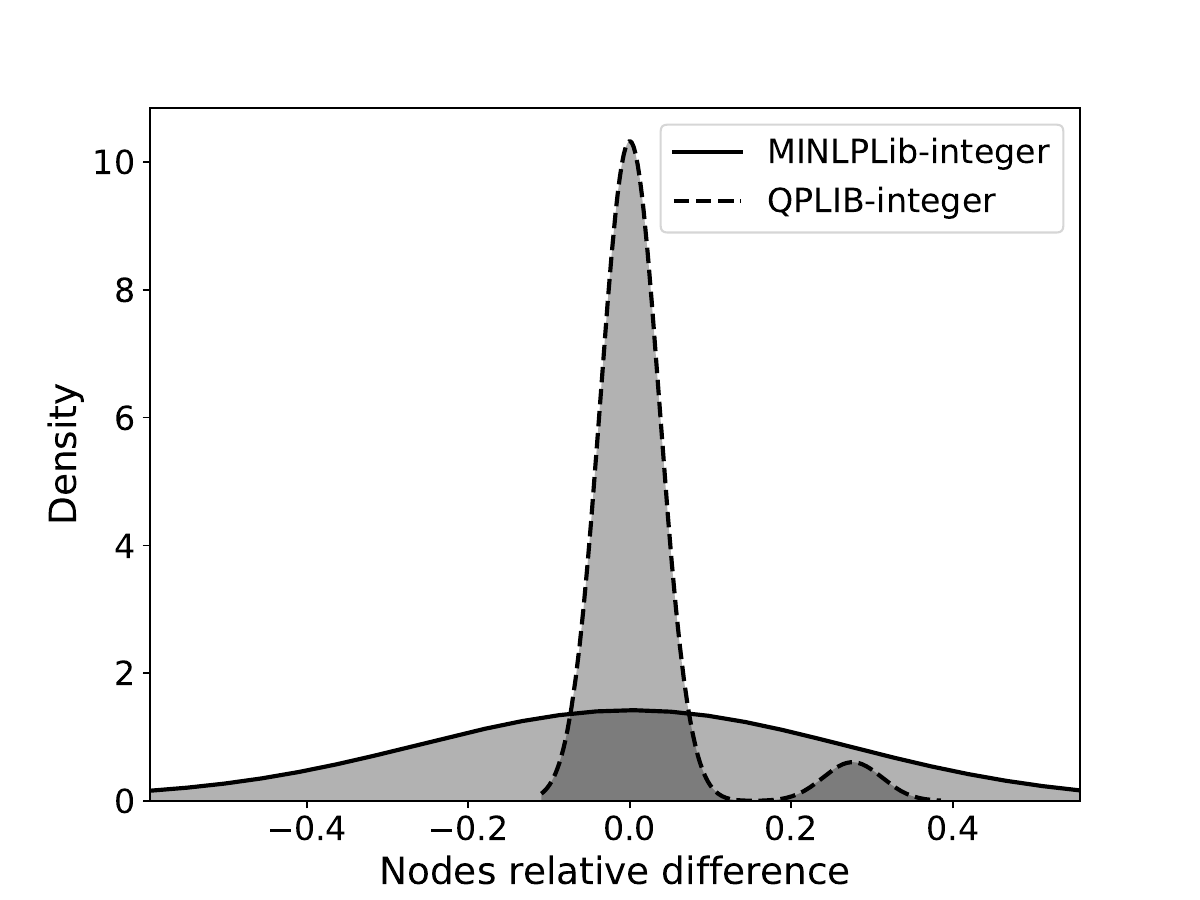}
        \caption{Nodes, \solver{Cbc}.}
    \end{subfigure}    
    \begin{subfigure}{0.4\textwidth}
        \includegraphics[trim=0 0 0 28, clip, width=\textwidth]{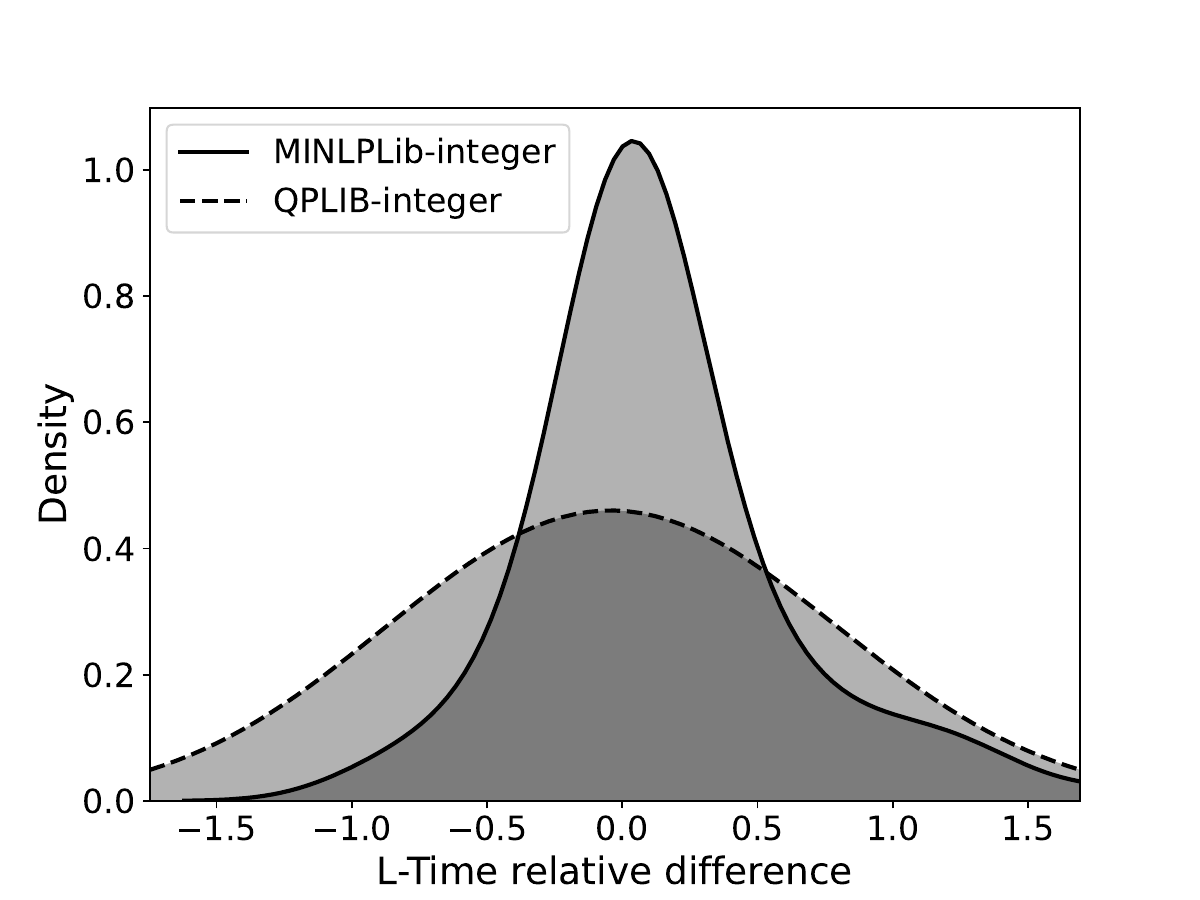}
        \caption{L-Time, \solver{Cbc}.}
    \end{subfigure}    
    
    \begin{subfigure}{0.4\textwidth}
        \includegraphics[trim=0 0 0 28, clip, width=\textwidth]{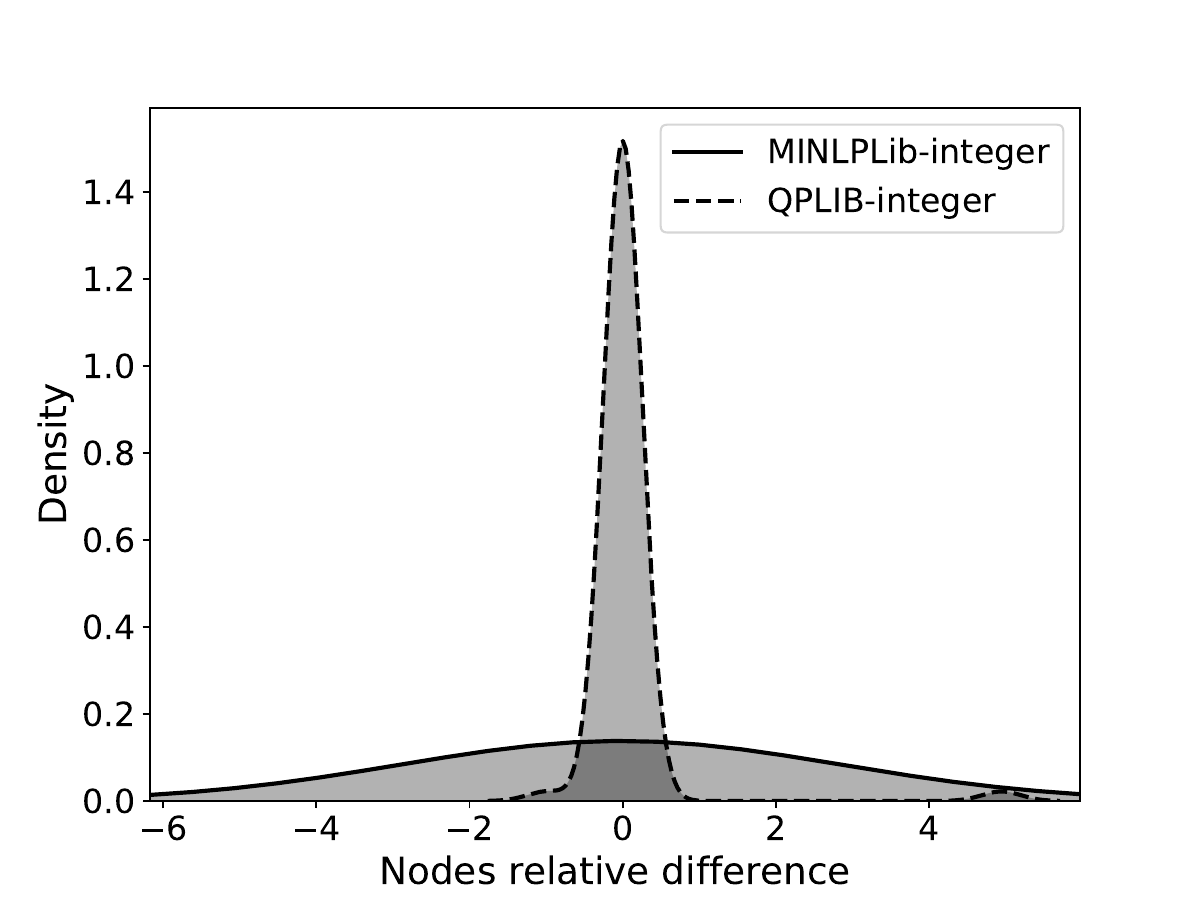}
        \caption{Nodes, \solver{Cplex}.}
    \end{subfigure}    
    \begin{subfigure}{0.4\textwidth}
        \includegraphics[trim=0 0 0 28, clip, width=\textwidth]{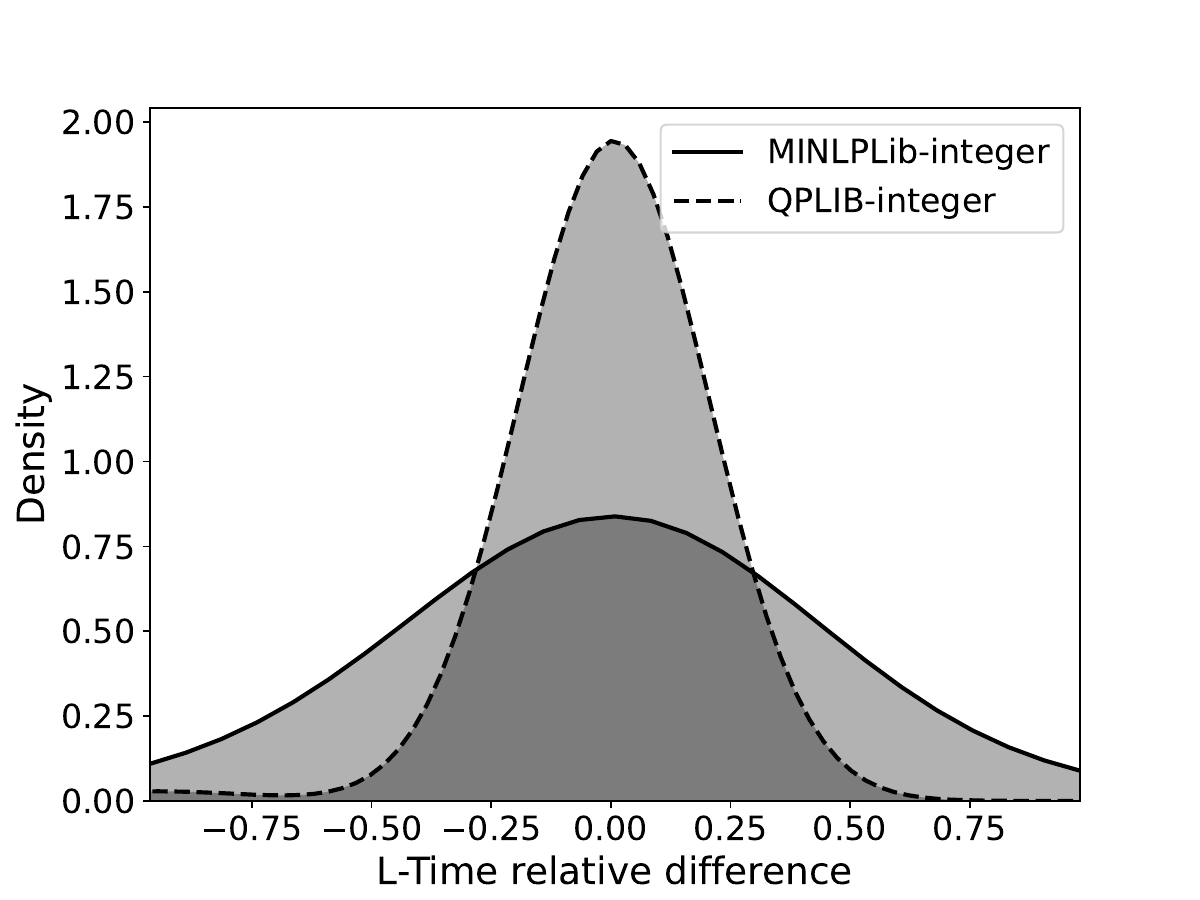}
        \caption{L-Time, \solver{Cplex}.}
    \end{subfigure}    
    
    \begin{subfigure}{0.4\textwidth}
        \includegraphics[trim=0 0 0 28, clip, width=\textwidth]{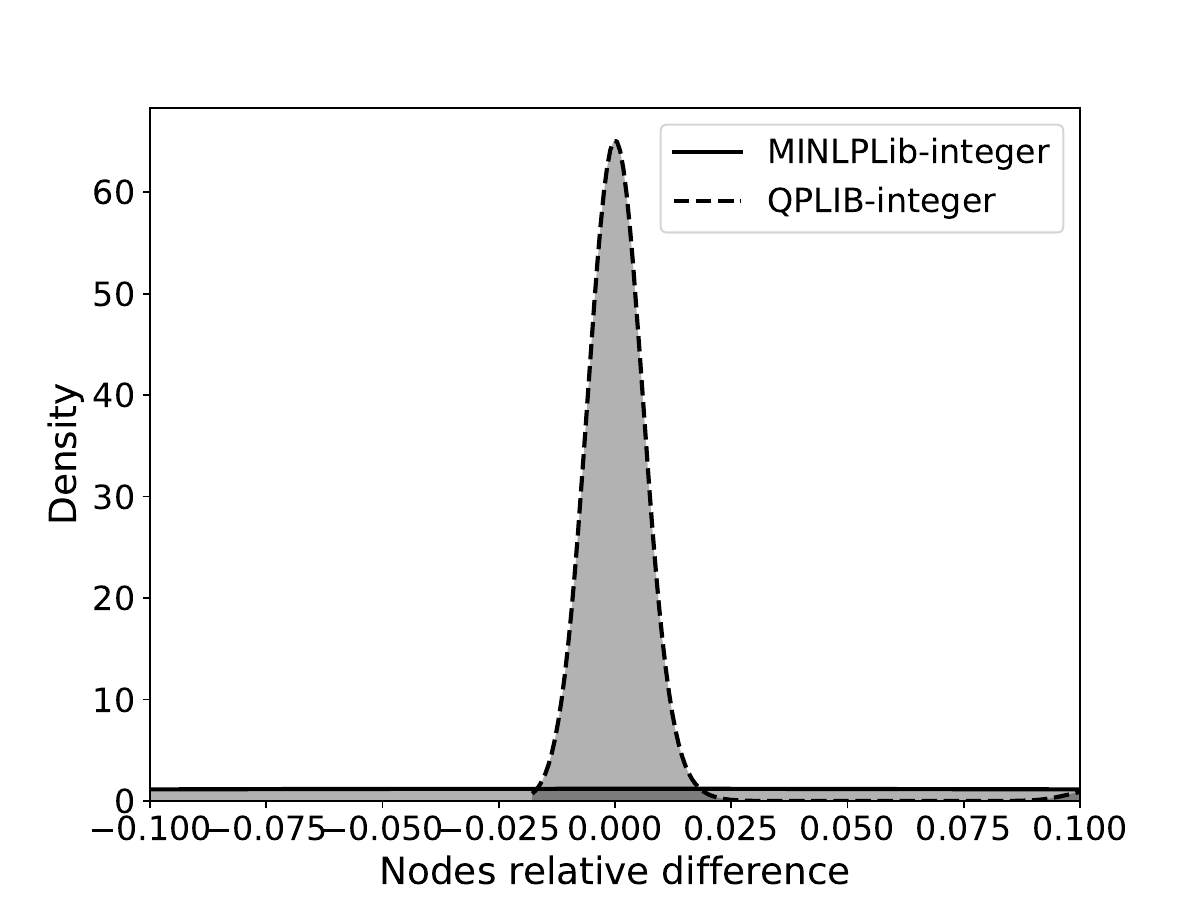}
        \caption{Nodes, \solver{Xpress}.}
    \end{subfigure}    
    \begin{subfigure}{0.4\textwidth}
        \includegraphics[trim=0 0 0 28, clip, width=\textwidth]{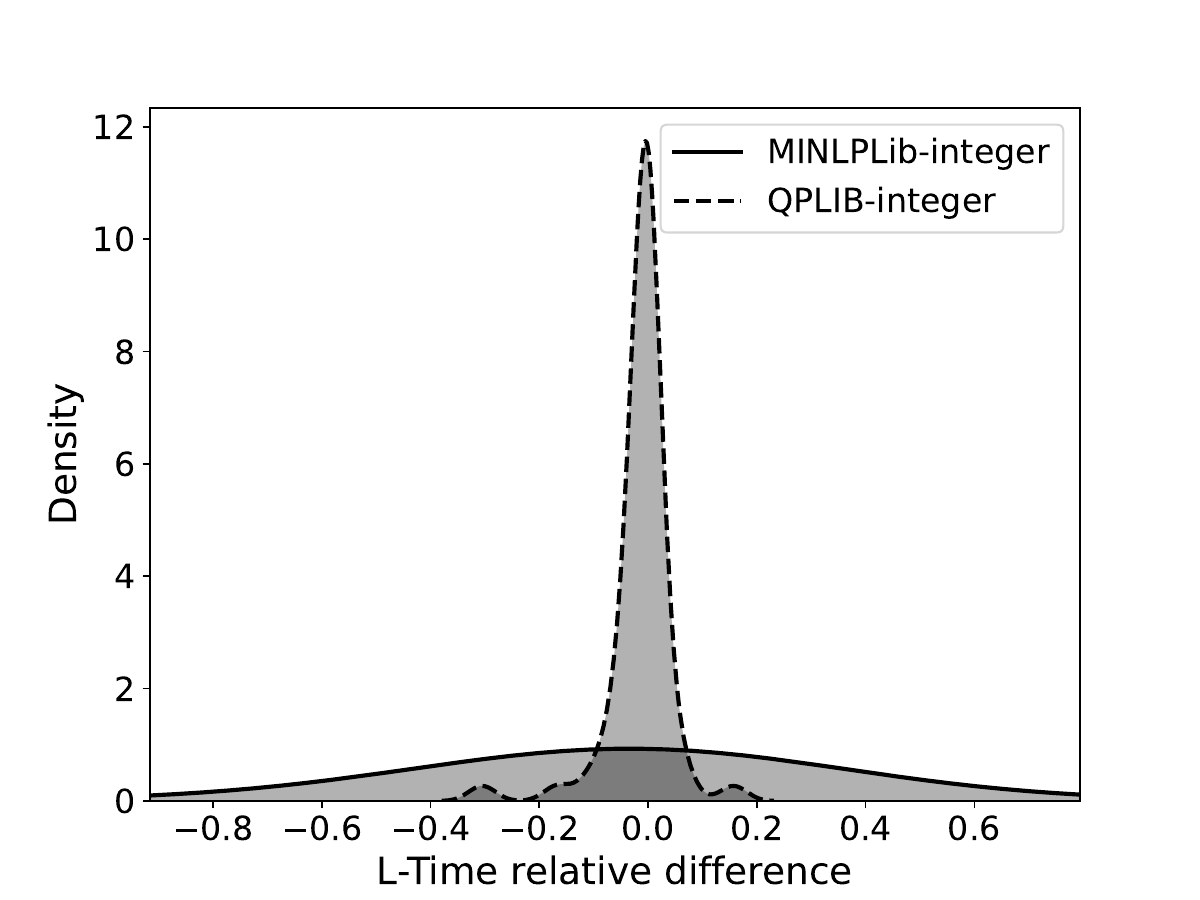}
        \caption{L-Time, \solver{Xpress}.}
    \end{subfigure}
    
    \caption{Relative differences in Nodes and L-Time between \RLTlooseB and \RLTtightB without bound tightening in mixed-integer instances (solvers other than \solver{Gurobi}).}
\end{figure}

\subsection{Mixed-integer instances with bound tightening}

\subsubsection{Main results}

\begin{table}[H]
\centering
\footnotesize
\renewcommand{\arraystretch}{0.88}
\begin{tabular}{|l|l|r|r|r|r|r|r|}
\hline
                       &            & \multicolumn{2}{c|}{MINLPLib-integer}& \multicolumn{2}{c|}{QPLIB-integer} \\
\hline
                       &            & Instances & Variation (\%)        & Instances & Variation (\%)       \\
\hline
\multirow{6}{*}{\solver{Gurobi}} 
	& Unsolved & 213 & +1.54	& 141 &	\phantom{$+$}0.00\\
	& Gap & 59 &	+16.88	&  56 &		+2.01\\
	& Time & 75 &		+11.15	& 81 &		+1.42\\
	& Pace & 140 &		+5.13	&  141 &		+0.56\\
	& Nodes & 147 &	+1.03	&  81 &		+0.37\\
	& L-Time & 147 &		+314.89	& 81 &		+8.43\\
\hline
\multirow{6}{*}{\solver{Cbc}} 
	& Unsolved & 210 &		+1.06	& 137 &		+1.77\\
	& Gap & 87 &	+3.31	&  109 &		+15.17\\
	& Time & 62 &		+5.66	&  24 &		+3.50\\
	& Pace & 156 &		+5.19	&  137 &		+0.56\\
	& Nodes & 115 &		\textbf{$-$3.15}	&  22 &		\textbf{$-$0.11}\\
	& L-Time & 115 &		+8.23	&22 &	+0.98\\
\hline
\multirow{6}{*}{\solver{Cplex}} 
	& Unsolved & 211 &		\phantom{$+$}0.00	& 141 &	\textbf{$-$2.94}\\
	& Gap & 58 &		+5.39	&  62 &		+4.39\\
	& Time & 80 &		\textbf{$-$8.21}	&  75 &		\textbf{$-$3.48}\\
	& Pace & 146 &		\textbf{$-$4.09}	&  141 &		\textbf{$-$2.14}\\
	& Nodes & 145 &		\textbf{$-$4.84}	&  73 &		\textbf{$-$1.99}\\
	& L-Time & 145 &		\textbf{$-$2.87}	&73 &		+0.87\\
\hline
\multirow{6}{*}{\solver{Xpress}} 
	& Unsolved & 212 &		\phantom{$+$}0.00	& 141 &		\phantom{$+$}0.00\\
	& Gap & 72 &		+0.64	&  74 &		+1.41\\
	& Time & 85 &	+13.65	&  61 &	+1.14\\
	& Pace & 164 &		+12.88	&  141 &		+1.52\\
	& Nodes & 131 &		+8.35	&  61 &	+0.70\\
	& L-Time & 131 &		+4.96	&  61 &	+2.56\\
\hline 
\end{tabular}
\caption{Results with bound tightening in mixed-integer instances.}
\label{tab:withbtmixedinteger}
\end{table}

\subsubsection{Densities for Nodes and L-Time. Results for \solver{Gurobi} solver}

\begin{figure}[H]
    \centering
    \begin{subfigure}{0.4\textwidth}
        \includegraphics[trim=0 0 0 28, clip, width=\textwidth]{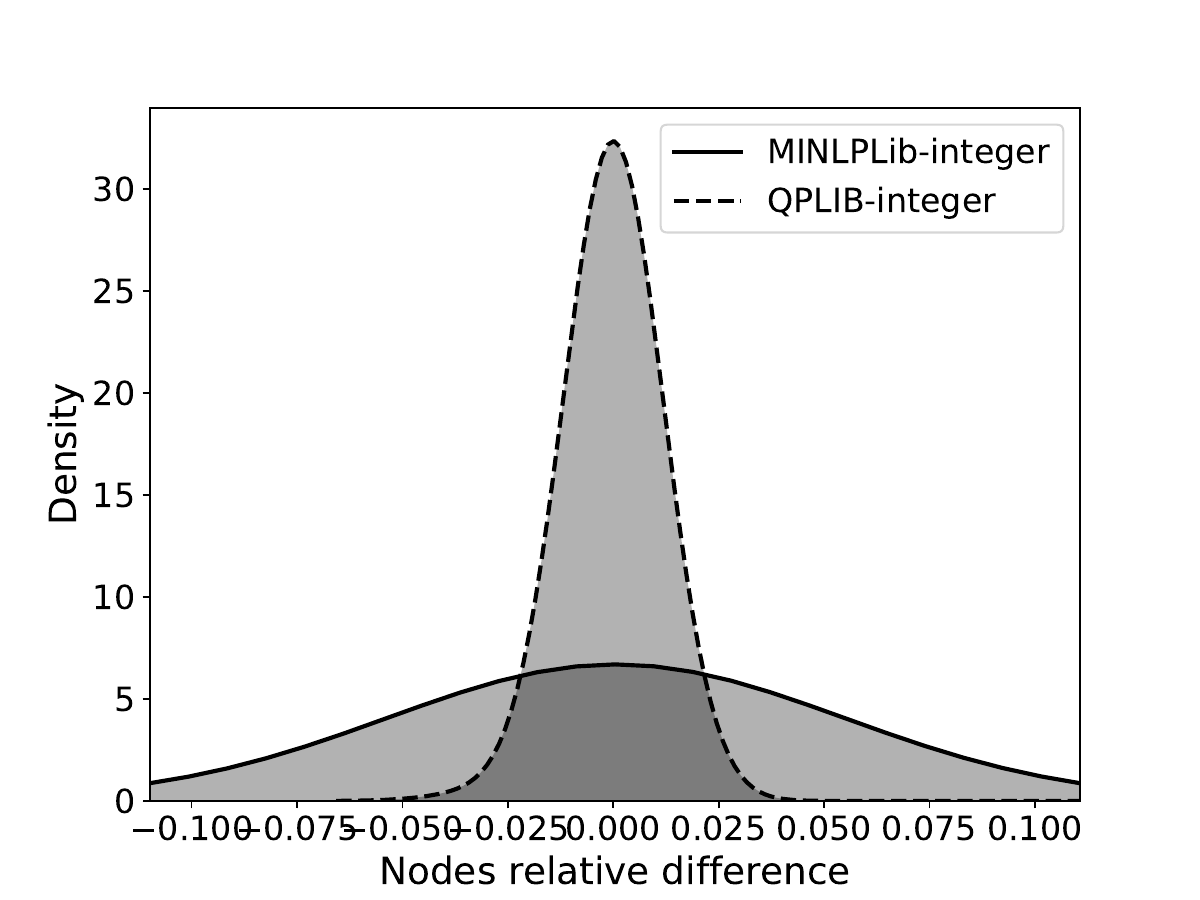}
        \caption{Nodes, \solver{Gurobi}.}
    \end{subfigure}    
    \begin{subfigure}{0.4\textwidth}
        \includegraphics[trim=0 0 0 28, clip, width=\textwidth]{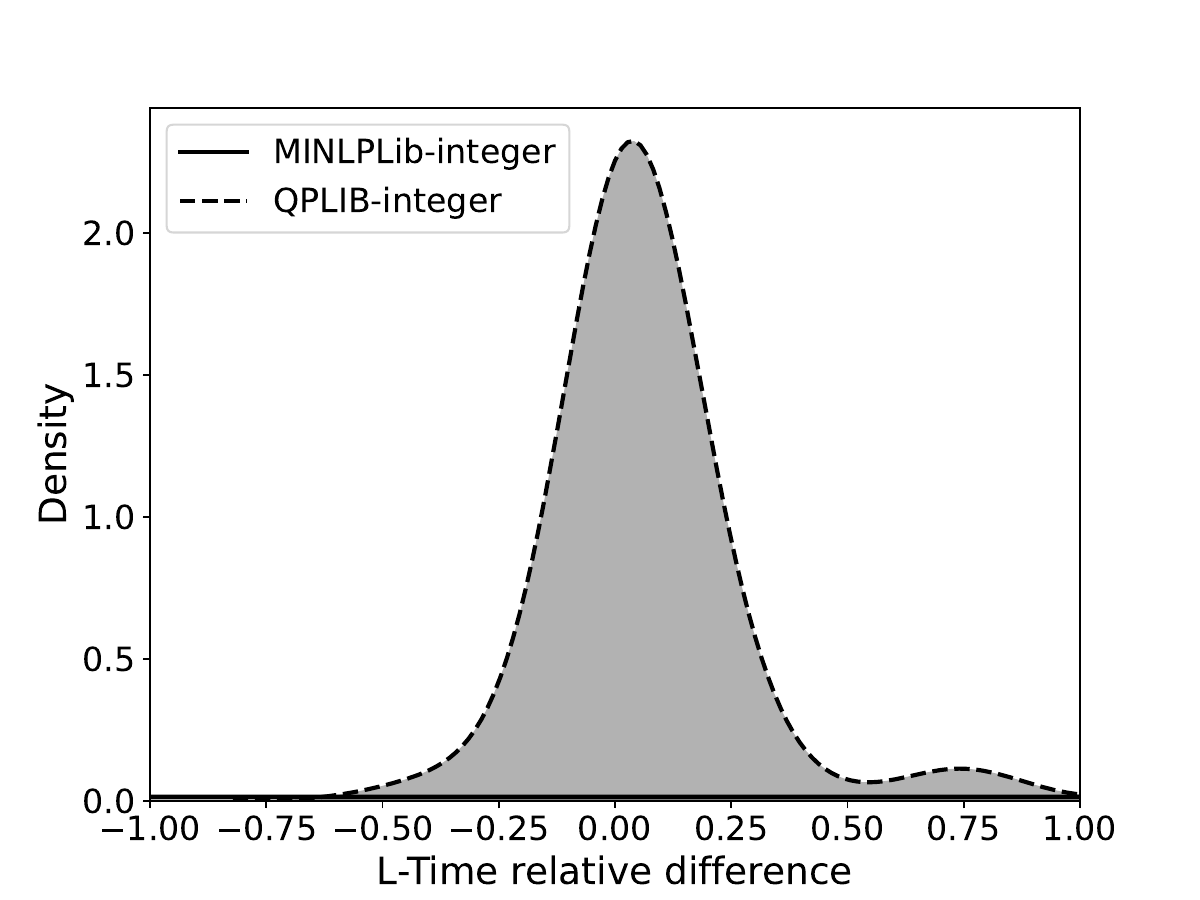}
        \caption{L-Time, \solver{Gurobi}.}
    \end{subfigure}    
    \caption{Relative differences in Nodes and L-Time between \RLTlooseB and \RLTtightB with bound tightening in mixed-integer instances (\solver{Gurobi} solver).}
\end{figure}

\subsubsection{Densities for Nodes and L-Time. Results for the rest of the solvers}

\begin{figure}[H]
    \centering
    \begin{subfigure}{0.4\textwidth}
        \includegraphics[trim=0 0 0 28, clip, width=\textwidth]{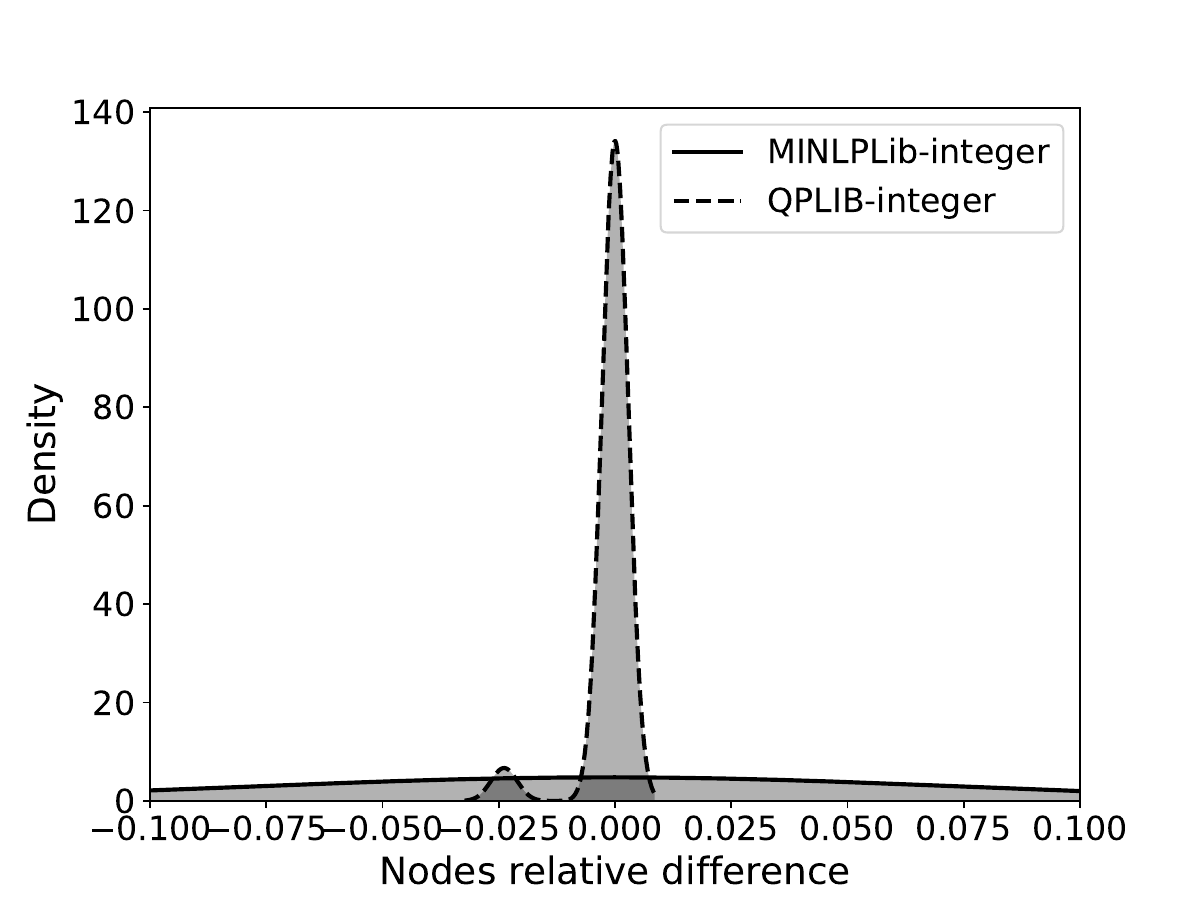}
        \caption{Nodes, \solver{Cbc}.}
    \end{subfigure}    
    \begin{subfigure}{0.4\textwidth}
        \includegraphics[trim=0 0 0 28, clip, width=\textwidth]{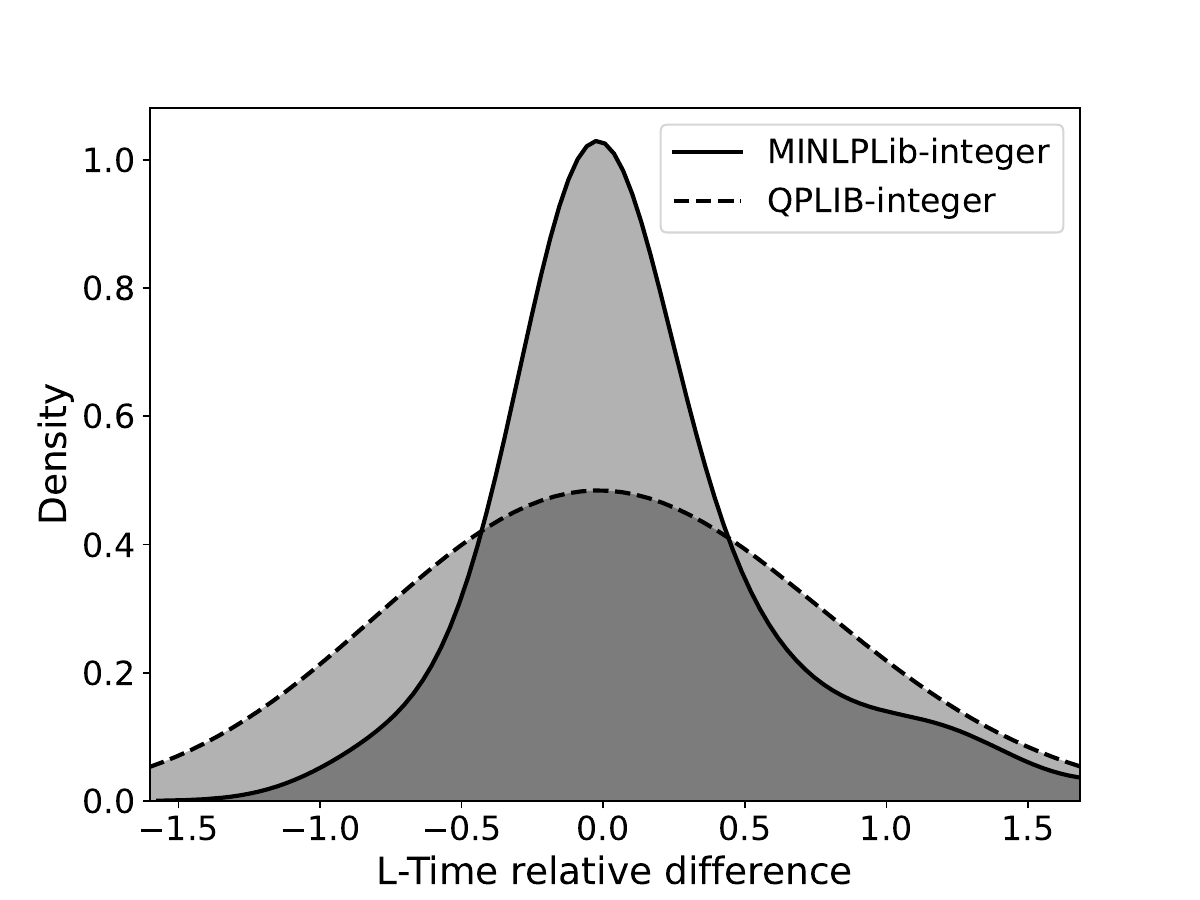}
        \caption{L-Time, \solver{Cbc}.}
    \end{subfigure}        
    \begin{subfigure}{0.4\textwidth}
        \includegraphics[trim=0 0 0 28, clip, width=\textwidth]{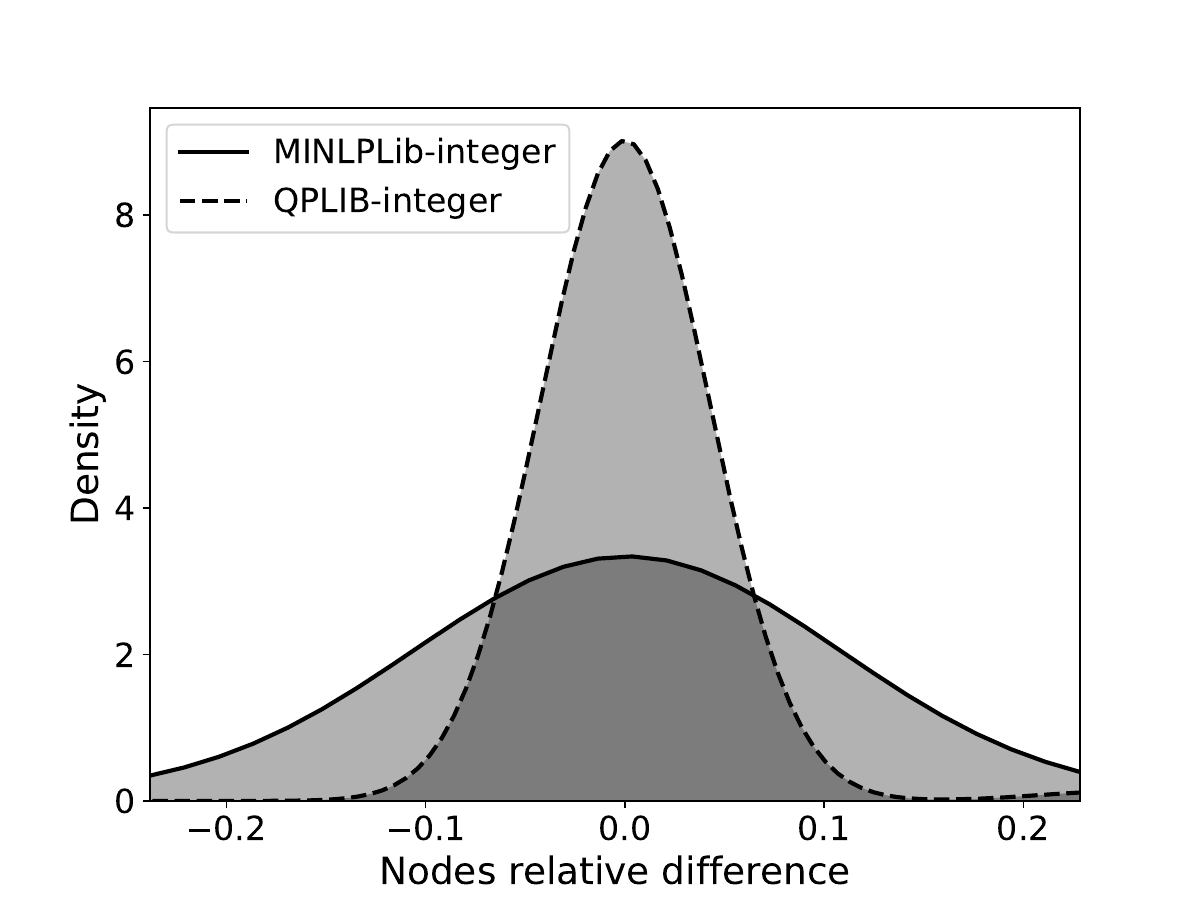}
        \caption{Nodes, \solver{Cplex}.}
    \end{subfigure}    
    \begin{subfigure}{0.4\textwidth}
        \includegraphics[trim=0 0 0 28, clip, width=\textwidth]{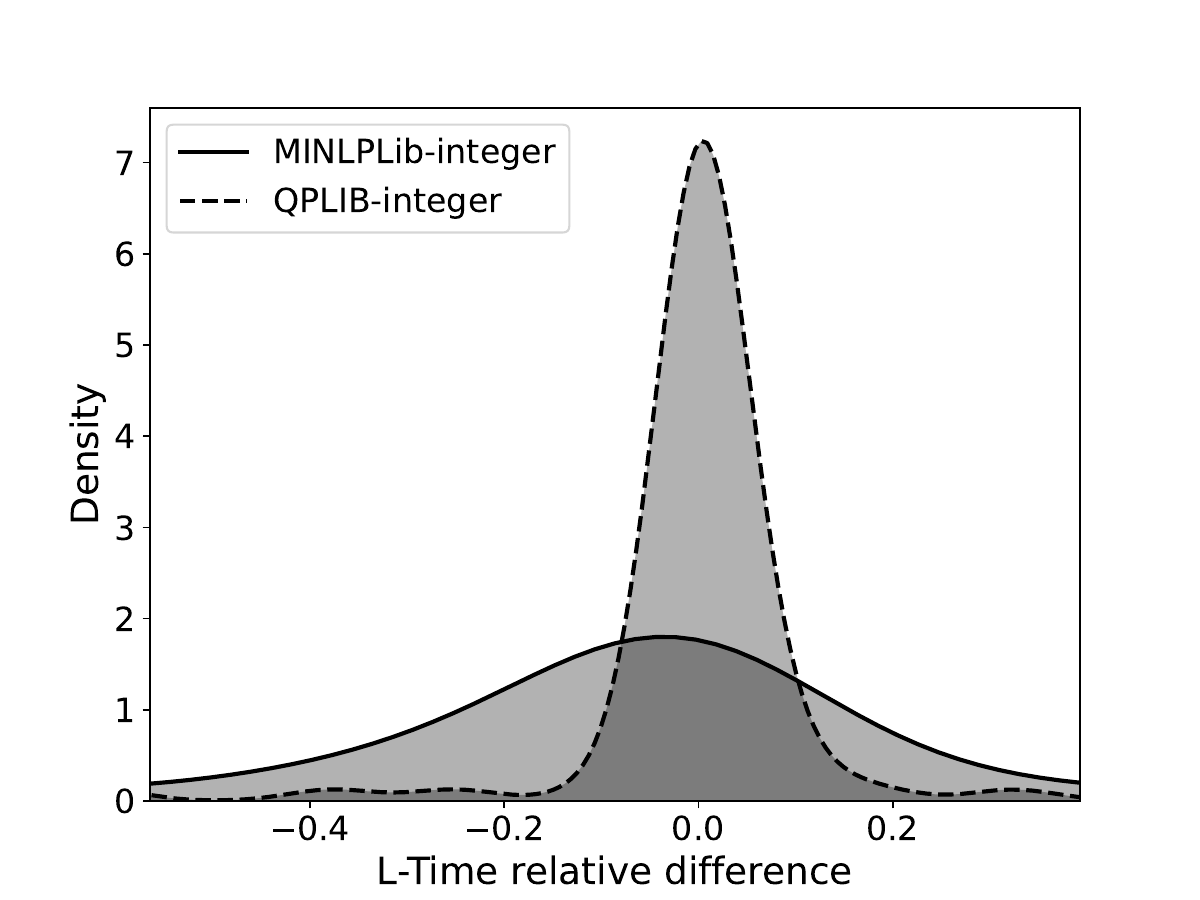}
        \caption{L-Time, \solver{Cplex}.}
    \end{subfigure}       
    \begin{subfigure}{0.4\textwidth}
        \includegraphics[trim=0 0 0 28, clip, width=\textwidth]{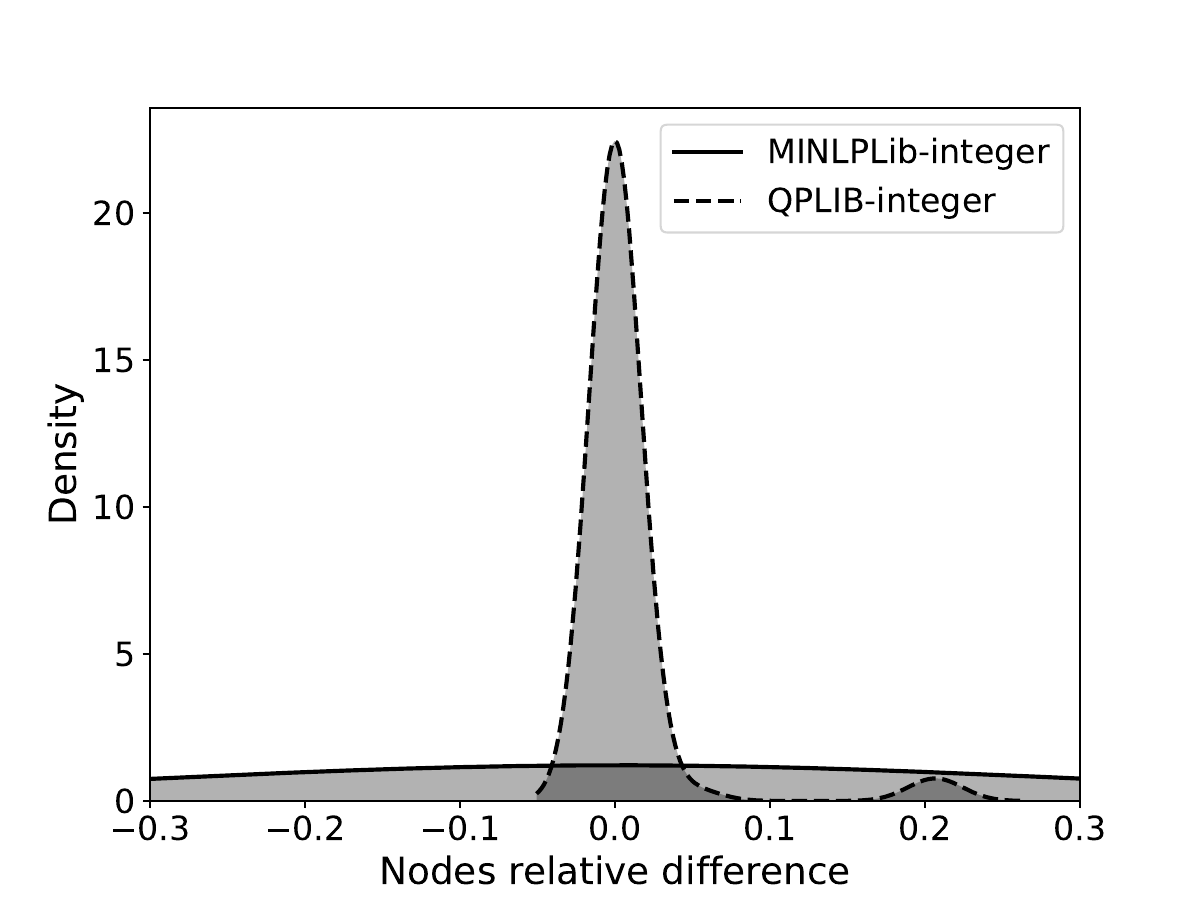}
        \caption{Nodes, \solver{Xpress}.}
    \end{subfigure}    
    \begin{subfigure}{0.4\textwidth}
        \includegraphics[trim=0 0 0 28, clip, width=\textwidth]{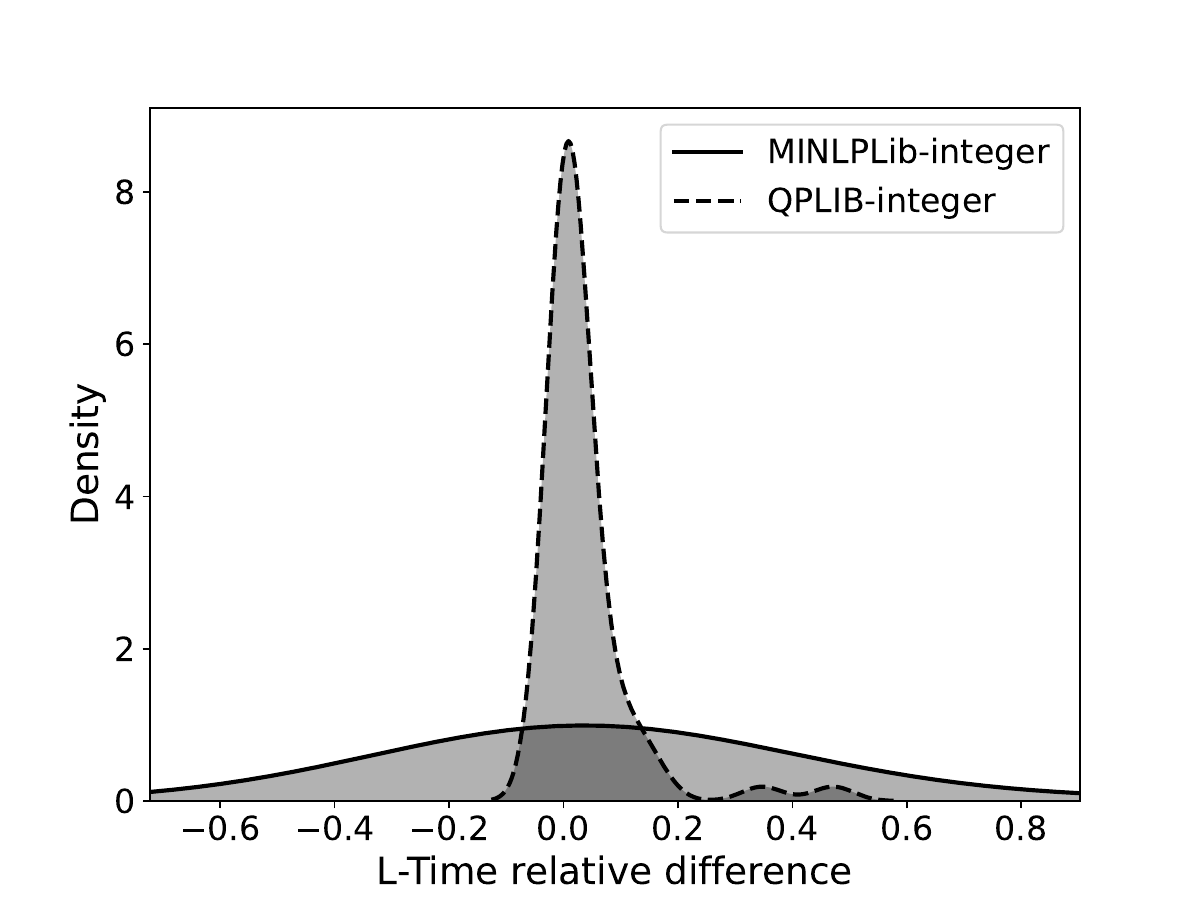}
        \caption{L-Time, \solver{Xpress}.}
    \end{subfigure}
    
    \caption{Relative differences in Nodes and L-Time between \RLTlooseB and \RLTtightB with bound tightening in mixed-integer instances (solvers other than \solver{Gurobi}).}
\end{figure}

\end{document}